\documentclass[11pt]{article}
\usepackage[shortalphabetic,initials]{amsrefs}

\usepackage{mathtools}
\usepackage{amsmath, amsthm, amssymb}
\usepackage{mathrsfs}
\usepackage[colorlinks,allcolors=blue]{hyperref}
\usepackage[all,cmtip]{xy}
\usepackage{tikz}
\usepackage{tikz-cd}
\usepackage{float}
\usepackage[shortlabels]{enumitem}
\usepackage{multicol}
\usetikzlibrary{arrows}
\usetikzlibrary{shapes.geometric}
\usetikzlibrary{decorations.markings}
\usetikzlibrary{decorations.pathreplacing}
\usetikzlibrary{hobby}

\usepackage{pgfmath}

\usepackage{ulem}
\usepackage{comment}
\usepackage{diagbox}

\let\amsamp=&

\theoremstyle{plain}

\newtheorem{thm}{Theorem}

\newtheorem{prop}[thm]{Proposition}

\theoremstyle{definition}

\newtheorem{definition}[thm]{Definition}
\newtheorem{example}[thm]{Example}

\newtheorem{rmk}[thm]{Remark}

\let\oldmarginpar\marginpar
\renewcommand\marginpar[1]{\-\oldmarginpar[\raggedleft\footnotesize #1]%
	        {\raggedright\footnotesize #1}}

\newskip\stdskip                      
\newcommand{\Alg}{\mathbf{Alg}}

\newcommand{\eightcurve}{%
\begin{tikzpicture}[baseline=-0.6ex,scale=0.25]
  \draw[thick] (0,0) .. controls (1,1) and (2,1) .. (2,0)
                     .. controls (2,-1) and (1,-1) .. (0,0)
                     .. controls (-1,1) and (-2,1) .. (-2,0)
                     .. controls (-2,-1) and (-1,-1) .. (0,0);
\end{tikzpicture}%
}
\newcommand{\tripleeye}{
\begin{tikzpicture}[baseline=-0.6ex,scale=0.25]
  \draw[thick] (0,0) .. controls (0,1) and (1,1) .. (2,0)
                     .. controls (2.5,-1) and (3.5,-1) .. (4,0)
                     .. controls (5,1) and (6,1) .. (6,0)
                     .. controls (6,-1) and (5,-1) .. (4,0)
                     .. controls (3.5,1) and (2.5,1) .. (2,0)
                     .. controls (1,-1) and (0,-1) .. (0,0);
\end{tikzpicture}
}

\newcommand{\eightinner}{%
\begin{tikzpicture}[baseline=-0.6ex, scale=0.35]
  \draw[thick] (0.5,0) .. controls (1,0.8) and (2,0.8) .. (2,0)
                       .. controls (2,-0.8) and (1,-0.8) .. (0.5,0)
                       .. controls (0,1) and (-2,1) .. (-1.5,0)       
                       .. controls (-1,-0.5) and (-0.5,-0.5) .. (-0.5,0)    
                       .. controls (-0.5,0.5) and (-1,0.5) .. (-1.5,0) 
                       .. controls (-2,-1) and (0,-1) .. (0.5,0); 
  \end{tikzpicture}%
}

\newcommand{\eears}{
\begin{tikzpicture}[baseline=1.5ex,scale=0.15]
  \draw[thick] (0.6,0.6) .. controls (-0.6,0.4) and (-1,2) .. (2,2)
                     .. controls (3,3) .. (4,4)
                     .. controls (3,5) and (3,6) .. (4,6)
                     .. controls  (5,6) and (5,5) .. (4,4)
                     .. controls (5,3) .. (6,2)
                     .. controls (6.7, 1.3) .. (7.4, 0.6)
                     .. controls (8.6, 0.4) and (9,2)  .. (6,2)
                     .. controls (4,2) .. (2,2)
                     .. controls (1.3,1.3) .. (0.6,0.6);                     
\end{tikzpicture}
}

\newcommand{\trefoil}{
\begin{tikzpicture} [baseline=-0.6ex,scale=0.15]
    \draw[thick, samples=200, smooth, domain=0:360] 
        plot ({sin(\x) + 2*sin(2*\x)}, {cos(\x) - 2*cos(2*\x)}) -- cycle;
\end{tikzpicture}
}

\newcommand{\quadeye}{
\begin{tikzpicture}[baseline=-0.6ex,scale=0.25]
  \draw[thick] (0,0) .. controls (0,1) and (1,1) .. (2,0)
                     .. controls (2.5,-1) and (3.5,-1) .. (4,0)
                     .. controls (4.5,-1) and (5.5,-1) .. (6,0)
                     .. controls (7,1) and (8,1) .. (8,0)
                     .. controls (8,-1) and (7,-1) .. (6,0)
                     .. controls (5.5,1) and (4.5,1) .. (4,0)
                     .. controls (3.5,1) and (2.5,1) .. (2,0)
                     .. controls (1,-1) and (0,-1) .. (0,0);
\end{tikzpicture}
}

\newcommand{\toric}{
\begin{tikzpicture}[baseline=-0.6ex, scale=0.35]
  \draw[thick] (0.5,0) .. controls (1,1) and (3,1) .. (2.5,0)
                       .. controls (2,-0.5) and (1.5,-0.5) .. (1.5,0)
                       .. controls (1.5,0.5) and (2,0.5) .. (2.5,0) 
                       .. controls (3,-1) and (1,-1) .. (0.5,0)           
                       .. controls (0,1) and (-2,1) .. (-1.5,0)      
                       .. controls (-1,-0.5) and (-0.5,-0.5) .. (-0.5,0)   
                       .. controls (-0.5,0.5) and (-1,0.5) .. (-1.5,0) 
                       .. controls (-2,-1) and (0,-1) .. (0.5,0); 
\end{tikzpicture}
}

\newcommand{\stabsym} {
\begin{tikzpicture}[baseline=-0.6ex, scale=0.35]
    \draw[thick] (0.5,0) .. controls (1,1) and (2,1) .. (2.5, 0)
                         .. controls (3,0.8) and (4,0.8) .. (4,0)
                         .. controls (4,-0.8) and (3,-0.8) .. (2.5,0)
                         .. controls (2, -1) and (1, -1) .. (0.5,0)
                         .. controls (0,1) and (-2,1) .. (-1.5,0)       
                         .. controls (-1,-0.5) and (-0.5,-0.5) .. (-0.5,0)  
                         .. controls (-0.5,0.5) and (-1,0.5) .. (-1.5,0)
                         .. controls (-2,-1) and (0,-1) .. (0.5,0);
\end{tikzpicture}
}

\newcommand{\doubleear}{
\begin{tikzpicture}[baseline=0.3ex, scale=0.35]
  \draw[thick] (0.5,0.5) .. controls (0.8,1.3) and (2,1.3) .. (2,0.5)
                       .. controls (2,-0.3) and (0.8,-0.3) .. (0.5,0.5)
                       .. controls (0,2) and (-2,2) .. (-1.5,1)       
                       .. controls (-1,0.5) and (-0.5,0.5) .. (-0.5,1)
                       .. controls (-0.5,1.5) and (-1,1.5) .. (-1.5,1)
                       .. controls (-1.7,0.5) .. (-1.5, 0)
                       .. controls (-1,-0.5) and (-0.5,-0.5) .. (-0.5,0)
                       .. controls (-0.5,0.5) and (-1,0.5) .. (-1.5,0)
                       .. controls (-2,-1) and (0,-1) .. (0.5,0.5); 
\end{tikzpicture}
}

\title{Curves on surfaces and moduli of associative algebras}

\author{Yank\i\ Lekili\thanks{The author thanks Jonny Evans, Martin Kalck, Daniil Mamaev and Richard Thomas for their interest, and acknowledges support by the EPSRC grant EP/W015889/1.}}

\date{}

\begin{document}
\maketitle
\begin{abstract}
  Given an immersion of a circle in a punctured surface $\Sigma$, we give an explicit (and finite) computation of the $A_\infty$-algebra associated with this curve when viewed as an object in a (relative) Fukaya category of $\Sigma$ in terms of the signed Gauss word recording the double points in a traversal of the curve and the visible polygons that it bounds in $\Sigma$. We illustrate our computational technique by fully determining the $A_\infty$-products for immersions with up to three self-intersections. In particular, it is proved that, over an algebraically closed field, all associative algebras of dimension $\leq 4$, with one exception, can be realized as the (degree 0) endomorphism algebra of some Lagrangian immersion of a circle equipped with a bounding cochain computed in some relative Fukaya category $\mathcal{F}(\Sigma,D)$. We also note that any finite-dimensional algebra with radical square zero arises as the (degree 0) endomorphism algebra of an object in the Fukaya category $\mathcal{F}(\Sigma)$ of some punctured surface $\Sigma$.

\end{abstract}

Which finite-dimensional (ungraded) algebras over an algebraically closed field $k$ can be realized as the endomorphism algebra $\mathrm{End}(L)$, in degree 0, of a {\it compact} Lagrangian $L$ (possibly equipped with a bounding cochain $\mathfrak{b}$) in a relative Fukaya category $\mathcal{F}(\Sigma,D)$ of a {\it surface} $\Sigma$ where $D$ is a divisor such that $L \subset \Sigma \setminus D$?

In this paper, we show that all algebras of dimension $\leq 4$ can be realized with a unique exception. We show this by studying immersions of a circle with fewer than 4 self-intersections. Any immersed curve $\gamma: S^1 \to L \subset \Sigma$ determines a smallest surface $\Sigma_\gamma$, the tubular neighborhood of $L$, and the endomorphism algebra of $L$ in $\mathcal{F}(\Sigma_{\gamma})$ is the most degenerate unital algebra 
\[ A_\bullet \simeq k[x_1,\ldots, x_{r-1}]/ (x_1,\ldots, x_{r-1})^2 \]
of some rank $r = 1 + \# \text{(nodes of $L$)}$. We can deform this algebra in two ways:

(i) Partially compactifying the surface $\Sigma_\gamma$ and considering $L$ as an object of the relative Fukaya category $\mathcal{F}(\Sigma,D)$ where $\Sigma_{\gamma} = \Sigma\setminus D$,

(ii) Turning on deformations coming from bounding cochains $\mathfrak{b} \in \mathrm{Ext}^1(L,L)$.

In a Fukaya category, the algebra structure on $\mathrm{End}(L)$ is given by counting rigid holomorphic triangles bounding $L$. Geometrically, (i) means that we count rigid holomorphic triangles passing through $D$ (not just in $\Sigma \setminus D$), and (ii) means that we also count rigid holomorphic $(k+3)$-gons where $k$ of the corners map to the bounding cochain $\mathfrak{b}$, which is just a collection of self-intersection points on $L$.

We first classify the deformations given by (i) for low rank via case-by-case analysis. All local algebras of rank $r \leq 4$ are given in the following list\footnote{Throughout, $k[\ldots]$ is the free commutative algebra, and $k\{ \ldots \}$ is the free associative algebra.}:
\begin{align*}
r=2 &:  k[x]/x^2 \\
r=3 &: k[x,y]/(x,y)^2, k[x]/x^3 \\
r=4 &: k[x,y,z]/(x,y,z)^2, k[x]/x^4, k[x,y]/(x^3,xy,y^2), k\{x,y\}/(x^2,y^2,xy-qyx), \\ &k\{x,y\} /(x^2+yx,y^2,xy+yx)
\end{align*}  
The deformations from (i) allow us to realize all of these except the last one. Note that the algebras $A_q =k\{x,y\}/(x^2,y^2,xy-qyx)$ depend on a parameter $q \in k$ (moduli), and provide an infinite family of pairwise non-isomorphic $4$-dimensional algebras with the exception that $A_q \simeq A_{1/q}$ for $q \neq 0$. It is necessary to work with {\it relative} Fukaya categories in order to be able to realize algebras that have moduli.

In this paper, we work with $\mathbb{Z}$-graded relative Fukaya categories of surfaces. Hence, our surfaces are implicitly equipped with a line field $\eta$ (see Section 1 \cite{LPol}) and $L$ has winding number zero with respect to $\eta$. In considering relative Fukaya categories, we are only interested in compactifications where $\eta$ extends. We allow $\mathbb{Z}/2\mathbb{Z}$-orbifold compactifications since these are the only orbifold compactifications over which $\eta$ may extend.

Some non-local algebras (containing non-trivial idempotents) also arise from (i), but to realize all non-local algebras (see Section 1 for a list), we next turn on deformations of type (ii). We first give an explicit computation of deformations of type (ii) only, staying within $\Sigma_{\gamma}$, but for arbitrary rank $r$ and arbitrary bounding cochain $\mathfrak{b}$. To do this, we provide a complete calculation of the $A_{\infty}$-algebra \[ hom(L,L)\] computed in $\mathcal{F}(\Sigma_{\gamma})$. The result is given as a strictly unital minimal $A_\infty$-algebra with only finitely many products and $\mathfrak{m}_i \neq 0$ only for $i=2,3$ computed from a signed Gauss word describing the immersion (see Definition \ref{newalg}). These $A_\infty$ algebras are quite easy to manipulate combinatorially. We deduce easily that the endomorphism algebras of $(L,\mathfrak{b}) \in \mathcal{F}(\Sigma_\gamma)$ correspond exactly to finite-dimensional radical square zero algebras (see Proposition \ref{proprad2}).  

Next, we combine both (i) and (ii) to realize many finite-dimensional algebras. Generally, arbitrary combinations of such deformations will cause the rank of $\mathrm{End}(L)$ to drop. Requiring the rank to stay constant is the flatness condition. We study possible deformations of $A_\bullet$ coming from combinations of (i) and (ii), and impose the flatness condition to obtain various finite-dimensional algebras. On the other hand, as we explain in Section 1, the varieties of associative algebras, $\Alg_r$, of low rank $r$ have been studied classically. It is rather amusing to contemplate parts of mathematics that cast their shadows on the irreducible components of these varieties.

\begin{thm} Let $k$ be an algebraically closed field. For every $r$-dimensional algebra $A$ over $k$ with $r \leq 4$, except $k\{x,y\} /(x^2+yx,y^2,xy+yx)$, there exists a Lagrangian immersion $\gamma : S^1 \to L \subset \Sigma$ equipped with a bounding cochain $\mathfrak{b}$, giving an object of a relative Fukaya category $\mathcal{F}(\Sigma,D)$ for some surface $\Sigma$ and a divisor $D$ such that $\mathrm{End}(L,\mathfrak{b})$ is a flat family of $r$-dimensional algebras realizing a flat deformation of $A_\bullet$ to $A$. In particular, a general member of this family gives an object in a Fukaya category of a surface whose endomorphism algebra is isomorphic to $A$.
\end{thm}

In fact, our computations are valid over an arbitrary coefficient ring $k$. For instance, any commutative ring that is free of rank 2 or 3 as a $\mathbb{Z}$-module and all quaternion algebras over $\mathbb{Q}$ are realized as endomorphism algebras. However, we have not systematically investigated the realization problem over non-closed fields. 

The algebra $J=k\{x,y\} / (x^2+yx,y^2,xy+yx)$ is the quadratic dual of the Jordan quantum plane $k\{u,v\}/ (uv-vu+u^2)$ with $|u|=|v|=1$. The latter can be realized as a twisted homogeneous co-ordinate ring on $\mathbb{P}^1$ (\cite{VdbStafford}). We also remark that $J$ is the limit of the family $J_q = k\{x,y\} / (x^2+yx, y^2, xy-qyx)$ as $q \to -1$. The change of basis $(x,y) \mapsto ((q+1)x+y, y)$ shows that $J_q$ is isomorphic to $A_q= k\{x,y\}/(x^2,y^2,xy-qyx)$ for $q \neq -1$ which, as we shall see, can be realized as the endomorphism algebra of a Lagrangian immersion in a relative Fukaya category of a surface. Our limited efforts did not lead to a construction of the limiting algebra $J$. So, we leave this as a question for the enthusiasts: Can the algebra $k\{x,y\}/(x^2+yx,y^2,xy+yx)$ be realized as an endomorphism algebra of an object in a Fukaya category of a {\it surface}? 

We end the introduction by displaying the champions of the realizability problem for rank 4. These two immersed curves equipped with various bounding cochains and put into various punctured surfaces cover almost all of the rank 4 algebras.

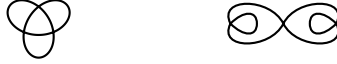
\begin{center}
  \texorpdfstring{\protect\trefoil}{} \qquad\qquad
  \texorpdfstring{\protect\toric}{}
\end{center}

\section{Moduli of associative algebras}

Let $k$ be a field. One would like to study the set of all non-isomorphic associative algebras over $k$. However, this set does not have a reasonable algebraic structure. To overcome this difficulty, we instead consider algebras with a fixed basis. Thus, let $V$ be an $r$-dimensional vector space over $k$ with a fixed basis $x_0,x_1,\ldots, x_{r-1}$. To make $V$ into an associative $k$-algebra with unit $x_0$, we need to give structure constants $c_{ij}^h \in k$ satisfying
\[ x_i x_j = \sum_{h=0}^{r-1} c_{ij}^h x_h \]
such that
\[ \sum_{h=0}^{r-1} (c_{ij}^h c_{hl}^m - c_{jl}^h c_{ih}^m) =0\]
for all $i,j, l,m$, and
\[ c_{0i}^j = c_{i0}^j = \begin{cases} 0, \quad i\neq j \\
   1, \quad i =j \end{cases} \]
We obtain a closed subvariety of $k^{r(r-1)^2}$, called the variety of unital associative algebras $\Alg_r$ which comes with a natural (left) action of the parabolic subgroup $G_r \subset GL_r(k)$ of invertible matrices whose first column is $(1,0,\ldots, 0)^t$, acting by a change of basis that preserves the unit $x_0$. The orbits of this action are in one-to-one correspondence with isomorphism classes of $r$-dimensional associative algebras with unit. We note that the dimension of $G_r$ is $r^2-r$, and the stabilizer of the orbit of an algebra $A$ is $\mathrm{Aut}_k(A)$, the group of $k$-algebra automorphisms of $A$.

The variety $\Alg_r$ is connected. In fact, any associative algebra structure on an $r$-dimensional vector space with unit may be degenerated to
\[ A_\bullet \cong k[x_1,\ldots, x_{r-1}]/ (x_1,\ldots, x_{r-1})^2 \]
To see this, given an arbitrary associative algebra with basis $1,x_1,\ldots, x_{r-1}$ and $t \in k^\times$, we can define a new basis $1, tx_1,\ldots, tx_{r-1}$. Then, the structure coefficients change as $c_{ij}^0 \to t^2 c_{ij}^0$, and $c_{ij}^h  \to t c_{ij}^h$ for $i,j,h \neq 0$.

If the $G_r$-orbit of an algebra $A_0$ is contained in the Zariski closure of the $G_r$-orbit of an algebra $A_1$, we say that $A_1$ is a deformation of $A_0$. In this way, one has a partial order on the set of isomorphism classes of $k$-algebras of dimension $r$.

\begin{example} \label{ex1} Every $2$-dimensional unital $k$-algebra $A$ can be written as $A = k\cdot 1 \oplus k\cdot x$ with multiplication determined by
\[
x^2 = s  + tx, \qquad s,t \in k .
\]
so the space of algebra structures is the affine plane $\Alg_2 \cong \mathbb A_k^2$ with coordinates $(s,t)$. Change of basis preserving the unit is given by $x \mapsto x' = \alpha + \beta x$ for $\beta \neq 0$, corresponding to the subgroup
\[
G =
\left\{
\begin{pmatrix}
1 & \alpha \\
0 & \beta
\end{pmatrix}
: \beta \in k^\times
\right\}
\subset GL_2(k).
\]
Under this transformation the parameters change as
\[
\begin{aligned}
t' &= 2\alpha + \beta t, \\
s' &= \beta^2 s - \alpha \beta t - \alpha^2 .
\end{aligned}
\]
The discriminant $\Delta = t^2 + 4s$ satisfies $\Delta' = \beta^2 \Delta$. Up to isomorphism (assuming $\mathrm{char}(k) \neq 2$) there are three types: (1) $\Delta = 0$:  $A \cong k[\varepsilon]/(\varepsilon^2)$ (the dual numbers), (2) $\Delta \neq 0$ and $\Delta$ is a square in $k$: $A \cong k \times k$ (the split semisimple algebra), (3) $\Delta$ non-square in $k$: $A \cong k(\sqrt{\Delta})$ (a quadratic field extension of $k$).
\end{example} 
\subsection{Classification in low rank}

For $k$ algebraically closed, Gabriel \cite{gabriel} describes all the irreducible components and the orbit decomposition of $\Alg_3$ and $\Alg_4$. We note that Gabriel uses a slightly different definition where one simply requires that an identity element exists, rather than fixing one. However, it is easy to see that the degeneration picture associated with orbit closures remains the same.

As we have seen in Example \ref{ex1}, for $r=2$, there is only one irreducible component and the algebra $A_{\bullet}  \cong  k[x]/x^2$ corresponding to the closed orbit deforms to the algebra $k \times k$ in the open orbit.

For $r=3$, there are two irreducible components (of dimensions 6 and 4): the closure of the orbit of $k \times k \times k$ and the closure of the orbit of the path algebra of the $A_2$-quiver, $\begin{tikzcd}[ampersand replacement=\&, column sep=small]
{\scriptstyle\bullet} \arrow[r] \& {\scriptstyle\bullet}
\end{tikzcd}$. Both components are smooth, in fact, isomorphic to $\mathbb{A}^6$ and $\mathbb{A}^4$, respectively (see \cite{LeBruynReichstein}). The partial order induced by orbit closures is described by the following diagram (reproduced from \cite{gabriel}, see also \cite{LeBruynReichstein}). We also included the dimension of the orbit of each isomorphism type.
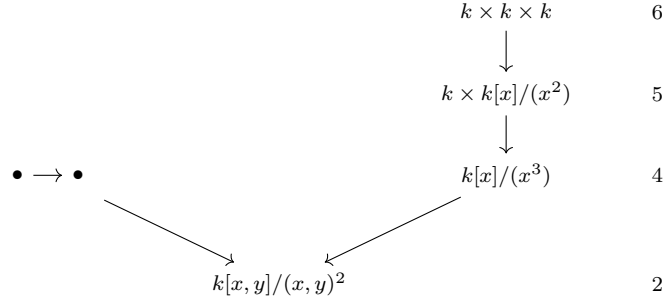
\begin{center}
\begin{tikzpicture}[node distance=2cm, every node/.style={align=center},  yscale=0.9]
\node (A) at (0,0.6) {
\tikz{
\node (x) at (0.7,0) {$\scriptstyle\bullet$};
\node (y) at (1.5,0) {$\scriptstyle\bullet$};
\draw[->] (x) -- (y);
}
};
\node (B) at (6,3) {\scriptsize $k \times k \times k$};
\node at (8,3) {\scriptsize $6$};
\node (C) at (6,1.8) {\scriptsize $k \times k[x]/(x^2)$};
\node at (8,1.8) {\scriptsize $5$};
\node (D) at (6,0.6) {\scriptsize $k[x]/(x^3)$};
\node at (8,0.6) {\scriptsize $4$};
\node (E) at (3,-1) {\scriptsize $k[x,y]/(x,y)^2$};
\node at (8,-1) {\scriptsize $2$};
\draw[->] (B) -- (C);
\draw[->] (C) -- (D);
\draw[->] (A) -- (E);
\draw[->] (D) -- (E);
\end{tikzpicture}
\end{center}

We next turn our attention to 4-dimensional algebras. Gabriel \cite{gabriel} announces (a detailed modern proof appears in \cite{FialowskiPenkava}) that there are 19 classes of unital associative algebras of rank 4 (18 isomorphism types and 1 continuous family), and there are 5 irreducible components of $\Alg_4$. Four of these are the orbit closures of the following rigid algebras.
\begin{enumerate}[label=(\roman*), nosep]
\item $k \times k \times k \times k$, orbit dimension 12.   
\item $M_2(k)$, 2-by-2 matrices,  orbit dimension 9 (isomorphic to $\mathbb{A}^{9}$, see \cite{seshadri} Ch. 5). 
\item $k \times (\begin{tikzcd}[ampersand replacement=\&, column sep=small]
{\scriptstyle\bullet} \arrow[r] \& {\scriptstyle\bullet}
\end{tikzcd})$, orbit dimension 10. 
\item $\begin{tikzcd}[ampersand replacement=\&, column sep=small]
{\scriptstyle\bullet} \arrow[r, shift left] \arrow[r, shift right] \& {\scriptstyle\bullet}
\end{tikzcd}$ Kronecker quiver, orbit dimension 6 (smooth component).
\end{enumerate}
The remaining irreducible component has a Zariski open subset that is a union of the orbits (each of dimension 8) of the following continuous family of algebras:
\noindent \begin{align*} A_q &= k \{ x,y \} / (x^2, y^2, xy-qyx) \text{\ for\ } q \neq -1 \quad (A_q \simeq A_{1/q} \text{\ for\ } q \neq 0),\\  J &= k \{x,y \} / (x^2+yx, y^2, xy+yx). \end{align*} 
Let $J_q = k\{x,y\} / (x^2+yx, y^2, xy-qyx)$. Then the change of basis $(x,y) \mapsto ((q+1)x+y, y)$ shows that $J_q$ is isomorphic to $A_q$ for $q \neq -1$. However, $A_q$ and $J_q$ limit to non-isomorphic algebras as $q \to -1$.

The Hasse diagram of deformations of $4$-dimensional algebras is as follows (reproduced from \cite{gabriel} \footnote{In op. cit. the arrow from $\scriptsize k \times (\begin{tikzcd}[ampersand replacement=\&, column sep=small]
{\scriptstyle\bullet} \arrow[r] \& {\scriptstyle\bullet}
\end{tikzcd})$ to $\scriptsize k \times k[x,y]/(x,y)^2$ seems to be omitted.}). We highlighted various irreducible components in different colors (overlapping parts are arbitrary). We also indicated the dimension of the orbit of each isomorphism type.

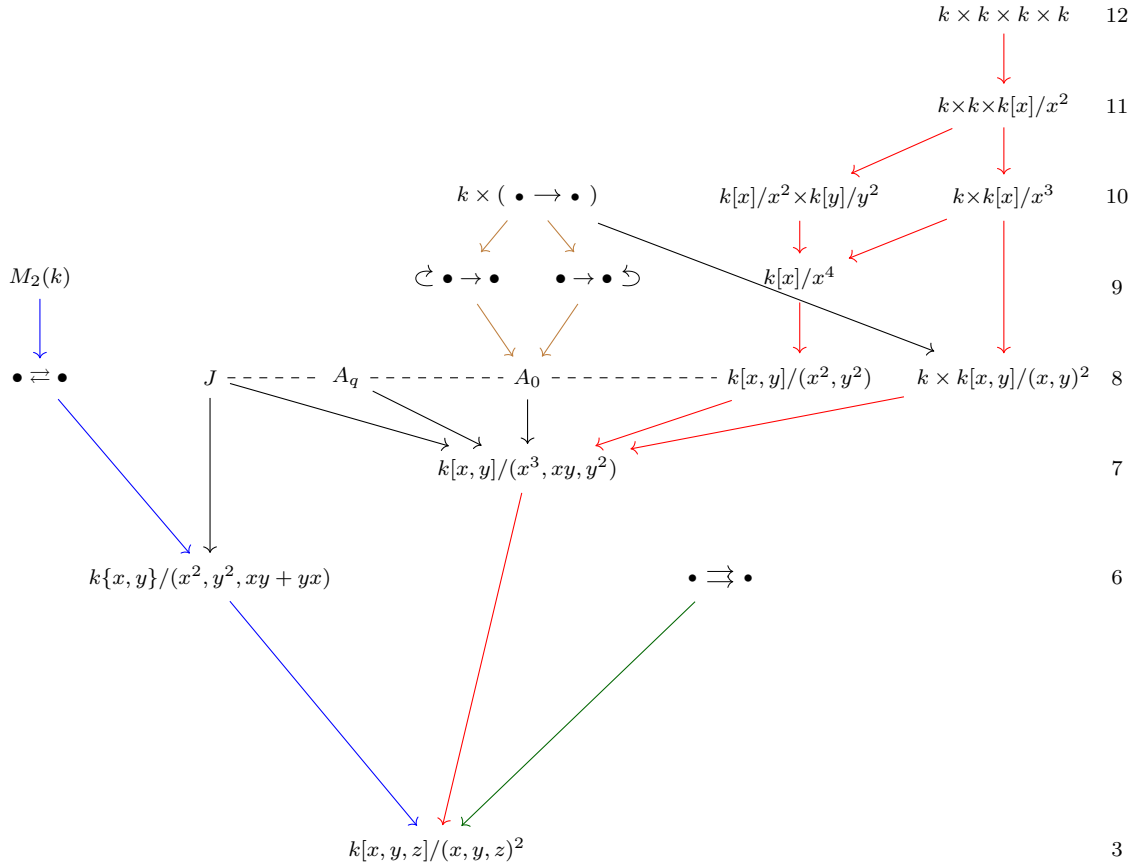
\begin{center}
\begin{tikzpicture}[node distance = 4cm, yscale=1.2, xscale=1.5]

    \tikzset{dot/.style={circle, fill=black, inner sep=0pt, minimum size=4pt, align=center}}

    \foreach \y/\label in {0/3, 3/6, 4.2/7, 5.2/8, 6.2/9, 7.2/10, 8.2/11, 9.2/12} {
        \node at (6, \y) {\scriptsize \label};
    }

    \node (A0) at (0,0) {\scriptsize $k[x,y,z]/ (x,y,z)^2$};
    \node(Am) at (-2, 3) {\scriptsize $k\{x,y\}/(x^2,y^2,xy+yx)$};
    \node(Ap) at (0.8, 4.2) {\scriptsize $k[x,y]/(x^3,xy,y^2)$};
    \node (A1) at (2.5,3) {\scriptsize $\begin{tikzcd}[ampersand replacement=\&, column sep=small]
{\scriptstyle\bullet} \arrow[r, shift left] \arrow[r, shift right] \& {\scriptstyle\bullet}
\end{tikzcd}$};
    \node  (B4) at (-3.5, 5.2) {\scriptsize $\bullet \rightleftarrows \bullet$ };
    \node (J1) at (-2, 5.2) {\scriptsize $J$}; 
    \node (J2) at (0.8, 5.2) {\scriptsize $A_0$};
    \node (J3) at (5, 5.2) {\scriptsize $k \times k[x,y]/(x,y)^2$};
    \node (J4) at (3.2, 5.2) {\scriptsize $k[x,y]/(x^2,y^2)$};
    \node (J5) at (-0.8,5.2) {\scriptsize $A_q$};
    \node (E4) at (-3.5, 6.3) {\scriptsize $M_2(k)$};    
    \node (L9_top) at (0.8, 7.2) {\scriptsize $k \times (\begin{tikzcd}[ampersand replacement=\&, column sep=small]
{\scriptstyle\bullet} \arrow[r] \& {\scriptstyle\bullet}
\end{tikzcd})$ };
    \node (L9_left) at (0.2, 6.3)  {\scriptsize $\begin{tikzcd}  \arrow[distance=1em, loop, out=190, in=170, looseness=2, swap] \bullet \to \bullet \end{tikzcd}$};
    \node (L9_right) at (1.4, 6.3) {\scriptsize $\begin{tikzcd} \bullet \to \bullet  \arrow[distance=1em, loop right, out=350, in=10, looseness=2, swap] \end{tikzcd}$ };
     \node(R9_mid) at (3.2, 6.3) {$\scriptstyle k[x]/x^4$};
    \node(R10_left) at (3.2, 7.2) {$\scriptstyle k[x]/x^2 \times k[y]/y^2$};
    \node(R10_right) at (5, 7.2) {$\scriptstyle k \times k[x]/x^3$};
    \node(R11_mid) at (5, 8.2) {$\scriptstyle k \times k \times k[x]/x^2$};
    \node (R12_top) at (5,9.2) {\scriptsize $k\times k\times k\times k$};

    \draw[->, blue, shorten >= 1pt, shorten <= 1pt ] (Am) -- (A0);
    \draw[->,red , shorten >= 1pt, shorten <= 1pt ] (Ap) -- (A0);
    \draw[->,green!40!black, shorten >= 1pt, shorten <= 1pt ] (A1) -- (A0);
    \draw[->,blue, shorten >= 1pt, shorten <= 1pt ] (B4) -- (Am) ;
    \draw[->, shorten >= 1pt, shorten <= 1pt ] (J1) -- (Am) ; 
    \draw[->, shorten >= 1pt, shorten <= 1pt ] (J1) -- (Ap); 
    \draw[->, shorten >= 1pt, shorten <= 1pt ] (J2) -- (Ap);
    \draw[->,red, shorten >= 1pt, shorten <= 1pt ] (J3) -- (Ap);
    \draw[->,red, shorten >= 1pt, shorten <= 1pt ] (J4) -- (Ap);
    \draw[->, shorten >= 1pt, shorten <= 1pt ] (J5) -- (Ap);
    \draw[->, blue] (E4) -- (B4);    
    \draw[->, brown, shorten >= 1pt, shorten <= 1pt ] (L9_left) -- (J2);
    \draw[->, brown, shorten >= 1pt, shorten <= 1pt ] (L9_top) -- (L9_left);
    \draw[->, brown, shorten >= 1pt, shorten <= 1pt ] (L9_top) -- (L9_right);
    \draw[->, brown, shorten >= 1pt, shorten <= 1pt] (L9_right) -- (J2);
    \draw[->, red, shorten >= 1pt, shorten <= 1pt ]  (R9_mid) -- (J4);
    \draw[->, red,shorten >= 1pt, shorten <= 1pt ]  (R10_right) -- (J3);
    \draw[->, red, shorten >= 1pt, shorten <= 1pt ]  (R10_left) -- (R9_mid);
    \draw[->, red, shorten >= 1pt, shorten <= 1pt ]  (R10_right) -- (R9_mid);
    \draw[->, red, shorten >= 1pt, shorten <= 1pt ]  (R11_mid) -- (R10_left);
    \draw[->, red] (R11_mid) -- (R10_right);
    \draw[->, red] (R12_top) --  (R11_mid);

    \draw[dashed] (J1) -- (J5) -- (J2) -- (J4);
    
   \draw[->, shorten >= 5pt, shorten <= 5pt ] (L9_top) -- (J3);
    
\end{tikzpicture}
\end{center}

Finally, we remark that for $r=5$ Mazzola showed that $\Alg_5$ has $10$ irreducible components, 9 of which are orbit closures of rigid algebras. There are 54 discrete isomorphism classes, and 5 one-parameter families of classes. We refer to \cite{mazzola} for a detailed description. 

For more on $\mathbf{Alg}_r$, one can begin with \cite{gabriel}, \cite{LeBruynReichstein}, \cite{shafarevich}, \cite{poonen}, \cite{FialowskiPenkava}.
\newpage
\section{Immersed Lagrangians}
An $A_{\infty}$-algebra over a commutative ring $k$ is a $\mathbb{Z}$-graded $k$-module with a collection of $k$-linear maps $\mathfrak{m}_i : \mathscr{A}^{\otimes i} \to \mathscr{A}[2-i]$ for $i \geq 1$, where $\mathscr{A}[2-i]$ means that $\mathfrak{m}_i$ lowers the degree by $i-2$, satisfying the $A_{\infty}$-relations:
\[ \sum_{j,k} (-1)^{|a_1|+\ldots+|a_j|-j} \mathfrak{m}_{i-k+1} (a_i, \ldots, a_{j+k+1}, \mathfrak{m}_k(a_{j+k},\ldots, a_{j+1}), a_j, \ldots, a_1) =0. \]
Given an immersed Lagrangian $L$ with $r-1$ transverse double points with brane data (orientation, spin structure, local system, ...) on a symplectic surface $\Sigma$, Fukaya's construction gives an $A_\infty$-algebra $(\mathscr{A}_L, \{\mathfrak{m}_{i}\}_{i \geq 1})$ (see \cite{Akaho}, \cite{AkahoJoyce}). If the surface is non-compact, the symplectic form on $\Sigma$ is exact, and we can work with the exact Fukaya category as defined in \cite{SeidelBook}. By construction, $\mathscr{A}_L$ is linear over a coefficient ring $k$ which can be taken to be arbitrary in the exact setting. The underlying complex is the Floer cochain complex given by
$$ hom(L,L) = \bigoplus_{i=0}^{r-1} k \cdot w_i \oplus  k \cdot \bar{w}_i $$ 
where for $i \neq 0$, the pair $\{ w_i ,\bar{w}_i \}$ are associated with a self-intersection point of $L$, and the generators $w_0= e$ in degree 0 and $\bar{w}_0 = \bar{e}$ in degree 1 are associated to the minimum and maximum of a Morse function chosen on the domain of the Lagrangian immersion. By choosing a line field $\eta$ on $\Sigma$ which has winding number zero over $L$, we can equip $hom(L,L)$ with a $\mathbb{Z}$-grading. By construction, the degrees of $w_i$ and $\bar{w}_i$ satisfy \[ |w_i| + | \bar{w}_i | = 1.\]

Let $\gamma: S^1 \to \Sigma$ be an immersion with double points $\{x_1,\ldots, x_{n}\}$. For each $x_i$, let its preimages be $t_1,t_2 \in S^1$. If $(\dot{\gamma}(t_1),\dot{\gamma}(t_2))$ is positively oriented, then we put $x_i^+$ at $t_1$ and $x_i^{-}$ at $t_2$. We pick a starting point and go around the $S^1$ recording the double points $x_i^{\pm}$ and obtain a (cyclic) sequence of signed letters called the {\it signed Gauss word}. Each crossing appears exactly twice: once positively and once negatively. Two words are equivalent if they differ by (1) permuting the labels of the crossings, (2) changing all the exponents simultaneously to the opposite sign, (3) cyclically permuting the sequence, (4) reversing the sequence.  Thus, a Gauss word can be thought of as an element of the symmetric group $\mathfrak{S}_{2n}$ and the equivalence classes of Gauss words correspond to double cosets
\[ (\mathfrak{S}_n \times \mathbb{Z}_2) \backslash \mathfrak{S}_{2n} / D_{4n} \]
where $\mathfrak{S}_{n} \times \mathbb{Z}_2$ acts on the crossings by relabelling or changing overall sign,  and $D_{4n}$, dihedral subgroup of order $4n$, acts on the domain $S^1$. Using this description, we can easily compute the number of Gauss words for small $n$. This is given in the following table:
\begin{table}[H]
\centering
\caption{Classification of signed Gauss words with $n \leq 3$ }
\label{tab:gauss_genus_classification}
\begin{tabular}{|c|c|l|c|c|}
\hline
  \textbf{n} & \textbf{Total} & \textbf{Representative} & \textbf{Genus} & $\partial $ \textbf{components}\\ \hline
1 & \textbf{1} & $1^+ 1^-$ &  0  & (1,1,2) \\ \hline
2 & \textbf{3} & $1^+ 2^- 2^+ 1^-$ & 0 & (1,1,2,4) \\
  &  & $1^+ 2^+ 2^- 1^-$ & 0  & (1,1,3,3) \\ \cline{3-5}
  &  & $1^+ 2^+ 1^- 2^-$ & 1 & (2,6)  \\ \hline
3 & \textbf{12} & $1^+ 2^- 3^+ 3^- 2^+ 1^-$  & 0 & (1,1,2,2,6)\\
  &  & $1^+2^-3^-3^+2^+1^-$ & 0 & (1,1,2,3,5)\\
  &  & $1^+ 2^+ 2^- 3^+ 3^- 1^-$ & 0 & (1,1,1,4,5)\\
  &  & $1^+ 2^+ 3^+ 3^- 2^- 1^-$ & 0 & (1,1,3,3,4)\\
  &  & $1^+ 2^- 3^+ 1^- 2^+ 3^-$ & 0 &  (2,2,2,3,3)\\  
  &  & $1^+ 1^- 2^+ 2^- 3^+ 3^-$ & 0& (1,1,1,3,6)\\ \cline{3-5}
  &  & $1^+ 2^+ 3^+ 1^- 2^- 3^-$  & 1 & (2,4,6)\\
  &  & $1^+ 2^+ 3^- 3^+ 1^- 2^-$ & 1 & (1,4,7)\\
  &  & $1^+ 2^+ 1^- 3^+ 3^- 2^-$ & 1 & (1,2,9)\\
  &  & $1^+ 2^+ 3^+ 3^- 1^- 2^-$ & 1 &  (1,3,8) \\
  &  & $1^+ 2^+ 3^- 1^- 3^+ 2^-$ & 1 & (2,3,7)\\ \cline{3-5} 
  &  & $1^+ 2^+ 3^+ 1^- 3^- 2^-$  & 2 & (12) \\ \hline
\end{tabular}
\end{table}
Each signed Gauss word determines a unique compact surface $\overline{\Sigma}$ and an immersion $\gamma : S^1 \to L \subset \overline{\Sigma}$ which it {\it fills}. This means that $\overline{\Sigma} \setminus L$ is a union of disks (see \cite{Francis},\cite{Carter}). Therefore, one can define the genus of a Gauss word. From now on, we will always consider immersions in such surfaces, and we will write $\Sigma_\gamma$ for the ``tubular'' neighborhood of $L$ in such a surface $\Sigma$. The number of equivalence classes of signed Gauss words for a small number of double points with respect to the genus is given in the table below (reproduced from \cite{CairnsElton}). We also recommend \cite{arnold} for beautiful tables of planar curves.
\begin{table}[H]
\centering
\begin{tabular}{|c|c|c|c|c|c|c|}
\hline
\diagbox{$n$}{$g$} & 0 & 1 & 2 & 3 & 4 & TOTAL \\ \hline
1 & 1 & 0 & 0 & 0 & 0 & 1 \\ \hline
2 & 2 & 1 & 0 & 0 & 0 & 3 \\ \hline
3 & 6 & 5 & 1 & 0 & 0 & 12 \\ \hline
4 & 19 & 45 & 22 & 0 & 0 & 86 \\ \hline
5 & 76 & 335 & 427 & 56 & 0 & 894 \\ \hline
\end{tabular}
\end{table}
\begin{definition}\label{newalg} Given a signed Gauss word $L$ on the letters $\{w_1,\ldots, w_{r-1}\}$, we define a $\mathbb{Z}$-graded (strictly unital, minimal) $A_\infty$-structure with $\mathfrak{m}_i\neq 0$ only for $i=2,3$ on the free abelian group
  \[ \mathscr{A}_L = \bigoplus_{i=0}^{r-1} \mathbb{Z} w_i \oplus \mathbb{Z} \bar{w}_i \]
where grading is given by any assignments $|w_i|,|\bar{w}_i| \in \mathbb{Z}$ such that $|w_i|+ |\bar{w}_i|=1$. 

For $i=0,\ldots,r-1$, we have
\begin{align*}
  \mathfrak{m}_2(w_i,w_0) &= w_i = (-1)^{|w_i|} \mathfrak{m}_2(w_0,w_i) \\
  \mathfrak{m}_2(\bar{w}_i,w_0) &= \bar{w}_i = (-1)^{|\bar{w}_i|} \mathfrak{m}_2(w_0,\bar{w}_i) \\
  \mathfrak{m}_2(\bar{w}_i,w_i) &= \bar{w}_0 = (-1)^{|\bar{w}_i|}\mathfrak{m}_2(w_i,\bar{w}_i) 
\end{align*}
For each $i=1,\ldots, r-1$, we have
\begin{align*}
  \mathfrak{m}_3(\bar{w}_i, w_i, \bar{w}_i)  &= (-1)^{|\bar{w}_i|} \bar{w}_i \\
  \mathfrak{m}_3(\bar{w}_i, w_i, \bar{w}_0)  &= - \bar{w}_0 \\
   \mathfrak{m}_3(w_i,\bar{w}_i, \bar{w}_0)  &= - (-1)^{|\bar{w}_i|} \bar{w}_0
\end{align*}
For each subinterval of the form $w_i^{+} \ldots  w_j^{+}$, we have
\begin{align*}
  \mathfrak{m}_3 (\bar{w}_j, w_i, \bar{w}_i) &= (-1)^{|\bar{w}_i|} \bar{w}_j \\ 
  \mathfrak{m}_3 (w_i , \bar{w}_i, w_j) &= (-1)^{|\bar{w}_i| + |w_j|} w_j  
\end{align*}
For each subinterval of the form $w_i^{+} \ldots  w_j^{-}$, we have
\begin{align*}
  \mathfrak{m}_3 (w_j, w_i , \bar{w}_i) &= (-1)^{|\bar{w}_i|} w_j \\ 
  \mathfrak{m}_3 (w_i, \bar{w}_i, \bar{w}_j) &= (-1)^{|\bar{w}_i|+|\bar{w}_j|}\bar{w}_j 
\end{align*}
For each subinterval of the form $w_i^{-} \ldots  w_j^{+}$, we have
\begin{align*}
  \mathfrak{m}_3 (\bar{w}_j, \bar{w}_i, w_i) &= \bar{w}_j  \\
  \mathfrak{m}_3 (\bar{w}_i, w_i, w_j) &=  (-1)^{|w_j|} w_j 
\end{align*}
For each subinterval of the form $w_i^{-} \ldots  w_j^{-}$, we have
\begin{align*}
  \mathfrak{m}_3 (w_j , \bar{w}_i, w_i) &= w_j \\ 
  \mathfrak{m}_3 (\bar{w}_i, w_i, \bar{w}_j) &= (-1)^{|\bar{w}_j|} \bar{w}_j 
\end{align*}
  \end{definition}
  \begin{rmk} In Definition \ref{newalg}, we defined a $\mathbb{Z}$-graded $A_\infty$ algebra for any choice of $|w_i| \in \mathbb{Z}$ for $i=1,\ldots, r-1$. Any choice is allowed and corresponds to all possible line fields $\eta$ on the tubular neighborhood $\Sigma$ of $L$ which make $L$ gradable. In this paper, we will mostly work with grading structures with $|w_i| \in \{0,1\}$.
\end{rmk}

  \begin{prop}\label{ainf} $\mathscr{A}_L$ is a strictly unital, minimal  $A_{\infty}$-algebra, that is, the following equations hold:
    \begin{align*}  &\mathfrak{m}_2(x,\mathfrak{m}_2(y,z)) - (-1)^{|z|} \mathfrak{m}_2(\mathfrak{m}_2(x,y),z) = 0 \\
                    &\mathfrak{m}_3(x,y,\mathfrak{m}_2(z,u)) -  (-1)^{|u|} \mathfrak{m}_3(x,\mathfrak{m}_2(y,z),u ) + (-1)^{|z|+|u|} \mathfrak{m}_3(\mathfrak{m}_2(x,y),z,u) \\ &-  (-1)^{|u|} \mathfrak{m}_2(\mathfrak{m}_3(x,y,z),u) + \mathfrak{m}_2(x,\mathfrak{m}_3(y,z,u)) = 0 \\
      & \mathfrak{m}_3(x,y,\mathfrak{m}_3(z,u,v)) - (-1)^{|v|}  \mathfrak{m}_3(x,\mathfrak{m}_3(y,z,u),v) +  (-1)^{|u|+|v|} \mathfrak{m}_3(\mathfrak{m}_3(x,y,z),u,v) = 0  
    \end{align*}
    \end{prop}
\begin{proof} By the strict unitality constraints $\mathfrak{m}_2(w_0,x) = (-1)^{|x|}x$ and $\mathfrak{m}_2(x,w_0) = x$, and because $\mathfrak{m}_n$ identically vanishes if $w_0$ is among its inputs for $n \geq 3$, any relation containing $w_0$ trivially reduces to a lower-arity relation. Thus, we assume all inputs belong to the set of non-unit generators $\{w_i, \bar{w}_i, \bar{w}_0\}$.

  Excluding $w_0$, the only non-zero $\mathfrak{m}_2$ products occur on matching pairs $\{w_i, \bar{w}_i\}$, which evaluate to $\pm \bar{w}_0$. Because $\bar{w}_0$ strictly annihilates all non-$w_0$ generators under $\mathfrak{m}_2$, any sequence of two nested $\mathfrak{m}_2$ operations evaluates identically to zero. Thus, the first equation holds.

  An inspection of Definition \ref{newalg} reveals that $\mathfrak{m}_3$ operations simply codify rules for how a matching pair $(a,b)$ ``absorbs'' an adjacent third element $c$. Notice that a specifically ordered matching pair $(a,b)$ corresponds to exactly one signed letter in the Gauss word (e.g., $(w_i, \bar{w}_i)$ identifies $i^+$, while $(\bar{w}_i, w_i)$ identifies $i^-$). Let $\mathfrak{m}_2(a,b) = \epsilon \bar{w}_0$. For any element $c$, the rules universally yield:
  \begin{itemize}[nosep, label={}]
      \item \textit{Right Absorption:} $\mathfrak{m}_3(a,b,c) = \epsilon (-1)^{|c|} c \quad$ (or zero, depending on the Gauss word)
      \item \textit{Left Absorption:} $\mathfrak{m}_3(c,a,b) = \epsilon c \quad$ (or zero, depending on the Gauss word)
  \end{itemize}
  Because $\bar{w}_0$ is only permitted as the third argument of $\mathfrak{m}_3$, the second equation from the statement reduces to:
  \[ \mathfrak{m}_3(x,y,\mathfrak{m}_2(z,u)) -  (-1)^{|u|} \mathfrak{m}_2(\mathfrak{m}_3(x,y,z),u) + \mathfrak{m}_2(x,\mathfrak{m}_3(y,z,u)) = 0 \]
  For the first term to be non-zero, both $(x,y)$ and $(z,u)$ must be matching pairs. Let $\mathfrak{m}_2(x,y) = \epsilon_1 \bar{w}_0$ and $\mathfrak{m}_2(z,u) = \epsilon_2 \bar{w}_0$. 
  Applying Right Absorption, the first term is $\mathfrak{m}_3(x,y, \epsilon_2\bar{w}_0) = -\epsilon_1\epsilon_2\bar{w}_0$.
  
  Each signed letter appears exactly once in the Gauss word. Therefore, the specific signed letter corresponding to $(x,y)$ must appear either strictly before or strictly after the signed letter for $(z,u)$. 
 
  \emph{$(x,y)$ precedes $(z,u)$.}

  The third term vanishes because $(z,u)$ is forbidden from Left-Absorbing $y$. As $(x,y)$ is permitted to Right-Absorb $z$, the second term evaluates to \[ -(-1)^{|u|} \mathfrak{m}_2( \epsilon_1 (-1)^{|z|} z, u ) = -(-1)^{|u|+|z|} \epsilon_1 \mathfrak{m}_2(z,u).\] Since $|z|+|u|=1$, this becomes $+\epsilon_1 \epsilon_2 \bar{w}_0$, identically canceling the first term.

    \emph{$(x,y)$ follows $(z,u)$.}

    The second term vanishes because $(x,y)$ cannot Right-Absorb $z$. As $(z,u)$ Left-Absorbs $y$, the third term evaluates to $\mathfrak{m}_2(x, \epsilon_2 y) = \epsilon_2 \mathfrak{m}_2(x,y) = +\epsilon_1\epsilon_2\bar{w}_0$, again identically canceling the first term.

If the first term vanishes, in a similar way, the remaining terms can only be non-zero when the inputs form interleaved matching pairs - such as $(y,z)$ and $(x,u)$ - whose evaluations identically cancel each other out via the Left and Right absorption rules. Thus, the second equation holds.
  
     A non-zero nested $\mathfrak{m}_3$ evaluation requires exactly two matching pairs $(x,y)$ and $(z,u)$ acting on a singleton $v$. Let their evaluation signs be $\epsilon_1$ and $\epsilon_2$.
  The first term applies Right Absorption twice: $\mathfrak{m}_3(x,y, \epsilon_2(-1)^{|v|}v) = \epsilon_1\epsilon_2(-1)^{2|v|}v = +\epsilon_1\epsilon_2 v$.

  \emph{$(x,y)$ precedes $(z,u)$.}

  The second term vanishes. The third term applies Right Absorption twice, yielding \[ +(-1)^{|u|+|v|} \mathfrak{m}_3( \epsilon_1(-1)^{|z|}z, u, v) = (-1)^{|u|+|v|+|z|} \epsilon_1 \epsilon_2 (-1)^{|v|} v.\] Because $|u|+|z|=1$ and $2|v|$ is an even integer, this evaluates to $-\epsilon_1\epsilon_2 v$, canceling the first term.

  \emph{$(x,y)$ follows $(z,u)$.}

  The third term vanishes. The second term evaluates to \[ -(-1)^{|v|} \mathfrak{m}_3(x, \epsilon_2 y, v) = -(-1)^{|v|} \epsilon_2 (\epsilon_1 (-1)^{|v|} v) = -\epsilon_1\epsilon_2 v,\] again canceling the first term.

  Symmetric cancellations govern the two alternate configurations in a similar fashion: the adjacency requirements of the matching pairs guarantee that at least one term trivially vanishes, while any surviving non-zero evaluations mutually cancel. The third equation follows.
\end{proof}
   Definition \ref{newalg} is actually the result of a computation of the $A_\infty$ algebra $hom(L,L)$ in $\mathcal{F}(\Sigma_\gamma)$ obtained by taking successive push-offs of $L$. To understand this, the reader should examine the following figure for the Gauss word $1^-2^+3^-1^+2^-3^+$ and look for triangles and rectangles contributing to the product.
  
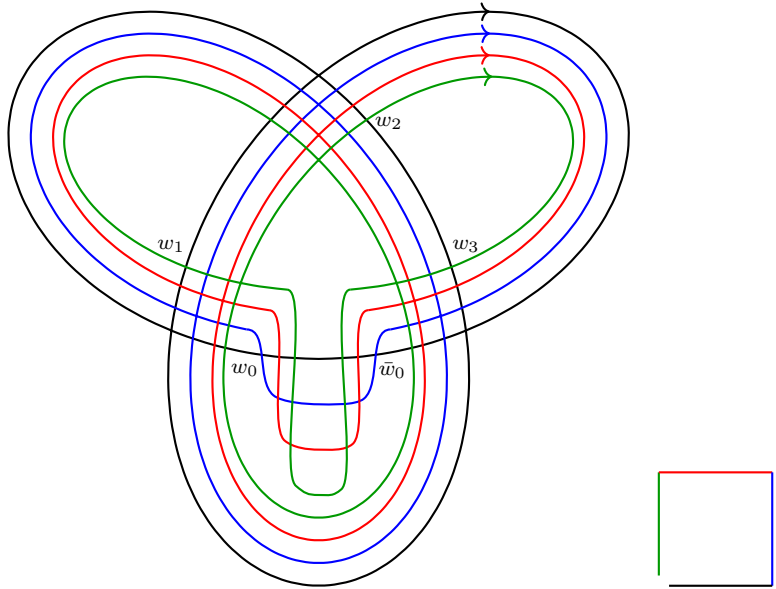
\begin{center}
    \begin{tikzpicture} [baseline=-0.6ex, scale=1.5]
     \draw[thick, samples=100, smooth, domain=0:360, postaction={decorate}, decoration={markings, mark=at position 0.22 with {\arrow{<}}}]
    plot ({sin(\x) + 2*sin(2*\x)}, {cos(\x) - 2*cos(2*\x)}) -- cycle;
    \draw[blue, thick, samples=100, smooth, domain=8:352,  postaction={decorate}, decoration={markings, mark=at position 0.2 with {\arrow{<}}} ] 
        plot ({sin(\x) + 1.8*sin(2*\x)}, {cos(\x) - 1.8*cos(2*\x)});
   \draw[red, thick, samples=100, smooth, domain=6:354,  postaction={decorate}, decoration={markings,  mark=between positions 0.2 and 0.21 step 1 with {\arrow{<}}}] 
        plot ({sin(\x) + 1.6*sin(2*\x)}, {cos(\x) - 1.6*cos(2*\x)});
        \draw[black!40!green, thick, samples=100, smooth, domain=4:356, postaction={decorate}, decoration={markings,   mark=between positions 0.2 and 0.21 step 1 with {\arrow{<}}}] plot ({sin(\x) + 1.5*sin(2*\x)}, {cos(\x) - 1.4*cos(2*\x)}) ;

\node[blue] (b1) at ({sin(8) + 1.8*sin(16)}, {cos(8) - 1.8*cos(16)}) {};
\node[blue] (b2) at ({sin(352) + 1.8*sin(704)}, {cos(352) - 1.8*cos(704)}) {};
\node[red] (r1) at ({sin(6) + 1.6*sin(12)}, {cos(6) - 1.6*cos(12)}) {};
\node[red] (r2) at ({sin(354) + 1.6*sin(708)}, {cos(354) - 1.6*cos(708)}) {};
\node[black!40!green] (g1) at ({sin(4) + 1.5*sin(8)}, {cos(4) - 1.4*cos(8)}) {};
\node[black!40!green] (g2) at ({sin(356) + 1.5*sin(712)}, {cos(356) - 1.4*cos(712)}) {};

\draw[thick, blue]
(b1.center) .. controls +(180:0.2) and +(40:0.2) ..  (0.4, -1.35)
            .. controls +(220:0.1) and +(0:0.1) ..  (0, -1.4)
            .. controls +(180:0.1) and +(320:0.1) .. (-0.4,-1.32)
             .. controls +(140:0.2) and +(0:0.2) .. (b2.center);

\draw[thick, red]
(r1.center) .. controls +(180:0.2) and +(40:0.2) ..  (0.3, -1.75)
            .. controls +(220:0.1) and +(0:0.1) ..  (0, -1.8)
            .. controls +(180:0.1) and +(320:0.1) .. (-0.3,-1.72)
             .. controls +(140:0.2) and +(0:0.2) .. (r2.center);

\draw[thick, black!40!green]
(g1.center) .. controls +(180:0.2) and +(40:0.2) ..  (0.2, -2.15)
            .. controls +(220:0.1) and +(0:0.1) ..  (0, -2.2)
            .. controls +(180:0.1) and +(320:0.1) .. (-0.2,-2.12)
             .. controls +(140:0.2) and +(0:0.2) .. (g2.center);

\node[] at (-1.3, 0)  {\scriptsize $w_1$};
\node[] at (0.62, 1.08)  {\scriptsize $w_2$};
\node[] at (1.3, 0)  {\scriptsize $w_3$};
\node[] at (0.65, -1.1)  {\scriptsize $\bar{w}_0$};
\node[] at (-0.65, -1.1)  {\scriptsize $w_0$};
             
    \node[] (v1) at (3,-3) {};
    \node[] (v2) at (3,-2) {} ;
    \node[] (v3) at (4,-2) {};
    \node[] (v4) at (4,-3) {};

    \draw[thick,black!40!green] (v1.north) -- (v2.center);
    \draw[thick,red] (v3.center) -- (v2.center);
    \draw[thick,blue] (v3.center) -- (v4.center);
    \draw[thick] (v4.center) -- (v1.east);
\end{tikzpicture}
  \end{center}
  We give a list of all the polygons that appear in such pictures.  
  \begin{center}
  \begin{tikzpicture}[scale=0.8]
    \tikzset{->-/.style={decoration={markings, mark=at position #1 with {\arrow{>}}},postaction={decorate}}}
    \node[] (o) at (0,0) {};
    \draw[thick, ->-=.5] (1,1) -- (o.east);
    \draw[thick, blue, ->-=.5] (1,1) -- (-1,1);
    \draw[thick, red, ->-=.5]  (o.west) -- (-1,1);
    \node[] at (0,-0.3) {\small $w_i$};
    \node[] at (1,1.3) {\small $w_0$};
    \node[] at (-1,1.3) {\small $w_i$};
      
    \node[] (o) at (3,0) {};
    \draw[thick, ->-=.5] (4,1) -- (o.east);
    \draw[thick, blue, ->-=.5] (4,1) -- (2,1);
    \draw[thick, red, ->-=.5] (2,1) -- (o.west);
    \node[] at (3,-0.3) {\small $\bar{w}_i$};
    \node[] at (4,1.3) {\small $w_0$};
    \node[] at (2,1.3) {\small $\bar{w}_i$};
     
      \node[] (o) at (6,0) {};
    \draw[thick, ->-=.5] (o.east) -- (7,1);
    \draw[thick, blue, ->-=.5] (5,1) -- (7,1);
    \draw[thick, red, ->-=.5] (5,1) -- (o.west);
 \node[] at (6,-0.3) {\small $w_i$};
    \node[] at (7,1.3) {\small $w_i$};
    \node[] at (5,1.3) {\small $w_0$};
   
      \node[] (o) at (9,0) {};
    \draw[thick, ->-=.5] (10,1) -- (o.east);
    \draw[thick, blue, ->-=.5] (8,1) -- (10,1);
    \draw[thick, red, ->-=.5] (8,1) -- (o.west);
 \node[] at (9,-0.3) {\small $\bar{w}_i$};
    \node[] at (10,1.3) {\small $\bar{w}_i$};
    \node[] at (8,1.3) {\small $w_0$};
   
      \node[] (o) at (12,0) {};
    \draw[thick, ->-=.5] (13,1) -- (o.east);
    \draw[thick, blue, ->-=.5] (13,1) -- (11,1);
    \draw[thick, red, ->-=.5] (11,1) -- (o.west);
 \node[] at (12,-0.3) {\small $w_0$};
    \node[] at (13,1.3) {\small $w_i$};
    \node[] at (11,1.3) {\small $\bar{w}_i$};
   
      \node[] (o) at (15,0) {};
    \draw[thick, ->-=.5] (16,1) -- (o.east);
    \draw[thick, blue, ->-=.5] (14,1) -- (16,1);
    \draw[thick, red, ->-=.5] (14,1) -- (o.west);
   \node[] at (15,-0.3) {\small $w_0$};
    \node[] at (16,1.3) {\small $\bar{w}_i$};
    \node[] at (14,1.3) {\small $w_i$};

  \end{tikzpicture}
    \end{center}

  \begin{center}
  \begin{tikzpicture}
    \tikzset{->-/.style={decoration={markings, mark=at position #1 with {\arrow{>}}},postaction={decorate}}}
  \node[] at (4.5,1.5) {$w_i^+ \ldots w_j^-$}; 
   \node[] (v1) at (1,0) {};
    \node[] (v2) at (1,1) {};
    \node[] (v3) at (2,1) {};
    \node[] (v4) at (2,0) {};
\node[xshift=8pt] at (2,0) {\small $\bar{w}_i$};
\node[xshift=8pt] at (2,1) {\small $w_i$};
\node[xshift=-8pt] at (1,1) {\small $w_j$};
\node[xshift=-8pt] at (1,0) {\small $w_j$};

\draw[thick,black!40!green, ->-=.5] (v1.north) -- (v2.center);
\draw[thick,red, ->-=.5] (v3.center) -- (v2.center);
\draw[thick,blue, ->-=.5] (v3.center) -- (v4.center);
\draw[thick, ->-=.5] (v4.center) -- (v1.east);

\node at (1.5, -1) {\small $\mathfrak{m}_3(w_j,w_i,\bar{w}_i) = (-1)^{|\bar{w}_i|} w_j$};

\node[] (v1) at (7,0) {};
\node[] (v2) at (7,1) {};
\node[] (v3) at (8,1) {};
\node[] (v4) at (8,0) {};
\node[xshift=8pt] at (8,0) {\small $\bar{w}_j$};
\node[xshift=8pt] at (8,1) {\small $\bar{w}_i$};
\node[xshift=-8pt] at (7,1) {\small $w_i$};
\node[xshift=-8pt] at (7,0) {\small $\bar{w}_j$};

\draw[thick,black!40!green, ->-=.6]  (v2.center) --  (v1.north) ;
    \draw[thick,red, ->-=.5] (v2.center) --  (v3.center)  ;
    \draw[thick,blue, ->-=.5] (v3.center) -- (v4.center);
    \draw[thick, ->-=.5] (v4.center) -- (v1.east);

\node at (7.5, -1) {\small $\mathfrak{m}_3(w_i,\bar{w}_i,\bar{w}_j) = (-1)^{|\bar{w}_i|+|\bar{w}_j|} \bar{w}_j$};
\end{tikzpicture}
\begin{tikzpicture}
    \tikzset{->-/.style={decoration={markings, mark=at position #1 with {\arrow{>}}},postaction={decorate}}}
  \node[] at (4.5,1.5) {$w_i^+ \ldots w_j^+$}; 
   \node[] (v1) at (1,0) {};
    \node[] (v2) at (1,1) {};
    \node[] (v3) at (2,1) {};
    \node[] (v4) at (2,0) {};
\node[xshift=8pt] at (2,0) {\small $\bar{w}_i$};
\node[xshift=8pt] at (2,1) {\small $w_i$};
\node[xshift=-8pt] at (1,1) {\small $\bar{w}_j$};
\node[xshift=-8pt] at (1,0) {\small $\bar{w}_j$};

\draw[thick,black!40!green, ->-=.6] (v2.center) --  (v1.north);
\draw[thick,red, ->-=.5] (v3.center) -- (v2.center);
\draw[thick,blue, ->-=.5] (v3.center) -- (v4.center);
\draw[thick, ->-=.5] (v1.east) -- (v4.center); 

\node at (1.5, -1) {\small $\mathfrak{m}_3(\bar{w}_j,w_i,\bar{w}_i) = (-1)^{|\bar{w}_i|} \bar{w}_j$};

\node[] (v1) at (7,0) {};
\node[] (v2) at (7,1) {};
\node[] (v3) at (8,1) {};
\node[] (v4) at (8,0) {};
\node[xshift=8pt] at (8,0) {\small $w_j$};
\node[xshift=8pt] at (8,1) {\small $\bar{w}_i$};
\node[xshift=-8pt] at (7,1) {\small $w_i$};
\node[xshift=-8pt] at (7,0) {\small $w_j$};

\draw[thick,black!40!green, ->-=.6] (v2.center) -- (v1.north)  ;
    \draw[thick,red, ->-=.5] (v2.center) --  (v3.center)  ;
    \draw[thick,blue, ->-=.5] (v3.center) -- (v4.center);
    \draw[thick, ->-=.5](v1.east) --  (v4.center) ;

\node at (7.5, -1) {\small $\mathfrak{m}_3(w_i,\bar{w}_i,w_j) = (-1)^{|\bar{w}_i|+|w_j|} w_j$};
\end{tikzpicture}
\begin{tikzpicture}
    \tikzset{->-/.style={decoration={markings, mark=at position #1 with {\arrow{>}}},postaction={decorate}}}
  \node[] at (4.5,1.5) {$w_i^- \ldots w_j^+$}; 
   \node[] (v1) at (1,0) {};
    \node[] (v2) at (1,1) {};
    \node[] (v3) at (2,1) {};
    \node[] (v4) at (2,0) {};
\node[xshift=8pt] at (2,0) {\small $w_i$};
\node[xshift=8pt] at (2,1) {\small $\bar{w}_i$};
\node[xshift=-8pt] at (1,1) {\small $\bar{w}_j$};
\node[xshift=-8pt] at (1,0) {\small $\bar{w}_j$};

\draw[thick,black!40!green, ->-=.6] (v2.center) --  (v1.north);
\draw[thick,red, ->-=.5] (v3.center) -- (v2.center);
\draw[thick,blue, ->-=.5] (v4.center) -- (v3.center);
\draw[thick, ->-=.5] (v1.east) -- (v4.center); 

\node at (1.5, -1) {\small $\mathfrak{m}_3(\bar{w}_j,\bar{w}_i,w_i) = \bar{w}_j$};

\node[] (v1) at (7,0) {};
\node[] (v2) at (7,1) {};
\node[] (v3) at (8,1) {};
\node[] (v4) at (8,0) {};
\node[xshift=8pt] at (8,0) {\small $w_j$};
\node[xshift=8pt] at (8,1) {\small $w_i$};
\node[xshift=-8pt] at (7,1) {\small $\bar{w}_i$};
\node[xshift=-8pt] at (7,0) {\small $w_j$};

\draw[thick,black!40!green, ->-=.6] (v2.center) -- (v1.north)  ;
    \draw[thick,red, ->-=.5] (v3.center) --  (v2.center)  ;
    \draw[thick,blue, ->-=.5] (v3.center) -- (v4.center);
    \draw[thick, ->-=.5](v1.east) --  (v4.center) ;

\node at (7.5, -1) {\small $\mathfrak{m}_3(\bar{w}_i,w_i,w_j) = (-1)^{|w_j|} w_j$};
\end{tikzpicture}
\begin{tikzpicture}
    \tikzset{->-/.style={decoration={markings, mark=at position #1 with {\arrow{>}}},postaction={decorate}}}
  \node[] at (4.5,1.5) {$w_i^- \ldots w_j^-$}; 
   \node[] (v1) at (1,0) {};
    \node[] (v2) at (1,1) {};
    \node[] (v3) at (2,1) {};
    \node[] (v4) at (2,0) {};
\node[xshift=8pt] at (2,0) {\small $w_i$};
\node[xshift=8pt] at (2,1) {\small $\bar{w}_i$};
\node[xshift=-8pt] at (1,1) {\small $w_j$};
\node[xshift=-8pt] at (1,0) {\small $w_j$};

\draw[thick,black!40!green, ->-=.6]  (v1.north) --  (v2.center);
\draw[thick,red, ->-=.5] (v3.center) -- (v2.center);
\draw[thick,blue, ->-=.5] (v4.center) -- (v3.center);
\draw[thick, ->-=.5](v4.center) --  (v1.east) ; 

\node at (1.5, -1) {\small $\mathfrak{m}_3(w_j,\bar{w}_i,w_i) = w_j$};

\node[] (v1) at (7,0) {};
\node[] (v2) at (7,1) {};
\node[] (v3) at (8,1) {};
\node[] (v4) at (8,0) {};
\node[xshift=8pt] at (8,0) {\small $\bar{w}_j$};
\node[xshift=8pt] at (8,1) {\small $w_i$};
\node[xshift=-8pt] at (7,1) {\small $\bar{w}_i$};
\node[xshift=-8pt] at (7,0) {\small $\bar{w}_j$};

\draw[thick,black!40!green, ->-=.6] (v2.center) -- (v1.north)  ;
    \draw[thick,red, ->-=.5] (v3.center) --  (v2.center)  ;
    \draw[thick,blue, ->-=.5] (v3.center) -- (v4.center);
    \draw[thick, ->-=.5]  (v4.center) -- (v1.east) ;

\node at (7.5, -1) {\small $\mathfrak{m}_3(\bar{w}_i,w_i,\bar{w}_j) = (-1)^{|\bar{w}_j|} \bar{w}_j$};
\end{tikzpicture}
\end{center}
In addition, for each $i$, we have a small rectangle and rectangles with corners at $\bar{w}_0$.
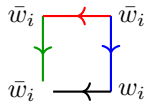
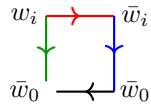
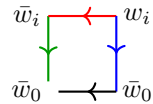
\begin{center}
\begin{tikzpicture}[xscale =0.9]
    \tikzset{->-/.style={decoration={markings, mark=at position #1 with {\arrow{>}}},postaction={decorate}}}
   \node[] (v1) at (1,0) {};
    \node[] (v2) at (1,1) {};
    \node[] (v3) at (2,1) {};
    \node[] (v4) at (2,0) {};
\node[xshift=8pt] at (2,0) {\small $w_i$};
\node[xshift=8pt] at (2,1) {\small $\bar{w}_i$};
\node[xshift=-8pt] at (1,1) {\small $\bar{w}_i$};
\node[xshift=-8pt] at (1,0) {\small $\bar{w}_i$};

\draw[thick,black!40!green, ->-=.6]  (v2.center) --  (v1.north) ;
\draw[thick,red, ->-=.5] (v3.center) -- (v2.center);
\draw[thick,blue, ->-=.5] (v3.center) -- (v4.center);
\draw[thick, ->-=.5](v4.center) --  (v1.east) ; 

\node at (1.5, -1) {$\mathfrak{m}_3(\bar{w}_i,w_i,\bar{w}_i) = (-1)^{|\bar{w}_i|}\bar{w}_i$};

  \node[] (v1) at (7,0) {};
    \node[] (v2) at (7,1) {};
    \node[] (v3) at (8,1) {};
    \node[] (v4) at (8,0) {};
\node[xshift=8pt] at (8,0) {\small $\bar{w}_0$};
\node[xshift=8pt] at (8,1) {\small $\bar{w}_i$};
\node[xshift=-8pt] at (7,1) {\small $w_i$};
\node[xshift=-8pt] at (7,0) {\small $\bar{w}_0$};

\draw[thick,black!40!green, ->-=.6]  (v2.center) --  (v1.north) ;
\draw[thick,red, ->-=.5] (v2.center) -- (v3.center);
\draw[thick,blue, ->-=.5] (v3.center) -- (v4.center);
\draw[thick, ->-=.5](v4.center) --  (v1.east) ; 

\node at (7.5, -1) {$\mathfrak{m}_3(w_i,\bar{w}_i,\bar{w}_0) = (-1)^{|\bar{w}_i|+|\bar{w}_0|}\bar{w}_0$};

  \node[] (v1) at (13,0) {};
    \node[] (v2) at (13,1) {};
    \node[] (v3) at (14,1) {};
    \node[] (v4) at (14,0) {};
\node[xshift=8pt] at (14,0) {\small $\bar{w}_0$};
\node[xshift=8pt] at (14,1) {\small $w_i$};
\node[xshift=-8pt] at (13,1) {\small $\bar{w}_i$};
\node[xshift=-8pt] at (13,0) {\small $\bar{w}_0$};

\draw[thick,black!40!green, ->-=.6]  (v2.center) --  (v1.north) ;
\draw[thick,red, ->-=.5] (v3.center) -- (v2.center);
\draw[thick,blue, ->-=.5] (v3.center) -- (v4.center);
\draw[thick, ->-=.5](v4.center) --  (v1.east) ; 

\node at (13.5, -1) {$\mathfrak{m}_3(\bar{w}_i,w_i,\bar{w}_0) = (-1)^{|\bar{w}_0|}\bar{w}_0$};
\end{tikzpicture}
\end{center}
Here is the general recipe to compute $hom(L,L)$ for a general immersion $\gamma: S^1 \to L \subset \Sigma$, that gives a finitely computable $A_\infty$ structure. Firstly, one has to decide where to put the generators $w_0, \bar{w}_0$. Then, we get a signed Gauss word using the orientation of $L$. Then, we get all the higher products in $\Sigma_{\gamma}$ as in Definition \ref{newalg}. Next, we want to put in the disk regions that are bounding $L$ in $\Sigma$. When we include those regions, $L$ begins to bound immersed polygons (which are all regular due to automatic transversality, \cite[Lemma 13.2]{SeidelBook}). Note that there are obvious disks coming from the components of $\Sigma \setminus L$, but the union of those might lead to even more polygons bounding $L$. There are only finitely many such polygons (but we should not miss any!). Suppose we have such a disk with corners $(v_k,v_{k-1},\ldots, v_0)$ in clockwise order. Then we get $k+1$ higher products
\begin{align*}
  \mathfrak{m}_{k}(v_k,v_{k-1},\ldots, v_1) &=  \pm \bar{v}_0 \\
  \mathfrak{m}_{k}(v_0,v_{k},\ldots, v_2) &=  \pm \bar{v}_1 \\
                                             &\vdots \\
  \mathfrak{m}_{k}(v_{k-1},v_{k-2},\ldots, v_0) &=  \pm \bar{v}_k 
\end{align*}
If the disk has an edge passing through $w_0,\bar{w}_0$ in the counter-clockwise order, we get additional contributions of the form:
\begin{align*}
\mathfrak{m}_{k+1}(v_k,v_{k-1},\ldots, v_0) &=  \pm w_0 \\
 \mathfrak{m}_{k+1}(v_{k-1},\ldots, v_0, \bar{w}_0) &=  \pm \bar{v}_k \\
          & \vdots \\
 \mathfrak{m}_{k+1}(\bar{w}_0, v_k, \ldots, v_1) &=  \pm \bar{v}_0 
\end{align*}
and
\[ \mathfrak{m}_{k+2}(v_k,v_{k-1},\ldots, v_0, \bar{w}_0) =  \pm w_0 \]

The signs are determined as in \cite[Figure 2]{Seidel}.

Finally, if we desire to end up with finitely many higher products, it is important to choose the placement of $\{w_0,\bar{w}_0\}$ so that there are no disks with an edge passing through $\bar{w}_0, w_0$ in the counter-clockwise order. Otherwise, we get infinitely many higher products, though there is still an easy combinatorial pattern for these. 

Let us recall that, given an $A_\infty$-algebra $\mathscr{A}$, a bounding cochain $\mathfrak{b} \in \mathscr{A}$ is a degree 1 element that satisfies the Maurer-Cartan equation:
\[ \mathfrak{m}_1(\mathfrak{b}) + \mathfrak{m}_2(\mathfrak{b},\mathfrak{b}) + \ldots  =0 \]
Given a bounding cochain, we can define a deformed $A_{\infty}$ algebra
 \[ \mathfrak{m}^{\mathfrak{b}}_i (v_i,v_{i-1},\ldots, v_1)= \sum_{j \geq i} \mathfrak{m}_j (\mathfrak{b},\ldots, \mathfrak{b},v_i,\mathfrak{b},\ldots, \mathfrak{b},v_{i-1},\mathfrak{b},\ldots,\ldots,\mathfrak{b},v_1,\mathfrak{b},\ldots, \mathfrak{b}) \] 

 Our goal is to compute the $A_\infty$ algebra $\mathscr{A} = hom(L,L)$ for various immersed Lagrangians $L$. Furthermore, we are interested in bounding cochains satisfying $\mathfrak{m}_1^{\mathfrak{b}}=0$, so that the endomorphism algebra of $(L,\mathfrak{b})$ has the same rank as the endomorphism algebra of $L$. Below, we generally refer to the degree 0 part of the endomorphism algebra of $(L,\mathfrak{b})$ as the deformed algebra. In addition, when we partially compactify the surface, we denote the corresponding deformation variables by $s_1,s_2,\ldots$ and write $\{ \mathfrak{m}^{\mathfrak{b},s}_i \}$ for the deformed $A_\infty$-structure. We again look for flat deformations, that is, we require $\mathfrak{m}_1^{\mathfrak{b},s} =0$.  

Let us first explain that if we only use deformations coming from bounding cochains, we get precisely radical square zero algebras.
 
\begin{prop}\label{proprad2} Suppose that $A$ is any finite-dimensional basic algebra over an algebraically closed field $k$ such that its radical satisfies $\mathrm{rad}(A)^2 =0$. Then $A$ can be realized as the degree-0 endomorphism algebra of an immersed curve $\gamma: S^1 \to L \subset \Sigma_\gamma$ equipped with a bounding cochain $\mathfrak{b}$. Conversely, for any bounding cochain $\mathfrak{b}$ with $\mathfrak{m}_1^{\mathfrak{b}}=0$, the endomorphism algebra of any $(L, \mathfrak{b})$ in $\mathcal{F}(\Sigma_\gamma)$ is a radical square zero algebra. 
\end{prop}
\begin{proof}  By hypothesis, $A$ is isomorphic to a quiver path algebra modulo the ideal generated by all paths of length $\ge 2$. That is, $A \cong kQ / Q_{\ge 2}$. Let the vertices (orthogonal primitive idempotents) of $Q$ be $V = \{0, 1, \dots, n\}$. Let the arrows (basis of the radical) of $Q$ be $E = \{n+1, \dots, n+m\}$. Each arrow $a \in E$ has a specific source vertex $s(a) \in V$ and target vertex $t(a) \in V$. We will construct a signed Gauss word using exactly $r-1 = n+m$ crossings.

For each explicit vertex $v \in V \setminus \{0\}$, we define a contiguous ``block'' of letters $B_v$:
$$B_v = v^- \;\;\; \Big( \prod_{t(a)=v} a^+ \Big) \;\;\; \Big( \prod_{s(c)=v} c^- \Big) \;\;\; v^+$$ (The relative order of the $a^+$ and $c^-$ letters strictly inside the bracket $v^- \dots v^+$ does not matter). We define the signed Gauss word as the sequential concatenation of these vertex blocks, placing any arrows connected to the vertex 0 at the very end:
$$B_1 B_2 \dots B_n \;\;\; \Big( \prod_{t(a)=0} a^+ \Big) \;\;\; \Big( \prod_{s(c)=0} c^- \Big)$$ Because each arrow has exactly one source and one target, $a^+$ and $a^-$ each appear exactly once. Each vertex $v$ also appears exactly twice ($v^-$ and $v^+$). Thus, this is a valid signed Gauss word.

We assign the grading $|w_k| = 0$ for all $k$. This maps the degree-0 generators to $x_k = w_k$. We define a bounding cochain by
\[ \mathfrak{b} = \sum_{v \in V \setminus \{0\} } \bar{w}_v. \]
We claim that $\mathfrak{m}_1^{\mathfrak{b}}=0$. By our grading choices, we have 
\[ \mathfrak{m}_1^{\mathfrak{b}}(w_k) = \mathfrak{m}_3(w_k, \mathfrak{b}, \mathfrak{b}) + \mathfrak{m}_3(\mathfrak{b}, w_k, \mathfrak{b}) + \mathfrak{m}_3(\mathfrak{b}, \mathfrak{b}, w_k)\]
For an arrow $a$, since $\bar{w}_a$ does not appear in $\mathfrak{b}$, no matching pair is possible: $\mathfrak{m}_1^{\mathfrak{b}}(w_a) = 0$. For a vertex $v$, $\mathfrak{m}_1^{\mathfrak{b}}(w_v)$ contains a $-\bar{w}_v$ term coming from $\mathfrak{m}_3(\bar{w}_v, w_v, \bar{w}_v) = - \bar{w}_v$ obtained from the small rectangle. On the other hand, we also have $\mathfrak{m}_3(\bar{w}_v, \bar{w}_v, w_v) = \bar{w}_v$ because, by construction of the Gauss word, we have the subinterval $v^- \dots v^+$. Hence, we have a cancellation. Finally, for cross terms $u \neq v$, the coefficient of $\bar{w}_v$ in $\mathfrak{m}_1^{\mathfrak{b}}(w_u)$ relies on four subintervals: $u^+ \dots v^-$, $u^+ \dots v^+$, $u^- \dots v^+$, and $u^- \dots v^-$. Since blocks $B_u$ and $B_v$ are disjoint, these either all exist (if $u$ precedes $v$) or none exist. The signs of these four $\mathfrak{m}_3$ evaluations yield exactly $+1 -1 +1 -1 = 0$. Thus $\mathfrak{m}_1^{\mathfrak{b}} = 0$.

Next, we calculate the algebra structure. For $v \in V \setminus \{0 \}$. The subinterval $v^- \dots v^+$ exists, yielding \[ \mathfrak{m}_2^{\mathfrak{b}}(w_v,w_v)= w_v.\] For $u \neq v$, if $B_u$ precedes $B_v$, both $u^- \dots v^+$ and $u^+ \dots v^+$ exist. These evaluate to $w_v$ and $-w_v$, perfectly cancelling. (If $B_u$ is after $B_v$, neither exists, again yielding $0$). Thus, $\{w_v\}_{v \in V}$ forms a set of mutually orthogonal primitive idempotents. Finally, we define the idempotent $e_0 = w_0 - \sum w_v$.

For a vertex $v \in V$ and arrow $a \in E$, left-multiplication $\mathfrak{m}_2^{\mathfrak{b}}(w_v,w_a)$ evaluates to $w_a$ if $v^- \dots a^+$ subinterval exists, and $-w_a$ if $v^+ \dots a^+$ exists. This coefficient is exactly $1$ if and only if $a^+$ is located strictly between $v^-$ and $v^+$. By our construction of the Gauss word, this happens if and only if $a^+$ was placed inside block $B_v$, which means $t(a) = v$. If $t(a) = 0$, $a^+$ is at the end, so $\mathfrak{m}_2^{\mathfrak{b}}(w_v, w_a) = 0$ for all $v$, forcing $\mathfrak{m}_2^{\mathfrak{b}}(e_0, w_a) = w_a$. Similarly, the right-action yields $\mathfrak{m}_2^{\mathfrak{b}}(w_a ,w_v) = w_a$ if and only if $a^-$ is strictly between $v^-$ and $v^+$, which is equivalent to $s(a) = v$.

Finally, we check that radical square is zero. For any two arrows $a, c \in E$, the product $\mathfrak{m}_2^{\mathfrak{b}}(w_a, w_c)$ involves $\mathfrak{m}_3$ evaluations with $\mathfrak{b}$. Since $\mathfrak{b}$ contains only $\bar{w}_v$, the inputs $\{w_a, w_c, \bar{w}_v\}$ lack any matching pair. Thus,  $\mathfrak{m}_2^{\mathfrak{b}}(w_a, w_c) =0$.

We have reproduced the vertices $V$, the arrows $E$, the source/target structures, and the generated algebra is exactly isomorphic to $A \cong kQ / Q_{\ge 2}$.

Conversely, suppose that $\mathfrak{b}$ is a bounding cochain with $\mathfrak{m}_1^{\mathfrak{b}}=0$. Write $1=x_0, x_1,\ldots, x_{r-1}$ for degree 0 generators, and $\mathfrak{b} = \sum_{k=1}^{r-1 } t_k y_k$. Recall that
\[ \mathfrak{m}_2^{\mathfrak{b}}(x_i,x_j) = \mathfrak{m}_2(x_i, x_j) + \mathfrak{m}_3(x_i, x_j, \mathfrak{b}) + \mathfrak{m}_3(x_i, \mathfrak{b}, x_j) + \mathfrak{m}_3(\mathfrak{b}, x_i, x_j).\]
By definition, $\mathfrak{m}_2(x_i, x_j) = 0$ as $\mathfrak{m}_2$ evaluates non-trivially only if the inputs form a matching pair $(w_k, \bar{w}_k)$, which requires one degree-0 and one degree-1 element. $\mathfrak{m}_3$ evaluates non-trivially only when exactly two adjacent inputs form a matching pair, which then ``absorbs'' the third element. Since $x_i$ and $x_j$ cannot match each other, the match must happen between one of them and a $y_k$ from the bounding cochain $\mathfrak{b}$. So $\mathfrak{b}$ must contain $t_i y_i$. This matching pair ``absorbs'' the remaining element $x_j$, outputting a multiple of $x_j$. Likewise, $x_j$ can match with $y_j$ (supplying $t_j$), which absorbs $x_i$ and outputs a multiple of $x_i$. Therefore, 
\[ \mathfrak{m}_2^{\mathfrak{b}}(x_i,x_j) = C_{ij} t_i x_j + C_{ji}' t_j x_i \]
where $C_{ij}, C_{ji}' \in \mathbb{Z}$ depend on the subintervals connecting $i$ and $j$.

Let us partition our generators by defining a subspace $V_0 = \text{span}(x_k :  t_k = 0 \text{ for } k \geq 1)$.
If we multiply two elements $x_i, x_j \in V_0$, then $t_i = t_j = 0$. Substituting this into our formula gives $\mathfrak{m}_2^{\mathfrak{b}}(x_i,x_j) = 0$. Hence, $V_0^2 = 0$.

Furthermore, if we multiply an element $x_i \in V_0$ with an element $x_k \notin V_0$, then $t_i = 0$. The product becomes
\[  \mathfrak{m}_2^{\mathfrak{b}}(x_i,x_k) = C_{ki}' t_k x_i \in V_0.\]
Because $V_0$ is a two-sided ideal that squares to 0, it is a nilpotent ideal, meaning it is entirely contained within $\mathrm{rad}\, A$.

Consider the quotient algebra $A / V_0$, which is generated by the elements $x_k$ where $t_k \neq 0$ and the unit $x_0$. To find the multiplication rules here, we evaluate the differential:
\[ \mathfrak{m}_1^{\mathfrak{b}}(x_k) = \mathfrak{m}_2(\mathfrak{b}, x_k) + \mathfrak{m}_2(x_k, \mathfrak{b}) + \mathfrak{m}_3(x_k, \mathfrak{b}, \mathfrak{b}) + \mathfrak{m}_3(\mathfrak{b}, x_k, \mathfrak{b}) + \mathfrak{m}_3(\mathfrak{b}, \mathfrak{b}, x_k) \]
The $\mathfrak{m}_2$ terms evaluate strictly to multiples of $\bar{w}_0$. For the $y_l$ coefficients, we only look at the $\mathfrak{m}_3$ terms. Because $\mathfrak{m}_3$ requires a matching pair and the match must be exactly between $x_k$ and the $t_k y_k$ term from $\mathfrak{b}$, the third, remaining element must be some $t_l y_l$ from the other $\mathfrak{b}$. Thus, every non-zero $\mathfrak{m}_3$ term in $\mathfrak{m}_1^{\mathfrak{b}}(x_k)$ is proportional to $t_k t_l y_l$. We denote the coefficient by $E_{kl}$ as in
\[ \mathfrak{m}_1^{\mathfrak{b}}(x_k) = E_{kl} t_k t_l y_l + \ldots \]
The contributions to $E_{kl}$ can come from the following four terms:
\[ \mathfrak{m}_3 (x_k, y_k , y_l),\; \mathfrak{m}_3(y_k, x_k, y_l),\; \mathfrak{m}_3(y_l,x_k,y_k),\; \mathfrak{m}_3(y_l, y_k ,x_k) \]
and the contributions to $C_{kl}$ can come from the following two terms:
\[ \mathfrak{m}_3(x_k, y_k, x_l),\; \mathfrak{m}_3(y_k, x_k, x_l) \]
and the contributions to $C_{kl}'$ can come from the following two terms:
\[ \mathfrak{m}_3(x_l, x_k, y_k),\; \mathfrak{m}_3(x_l, y_k,x_k) \]
Now, from Definition \ref{newalg}, we see the following correspondence between these terms
\begin{align*}
  \mathfrak{m}_3(x_k,y_k, y_l) \neq 0 &\iff  \mathfrak{m}_3(x_l,x_k,y_k) \neq 0 \\
  \mathfrak{m}_3(y_k,x_k, y_l) \neq 0 &\iff  \mathfrak{m}_3(x_l,y_k,x_k) \neq 0 \\
  \mathfrak{m}_3(y_l ,x_k,y_k) \neq 0 &\iff  \mathfrak{m}_3(x_k,y_k,x_l) \neq 0 \\
  \mathfrak{m}_3(y_l , y_k,x_k) \neq 0 &\iff  \mathfrak{m}_3(y_k,x_k,x_l) \neq 0 
\end{align*}
Because we require that $\mathfrak{m}_1^{\mathfrak{b}}=0$, the coefficient $E_{kl} t_k t_l$ must vanish for all $l$. Since we are evaluating the quotient $A / V_0$, we know $t_k \neq 0$ and $t_l \neq 0$, so we have $E_{kl} = 0$. From the above correspondence, it follows that $C_{kl} = C_{kl}'$ and $C_{lk} = C_{lk}'$. As a result, we get 
\[ \mathfrak{m}_2^{\mathfrak{b}} (x_k,x_l) = \mathfrak{m}_2^{\mathfrak{b}}(x_l,x_k) \quad \text{for all } k \neq l \text{ in } A/V_0. \]
We also have that $\mathfrak{m}_2^{\mathfrak{b}}(x_k,x_k) = m_k t_kx_k$. Any finite-dimensional commutative algebra generated by such elements over an algebraically closed field is isomorphic to a product $k \times k \times \ldots \times k$. Therefore, the quotient $A/V_0$ is semisimple, meaning its radical is $0$. So, $\mathrm{rad}\, A = V_0 $. Since we had shown that $V_0^2 = 0$, we conclude that $(\mathrm{rad}\, A)^2 = 0$. \end{proof}

Hence, before the partial compactifications of $\Sigma_\gamma$, the deformed algebra always has radical square zero. Thus, we need the surface compactifications in order to ``unlock'' algebras with longer nilpotent chains like $k[x]/(x^3)$. 

\begin{example} As an example of a radical square zero algebra, let us consider the 3-Kronecker quiver, with vertices $V = \{0,1\}$ and three edges $E=\{2,3,4\}$ all have source 0 and target 1. The block $B_1$ associated with $v=1$ is $1^{-} 2^{+} 3^{+} 4^{+} 1^{+}$ and the signed Gauss word from our construction is
   \[ 1^{-} 2^{+} 3^{+} 4^{+} 1^{+} 2^{-} 3^{-} 4^{-} \]
   If we now consider $\mathfrak{b} = \bar{w}_1$, then we can calculate that the result gives the 3-Kronecker quiver.
\end{example}

 We next begin to work through the immersions with 3 or fewer nodes. We will not give full details in every case to avoid too much repetition. However, we provide a fully explicit description of all the $A_\infty$-products in most cases so that the reader can feed this to their favourite computer algebra system to check that what we calculated indeed satisfies the $A_\infty$ relations. In our experience, this seems to be a highly non-trivial (and satisfying) check as a single sign error or a missing polygon causes the violation of many of the $A_\infty$ relations.  

\subsection*{n=1 : \texorpdfstring{\protect\eightcurve}{} $1^+1^-$}
Here is the figure-eight curve and its canonical neighborhood $\Sigma$, which is a three-punctured sphere. In the figure, the left and right sides are identified. We labelled the three boundary components with the letters $m_1,m_2,s$, where $m_1,m_2$ are monogons, and $s$ is a bigon. The gradings are related by $|s| = 2 -2|w|$, $|m_1|=|m_2|=|w|$.  

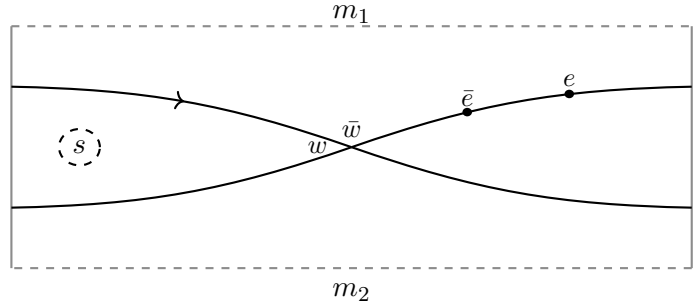
\begin{center}
  \begin{tikzpicture}[xscale=0.9, yscale=0.8]
    \tikzset{->-/.style={decoration={markings, mark=at position #1 with {\arrow{>}}},postaction={decorate}}}

    \node[] (t1) at (0,0) {};
    \node[] (t2) at (0,4) {};
    \node[] (t3) at (10,4) {};
    \node[] (t4) at (10,0) {};

    \node[] (L0left) at (0,1) {};
    \node[] (L0mid)  at (5,2) {};
    \node[] (L0right) at (10,3) {};

    \node[] (L1left) at (0,3) {};
    \node[] (L1mid)  at (5,2) {};
    \node[] (L1right) at (10,1) {};

    \node[] (x1) at (4.5,2) {\small $w$}; 
    \node[] (x2) at (5, 2.3) {\small $\bar{w}$};

    \fill[black] (6.7,2.58) circle (2pt);    
    \fill[black] (8.2,2.88) circle (2pt);    

    \node[] (q) at (6.7,2.8) {\small $\bar{e}$};
    \node[] (e) at (8.2,3.1) {\small $e$};
    
    \draw[gray!90, thick, dashed] (t2.center) -- (t3.center);
    \draw[gray!90, thick, dashed] (t4.center) -- (t1.center);
    \draw[gray!90, thick] (t1.center) -- (t2.center);
    \draw[gray!90, thick] (t3.center) -- (t4.center); 

    \draw[thick, dashed] (1,2) circle (0.3cm);
    \node (t2) at (1,2) {$s$};

    \node (t1) at (5,4.2) {$m_1$};
    \node (t1) at (5,-0.4) {$m_2$};

    \draw[thick] (L0left.center) edge[bend right=10] (L0mid.center);
    \draw[thick] (L0mid.center) edge[bend left=10] (L0right.center);

    \draw[thick, ->-=.5] (L1left.center) to[bend left=10] (L1mid.center);
    \draw[thick] (L1mid.center) edge[bend right=10] (L1right.center);

  \end{tikzpicture}
\end{center}
Let us choose $|w|=0$. There are the following triple products:
\begin{align*}
   \mathfrak{m}_3(w, w, \bar{w}) &= -w \\
   \mathfrak{m}_3(w,\bar{w} ,\bar{w}) &= \bar{w} = - \mathfrak{m}_3(\bar{w},w,\bar{w}) \\
    \mathfrak{m}_3 (w,\bar{w},\bar{e}) &= \bar{e}  = - \mathfrak{m}_3 (\bar{w},w,\bar{e})  
\end{align*}
Compactifying the $s$ component, we get the additional products
\begin{align*}
\mathfrak{m}_2(w,w) &= s e\\
\mathfrak{m}_2(\bar{e},w) = -\mathfrak{m}_2(w,\bar{e}) &= s \bar{w}\\
\mathfrak{m}_3(w,w,\bar{e}) &= -s e
\end{align*}
Let $\mathfrak{b}= t \overline{w}$. We have $\mathfrak{m}_1^{\mathfrak{b},s}=0$, and the deformed algebra over $k[t,s]$ is given by
\[ \boxed{k[x]/ (x^2 + tx -s)} \]
In particular, note that by Example \ref{ex1}, every rank 2 algebra appears in this family.

\begin{rmk} There is a mirror symmetric description available. The mirror family (with respect to compactifying $s$) is given by $k[x,y,s]/(xy-s)$ and the mirror to the $t$-family of figure-eight Lagrangians corresponds to the module $k[x,y,s,t]/(x-y+t, xy-s)$. 
\end{rmk}

We have already obtained all possible two-dimensional algebras above; however, to illustrate some of the additional possible choices and complications, and to prepare ourselves for the more complicated examples, let us do another exercise on the figure-eight curve. Consider the Gauss word $1^-1^+$ and let us choose $|w|=1$. Then, the product is given by
\begin{multicols}{3}
\begin{itemize}[label={}]
\item  $\mathfrak{m}_2(e,e) = e$, 
\item  $\mathfrak{m}_2(w,e) = w$, 
\item  $\mathfrak{m}_2(e,w) = -w$, 
\item  $\mathfrak{m}_2(e,\bar{w}) = \bar{w}$, 
\item  $\mathfrak{m}_2(\bar{w},e) = \bar{w}$, 
\item  $\mathfrak{m}_2(e,\bar{e}) = -\bar{e}$, 
\item  $\mathfrak{m}_2(\bar{e},e) = \bar{e}$, 
\item  $\mathfrak{m}_2(w,\bar{w}) = \bar{e}$, 
\item  $\mathfrak{m}_2(\bar{w},w) = \bar{e}$. 
\end{itemize}
\end{multicols}
\vspace{-.1in}
and the triple products are 
\begin{align*}
   \mathfrak{m}_3(\bar{w}, w, w) &= -w \\
   \mathfrak{m}_3(\bar{w} ,\bar{w},w) &= \bar{w} \; = \mathfrak{m}_3(\bar{w},w,\bar{w}) \\
    \mathfrak{m}_3 (w,\bar{w},\bar{e}) &= - \bar{e}  = \mathfrak{m}_3 (\bar{w},w,\bar{e})  
\end{align*}
Now, we can no longer compactify $s$, as the winding number around $s$ is 0. However, the winding number around $m_1$ and $m_2$ are 1, and hence, they can each be compactified with a $\mathbb{Z}/2\mathbb{Z}$-orbifold point, and the grading extends across that. As a result, we get the following additional contributions:
\begin{multicols}{3}
\begin{itemize}[label={}]
\item $\mathfrak{m}_1(\bar{w}) = -m_1 w$,
\item $\mathfrak{m}_1(\bar{w}) = m_2 w$,
\item $\mathfrak{m}_2(\bar{w},\bar{w})= m_1 e$,
\item $\mathfrak{m}_2(\bar{w},\bar{e})= m_1 w$,
\item $\mathfrak{m}_2(\bar{e},\bar{w})= m_1 w$
\item $\mathfrak{m}_3 (\bar{e}, \bar{w},\bar{e}) = -m_1 w$
\item $\mathfrak{m}_3 (\bar{w}, \bar{w},\bar{e}) = -m_1 e$
\item $\mathfrak{m}_3 (\bar{w}, \bar{e},\bar{w}) = -m_1 e$
\item $\mathfrak{m}_4 (\bar{w}, \bar{e}, \bar{w}, \bar{e}) = m_1e$.  
\end{itemize}
\end{multicols}
\vspace{-.1in}
If we now let $\mathfrak{b}= tw + \lambda \bar{w}_0$,  then we have
\[ \mathfrak{m}_1^{\mathfrak{b},m_1,m_2}(\bar{w})= (-t^2 - m_1 + m_2 + 2m_1\lambda - m_1\lambda^2)w + 2(t - t\lambda) \bar{w}_0 \]
Writing $x=\overline{w}$, we get a flat deformation when either $t=0$ and $m_2 = m_1(1-\lambda)^2$, in which case the deformed algebra over $k[m_1,\lambda]$ is
\[ \boxed{k[x]/  (x^2 - m_1 (1-\lambda)^2)} \]
or if $\lambda=1$ and $m_2 = t^2$, then the deformation over $k[m_1,t]$ is given by
\[ \boxed{k[x] / (x^2 - 2tx)} \]
\subsection*{n=2 : \texorpdfstring{\tripleeye}{} $1^+2^-2^+1^-$}
This curve is immersed in a four-punctured sphere. The boundary components are $m_1,m_2,s_1, s_2$, where $m_1,m_2$ are monogons, $s_1$ is a bigon, and $s_2$ is a rectangle. The gradings are related by $|s_1| = |w_1| + |w_2|$, $|s_2| = 4- 2|w_1| -2|w_2|$, $|m_1|=|w_1|, |m_2|=|w_2|$. We do not give explicit computations in this case. 

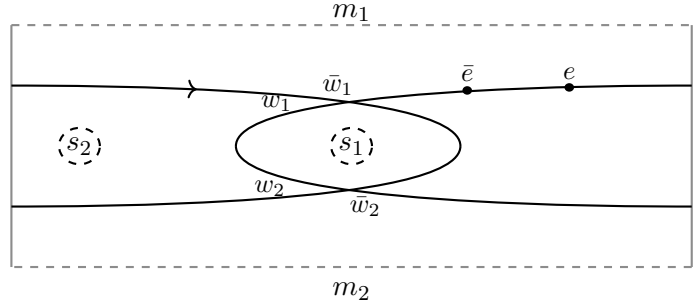
\begin{center}
  \begin{tikzpicture}[xscale=0.9, yscale=0.8]

    \tikzset{->-/.style={decoration={markings, mark=at position #1 with {\arrow{>}}},postaction={decorate}}}

    \node[] (t1) at (0,0) {};
    \node[] (t2) at (0,4) {};
    \node[] (t3) at (10,4) {};
    \node[] (t4) at (10,0) {};

    \node[] (L0left) at (0,1) {};
    \node[] (L0mid)  at (6.6,2) {};
    \node[] (L0right) at (0,3) {};

    \node[] (L1left) at (10,3) {};
    \node[] (L1mid)  at (3.3,2) {};
    \node[] (L1right) at (10,1) {};

    \node[] (x1) at (3.9,2.7) {\small $w_1$}; 
    \node[] (x2) at (4.8, 3) {\small $\bar{w}_1$};

    \node[] (y1) at (3.8,1.3) {\small $w_2$}; 
    \node[] (y2) at (5.2, 1) {\small $\bar{w}_2$};

    \fill[black] (6.7,2.92) circle (2pt);    
    \fill[black] (8.2,2.97) circle (2pt);    

    \node[] (q) at (6.7,3.2) {\small $\bar{e}$};
    \node[] (e) at (8.2,3.2) {\small $e$};
    
    \draw[gray!90, thick, dashed] (t2.center) -- (t3.center);
    \draw[gray!90, thick, dashed] (t4.center) -- (t1.center);
    \draw[gray!90, thick] (t1.center) -- (t2.center);
    \draw[gray!90, thick] (t3.center) -- (t4.center); 

    \draw[thick, dashed] (1,2) circle (0.3cm);
    \node (s2) at (1,2) {$s_2$};

    \node (s1) at (5,4.2) {$m_1$};
    \node (s3) at (5,-0.4) {$m_2$};

    \draw[thick, dashed] (5,2) circle (0.3cm);

    \node (s4) at (5,2) {$s_1$};
    
    \draw[thick] (L0left.center)  .. controls +(0:1) and +(270:1) .. (L0mid.center);
    \draw[thick, ->-=.4] (L0right.center)  .. controls +(0:1) and +(90:1) .. (L0mid.center);

    \draw[thick] (L1left.center)  .. controls +(180:1) and +(90:1) .. (L1mid.center);
    \draw[thick] (L1mid.center)  .. controls +(270:1) and +(180:1) .. (L1right.center);

  \end{tikzpicture}
\end{center}
\subsection*{n=2 : \texorpdfstring{\protect\eightinner}{} $1^+2^+2^-1^-$}
This curve is immersed in a four-punctured sphere. The boundary components are $m_1, m_2, s_1, s_2$, where $m_1$ and $m_2$ are monogons, and $s_1$ and $s_2$ are triangles. The gradings are related by $|s_1|= 2 + |w_1| - 2|w_2|$, $|s_2|= 2 - 2|w_1| + |w_2|$, $|m_1| = |w_1|$ and $|m_2|=|w_2|$.   
\begin{center}
  \begin{tikzpicture}[xscale=0.9, yscale=0.8]

    \tikzset{->-/.style={decoration={markings, mark=at position #1 with {\arrow{>}}},postaction={decorate}}}

    \node[] (t1) at (0,0) {};
    \node[] (t2) at (0,4) {};
    \node[] (t3) at (10,4) {};
    \node[] (t4) at (10,0) {};

    \node[] (L0left) at (0,1.9) {};
    \node[] (L0right) at (10,0.4) {};

    \node[] (L1left) at (0,3.4) {};
    \node[] (L1right) at (10,1.9) {};

    \node[] (L2left) at (0,0.4) {};
    \node[] (L2right) at (10,3.4) {};

    \draw[gray!90, thick, dashed] (t2.center) -- (t3.center);
    \draw[gray!90, thick, dashed] (t4.center) -- (t1.center);
    \draw[gray!90, thick] (t1.center) -- (t2.center);
    \draw[gray!90, thick] (t3.center) -- (t4.center); 


    \draw[thick, dashed] (3,2.2) circle (0.3cm);
    \node (s1) at (3,2.2) {$s_2$};
    \draw[thick, dashed] (8,1.5) circle (0.3cm);
    \node (s2) at (8,1.5) {$s_1$};

    \node (s1) at (5,4.2) {$m_1$};
    \node (s3) at (5,-0.4) {$m_2$};

    \draw[thick] (L0left.center) .. controls +(0:1) and +(180:1) .. (L0right.center);

     \draw[thick] (L2left.center) .. controls +(0:1) and +(180:1) .. (L2right.center);

    \draw[thick, ->-=.5] (L1left.center) .. controls +(0:1) and +(180:1) .. (L1right.center);

    \fill[black] (7.4,2.7) circle (2pt);    
    \fill[black] (8.4,3) circle (2pt);    

    \node[] (q) at (7.4,3) {\small $\bar{e}$};
    \node[] (e) at (8.3,3.2) {\small $e$};
    

    \node[] (x1) at (2.8,1.35) {\small $w_2$}; 
    \node[] (y1) at (5.8,2.35) {\small $w_1$}; 

  \end{tikzpicture}
\end{center}
  \begin{multicols}{3}\small
\begin{itemize}[label={}]
\item $\mathfrak{m}_3(w_{1},w_{1},\bar{w}_{1})= - w_{1}$
\item $\mathfrak{m}_3(w_{1},w_{2},\bar{w}_{2})= - w_{1}$
\item $\mathfrak{m}_3(w_{1},\bar{w}_{1},w_{2})= - w_{2}$
\item $\mathfrak{m}_3(w_{1},\bar{w}_{1},\bar{w}_{0})=\bar{w}_{0}$
\item $\mathfrak{m}_3(w_{1},\bar{w}_{1},\bar{w}_{1})=\bar{w}_{1}$
\item $\mathfrak{m}_3(w_{1},\bar{w}_{1},\bar{w}_{2})=\bar{w}_{2}$
\item $\mathfrak{m}_3(w_{1},\bar{w}_{2},w_{2})=w_{1}$
\item $\mathfrak{m}_3(w_{2},w_{1},\bar{w}_{1})= - w_{2}$
\item $\mathfrak{m}_3(w_{2},w_{2},\bar{w}_{2})= - w_{2}$
\item $\mathfrak{m}_3(w_{2},\bar{w}_{2},\bar{w}_{0})=\bar{w}_{0}$
\item $\mathfrak{m}_3(w_{2},\bar{w}_{2},\bar{w}_{1})=\bar{w}_{1}$
\item $\mathfrak{m}_3(w_{2},\bar{w}_{2},\bar{w}_{2})=\bar{w}_{2}$
\item $\mathfrak{m}_3(\bar{w}_{1},w_{1},\bar{w}_{0})= - \bar{w}_{0}$
\item $\mathfrak{m}_3(\bar{w}_{1},w_{1},\bar{w}_{1})= - \bar{w}_{1}$
\item $\mathfrak{m}_3(\bar{w}_{2},w_{1},\bar{w}_{1})= - \bar{w}_{2}$
\item $\mathfrak{m}_3(\bar{w}_{2},w_{2},\bar{w}_{0})= - \bar{w}_{0}$
\item $\mathfrak{m}_3(\bar{w}_{2},w_{2},\bar{w}_{1})= - \bar{w}_{1}$
\item $\mathfrak{m}_3(\bar{w}_{2},w_{2},\bar{w}_{2})= - \bar{w}_{2}$
\end{itemize}
\end{multicols}

  Turning on the $s_1$-deformation, we get the following contributions:
  \begin{align*}
    \mathfrak{m}_2 (w_2,w_2) &= s_1w_1 \\
    \mathfrak{m}_2 (w_2,\bar{w}_1) &= -s_1\bar{w}_2 \\
    \mathfrak{m}_2 (\bar{w}_1,w_2) &= s_1\bar{w}_2 
  \end{align*}
    Let $\mathfrak{b} = t_1 \bar{w}_1 + t_2 \bar{w}_2 $. We have $\mathfrak{m}_1^{\mathfrak{b},s_1} =0$ for all $\mathfrak{b}$. Letting $(x_1,x_2)=(w_1,w_2)$, the deformed algebra over $k[t_1,t_2,s_1]$ is given by
    \[  \boxed{k[x_1,x_2] / (x_1^2 + t_1x_1, x_2^2 + t_2x_2-s_1x_1, x_1x_2 + t_1x_2)}  \]
    Let $\Delta = t_2^2 - 4s_1t_1$, then we have the following isomorphism types of algebras.
\[    
    \begin{tabular}{c|c}
\hline
$t_1 \neq 0, \Delta \neq 0$ & $k \times k \times k$ \\
($t_1 \neq 0, \Delta = 0$) or ($t_1 = 0, t_2 \neq 0$) & $k \times k[x]/(x^2)$ \\
$t_1 = 0, t_2 = 0, s_1 \neq 0$ & $k[x]/(x^3)$ \\
$t_1 = 0, t_2 = 0, s_1 = 0$ & $k[x,y]/(x,y)^2$\\
\hline
\end{tabular}
\]
In particular, every isomorphism type in the commutative component of $\mathbf{Alg}_3$ is recovered. Therefore, turning on the $s_2$-deformation will not give rise to new algebras, although the full computation is still doable. We only record the relevant additional products for our purpose:
\begin{align*}      \mathfrak{m}_2(w_1,w_1) &=s_2w_2  \\
                    \mathfrak{m}_2(w_2,w_1) &= s_1s_2e \\
                   \mathfrak{m}_2(w_1,w_2) & = s_1 s_2 e  \\
        \mathfrak{m}_3 (w_1, \bar{w}_2,w_1) &= s_2e  
\end{align*}
One still has $\mathfrak{m}_1^{\mathfrak{b},s_1,s_2}=0$, and the deformed algebra over $k[t_1,t_2,s_1,s_2]$ is
\[  \boxed{k[x_1,x_2] / (x_1^2 + t_1x_1 - s_2x_2 - s_2 t_2, x_2^2 + t_2x_2-s_1x_1, x_1x_2 + t_1x_2 - s_1s_2 )}  \]
 The behaviour of this algebra is governed by the binary cubic
\[ f(x,y)= -s_2 x^3 + t_1 x^2 y + t_2 xy^2 + s_1 y^3 \]
A remarkable fact about this family is that (letting  $k=\mathbb{Z}$), it matches perfectly with the Delone-Faddeev classification of cubic rings (see \cite{Bhargava}) which shows that isomorphism classes of cubic rings (free of rank 3 over $\mathbb{Z}$) are parametrized by $GL_2(\mathbb{Z})$-equivalence classes of integral binary cubic forms. Indeed, if we let $\omega = x_1+t_1$ and $\theta = x_2$, the multiplication table given in \cite[Eq. 3]{Bhargava} matches with our family.

\begin{rmk} One can think of this in terms of mirror symmetry as follows. The binary cubic $f$ gives rise to a map of sheaves
\[ \mathcal{O}_{\mathbb{P}^1}(-3) \xlongrightarrow{f}  \mathcal{O}_{\mathbb{P}^1} \]
The image of this map defines the structure sheaf $\mathcal{O}_Z = \mathcal{O}_{\mathbb{P}^1} / f(\mathcal{O}_{\mathbb{P}^1}(-3))$ of a 0-dimensional scheme $Z$ and the cubic ring is the global functions on $Z$. Our Lagrangian immersion can be understood as the mirror dual to $\mathcal{O}_Z$. Though, this only works for $f \neq 0$, as the rank of $H^0(\mathcal{O}_{\mathbb{P}^1}/f(\mathcal{O}_{\mathbb{P}^1}(-3)))$ drops when $f=0$. One solution is to consider the hypercohomology of $\mathcal{O}_{\mathbb{P}^1}(-3) \xlongrightarrow{f} \mathcal{O}_{\mathbb{P}^1}$ as a complex in degrees $-1$ and $0$. This was pointed out by Deligne (in some letter!) and was further studied and generalized by Wood (see \cite{wood}). An approach that is more directly related is to consider a degeneration of $\mathbb{P}^1$ to a nodal curve mirror to the four-punctured sphere as in \cite{LPaus} or \cite{EvansLekili}.  
\end{rmk}

 \subsection*{n=2 : $1^+2^+1^-2^-$}
 This curve is immersed in a twice-punctured torus. In the figure, the left-right, and the top-bottom sides are identified. The boundary components are $s_1, s_2$, where $s_1$ is a bigon, and $s_2$ is a hexagon. The gradings are determined by $|s_1| = 2- |w_1| - |w_2|$ and $ |s_2| = 2 + |w_1| + |w_2|$.   
\begin{center}
  \begin{tikzpicture}[xscale=0.9, yscale=0.8]

    \tikzset{->-/.style={decoration={markings, mark=at position #1 with {\arrow{>}}},postaction={decorate}}}

    \node[] (t1) at (0,0) {};
    \node[] (t2) at (0,4) {};
    \node[] (t3) at (10,4) {};
    \node[] (t4) at (10,0) {};

    \node[] (L0left) at (0,1.9) {};
    \node[] (L0up)  at (6,4) {};
    \node[] (L0down) at (6, 0) {};
    \node[] (L0right) at (10,0.4) {};

    \node[] (L1left) at (0,3.3) {};
    \node[] (L1right) at (10,1.9) {};

    \node[] (L2left) at (0,0.4) {};
    \node[] (L2right) at (10,3.3) {};
       
    \draw[gray!90, thick] (t1.center) -- (t2.center) -- (t3.center) -- (t4.center) -- (t1.center);

   \draw[thick, dashed] (2.5,2) circle (0.3cm);
    \node (s2) at (2.5,2) {$s_2$};
    \draw[thick, dashed] (1,1.1) circle (0.3cm);
    \node (s1) at (1,1.1) {$s_1$};

    \draw[thick] (L2right.center) .. controls +(180:1) and +(330:1) .. (L0up.center);
    \draw[thick] (L1left.center) .. controls +(0:1) and +(180:1) .. (L0right.center);

    \draw[thick] (L0left.center) .. controls +(0:1) and +(150:1) .. (L0down.center) ;
   
    \draw[thick, ->-=.1] (L2left.center) .. controls +(0:1) and +(180:1) .. (L1right.center);
    
    \fill[black] (7.9,1.62) circle (2pt);    
    \fill[black] (8.7,1.74) circle (2pt);    

    \node[] (q) at (7.9,1.9) {\small $\bar{e}$};
    \node[] (e) at (8.7,2) {\small $e$};
    
    \node[] (x1) at (5.8,1.44) {\small $w_1$}; 
    \node[] (y1) at (3,1) {\small $w_2$}; 
  \end{tikzpicture}
\end{center}
The triple product contributions are as follows:
  \begin{align*}
    \mathfrak{m}_3(w_1,w_1,\bar{w}_1) = \mathfrak{m}_3(w_1,w_2,\bar{w}_2)   &= -w_1 \\ 
    \mathfrak{m}_3(w_1,\bar{w}_1,\bar{w}_1) = \mathfrak{m}_3(w_2,\bar{w}_2,\bar{w}_1) = - \mathfrak{m}_3(\bar{w}_1,w_1,\bar{w}_1) &= \bar{w}_1 \\
    -\mathfrak{m}_3(w_2,\bar{w}_1, w_1) =  \mathfrak{m}_3(w_1,\bar{w}_1,w_2) = \mathfrak{m}_3(w_2,w_2,\bar{w}_2) = \mathfrak{m}_3(w_2,w_1,\bar{w}_1) &= -w_2  \\
    -\mathfrak{m}_3(\bar{w}_1,w_1,\bar{w}_2) = -\mathfrak{m}_3(\bar{w}_2,w_1,\bar{w}_1) =  \mathfrak{m}_3(w_2, \bar{w}_2,\bar{w}_2) = \mathfrak{m}_3 (w_1,\bar{w}_1,\bar{w}_2) = -\mathfrak{m}_3 (\bar{w}_2,w_2,\bar{w}_2) &=\bar{w}_2 \\
    \mathfrak{m}_3 (w_1,\bar{w}_1,\bar{e}) =  \mathfrak{m}_3 (w_2,\bar{w}_2,\bar{e}) = - \mathfrak{m}_3 (\bar{w}_2,w_2,\bar{e})  = -  \mathfrak{m}_3 (\bar{w}_1,w_1,\bar{e}) &= \bar{e} 
  \end{align*}
  Let $\mathfrak{b} = t_1 \bar{w}_1 + t_2 \bar{w}_2$, then the condition that $\mathfrak{m}_1^{\mathfrak{b}} =0$ can be computed from
  \begin{align*} 
    \mathfrak{m}^{\mathfrak{b}}_1 (w_1) & = -t_1t_2 \cdot \bar{w}_2 \\
    \mathfrak{m}^{\mathfrak{b}}_1 (w_2) & = t_1t_2  \cdot\bar{w}_1  
\end{align*}
Hence, the deformed algebra over $k[t_1,t_2] / (t_1t_2)$ is given by
\[  \boxed{k\{w_1,w_2\} / (w_1^2 + t_1 w_1 , w_2^2 + t_2w_2, w_1w_2 + t_1 w_2 + t_2 w_1, w_2w_1)} \]
We give the isomorphism types of the algebras occurring in this family in the following table.
  \[
\begin{array}{c|c}
\hline 
t_1 \neq 0,\; t_2 = 0 
& $\begin{tikzcd}[ampersand replacement=\&, column sep=small]
{\scriptstyle\bullet} \arrow[r] \& {\scriptstyle\bullet}
\end{tikzcd}$ \\[6pt]
t_1 = 0,\; t_2 \neq 0 
& $\begin{tikzcd}[ampersand replacement=\&, column sep=small]
{\scriptstyle\bullet} \arrow[r] \& {\scriptstyle\bullet}
\end{tikzcd}$ \\[6pt]
t_1 = 0,\; t_2 = 0 
& k[w_1,w_2]/(w_1,w_2)^2 \\ \hline
\end{array}
\]
Turning on the $s_1$-deformation, we additionally get the following contributions:
\begin{align*} \mathfrak{m}_1 (w_1) &= -s_1\bar{w}_2 , \quad \mathfrak{m}_1 (w_2) = s_1 \bar{w}_1  \\
  \mathfrak{m}_2(w_1,w_2) &=s_1 e, \quad \mathfrak{m}_2(w_2,\bar{e}) = -s_1 \bar{w}_1 , \quad \mathfrak{m}_2(\bar{e},w_1) = s_1 \bar{w}_2, \quad \mathfrak{m}_3(w_1,w_2,\bar{e}) =-s_1 e 
\end{align*}
Hence, $\mathfrak{m}_1^{\mathfrak{b},s_1}(w_1) = -(s_1+t_1t_2) \bar{w}_2$ and $\mathfrak{m}_1^{\mathfrak{b},s_1}(w_2) = (s_1+t_1t_2) \bar{w}_1$. Thus, the deformed algebra over $k[t_1,t_2]$ is given by
\[  \boxed{k\{w_1,w_2\} / (w_1^2 + t_1 w_1 , w_2^2 + t_2w_2, w_1w_2 + t_1 w_2 + t_2 w_1 + t_1t_2, w_2w_1 ) } \]
This family maps to the component of $\mathbf{Alg}_3$ that contains  $\bullet \rightarrow \bullet$.

One can also try turning on $s_2$ in this example, but then there are infinitely many polygons to worry about.

\subsection*{n=3: \texorpdfstring{\protect\quadeye}{} $1^+2^-3^+3^-2^+1^-$}

This curve is immersed in a five-punctured sphere. The boundary components are $m_1, m_2$ and $s_1, s_2, s_3$, where $m_1$ and $m_2$ are monogons, and $s_1$ and $s_2$ are bigons, $s_3$ is a hexagon. The gradings are determined by $|s_1|= |w_1| + |w_2|$, $|s_2|=|w_2|+|w_3|$, $|s_3|= 6 - 2|w_1| - 2|w_2| - 2|w_3|$, and $|m_1|=|w_1|$, $|m_2|= |w_3|$.  
\begin{center}
  \begin{tikzpicture}[xscale=0.9, yscale=0.8]

    \tikzset{->-/.style={decoration={markings, mark=at position #1 with {\arrow{>}}},postaction={decorate}}}

    \node[] (t1) at (0,0) {};
    \node[] (t2) at (0,4) {};
    \node[] (t3) at (10,4) {};
    \node[] (t4) at (10,0) {};

    \node[] (L0left) at (0,0.5) {};
    \node[] (L00mid)  at (6,3) {};
    \node[] (L0mid)  at (6,1) {};
    \node[] (L0right) at (0,3.5) {};

    \node[] (L1left) at (10,3.5) {};
    \node[] (L11mid)  at (4,3) {};
    \node[] (L1mid)  at (4,1) {};
    \node[] (L1right) at (10,0.5) {};

    \node[] (x1) at (3.8,3.4) {\small $w_1$}; 

   \node[] (z1) at (4.3,2) {\small $w_2$}; 
  
    \node[] (y1) at (3.8,0.6) {\small $w_3$}; 

    \fill[black] (6.7,3.6) circle (2pt);    
    \fill[black] (8.2,3.59) circle (2pt);    

    \node[] (q) at (6.7,3.8) {\small $\bar{e}$};
    \node[] (e) at (8.2,3.78){\small $e$};
    
    \draw[gray!90, thick, dashed] (t2.center) -- (t3.center);
    \draw[gray!90, thick, dashed] (t4.center) -- (t1.center);
    \draw[gray!90, thick] (t1.center) -- (t2.center);
    \draw[gray!90, thick] (t3.center) -- (t4.center); 

    \draw[thick, dashed] (5,1.2) circle (0.3cm);
    \node (s2) at (5,1.2) {$s_2$};

    \node (s1) at (5,4.2) {$m_1$};
    \node (s3) at (5,-0.4) {$m_2$};

    \draw[thick, dashed] (5,2.8) circle (0.3cm);
    \node (s4) at (5,2.8) {$s_1$};

    \draw[thick, dashed] (1,2) circle (0.3cm);
    \node (s5) at (1,2) {$s_3$};
  
    \draw[thick] (L0left.center)  .. controls +(0:1) and +(270:1) .. (L0mid.center);
    \draw[thick, ->-=.3] (L0right.center)  .. controls +(0:1) and +(90:1) .. (L00mid.center);

    \draw[thick] (L1left.center)  .. controls +(180:1) and +(90:1) .. (L11mid.center);
    \draw[thick] (L1mid.center)  .. controls +(270:1) and +(180:1) .. (L1right.center);

    \draw[thick] (L0mid.center)  .. controls +(90:1) and +(270:1) .. (L11mid.center);
    \draw[thick] (L00mid.center)  .. controls +(270:1) and +(90:1) .. (L1mid.center);

  \end{tikzpicture}
\end{center}
\begin{multicols}{3}\small
\begin{itemize}[label={}]
\item $\mathfrak{m}_3(w_{1},w_{1},\bar{w}_{1})= - w_{1}$
\item $\mathfrak{m}_3(w_{1},w_{2},\bar{w}_{2})= - w_{1}$
\item $\mathfrak{m}_3(w_{1},w_{3},\bar{w}_{3})= - w_{1}$
\item $\mathfrak{m}_3(w_{1},\bar{w}_{1},w_{2})= - w_{2}$
\item $\mathfrak{m}_3(w_{1},\bar{w}_{1},w_{3})= - w_{3}$
\item $\mathfrak{m}_3(w_{1},\bar{w}_{1},\bar{w}_{0})=\bar{w}_{0}$
\item $\mathfrak{m}_3(w_{1},\bar{w}_{1},\bar{w}_{1})=\bar{w}_{1}$
\item $\mathfrak{m}_3(w_{1},\bar{w}_{1},\bar{w}_{2})=\bar{w}_{2}$
\item $\mathfrak{m}_3(w_{1},\bar{w}_{1},\bar{w}_{3})=\bar{w}_{3}$
\item $\mathfrak{m}_3(w_{1},\bar{w}_{2},w_{2})=w_{1}$
\item $\mathfrak{m}_3(w_{1},\bar{w}_{3},w_{3})=w_{1}$
\item $\mathfrak{m}_3(w_{2},w_{1},\bar{w}_{1})= - w_{2}$
\item $\mathfrak{m}_3(w_{2},\bar{w}_{2},\bar{w}_{0})=\bar{w}_{0}$
\item $\mathfrak{m}_3(w_{2},\bar{w}_{2},\bar{w}_{1})=\bar{w}_{1}$
\item $\mathfrak{m}_3(w_{3},w_{1},\bar{w}_{1})= - w_{3}$
\item $\mathfrak{m}_3(w_{3},w_{3},\bar{w}_{3})= - w_{3}$
\item $\mathfrak{m}_3(w_{3},\bar{w}_{2},w_{2})=w_{3}$
\item $\mathfrak{m}_3(w_{3},\bar{w}_{3},w_{2})= - w_{2}$
\item $\mathfrak{m}_3(w_{3},\bar{w}_{3},\bar{w}_{0})=\bar{w}_{0}$
\item $\mathfrak{m}_3(w_{3},\bar{w}_{3},\bar{w}_{1})=\bar{w}_{1}$
\item $\mathfrak{m}_3(w_{3},\bar{w}_{3},\bar{w}_{3})=\bar{w}_{3}$
\item $\mathfrak{m}_3(\bar{w}_{1},w_{1},\bar{w}_{0})= - \bar{w}_{0}$
\item $\mathfrak{m}_3(\bar{w}_{1},w_{1},\bar{w}_{1})= - \bar{w}_{1}$
\item $\mathfrak{m}_3(\bar{w}_{2},w_{1},\bar{w}_{1})= - \bar{w}_{2}$
\item $\mathfrak{m}_3(\bar{w}_{2},w_{2},w_{2})=w_{2}$
\item $\mathfrak{m}_3(\bar{w}_{2},w_{2},w_{3})=w_{3}$
\item $\mathfrak{m}_3(\bar{w}_{2},w_{2},\bar{w}_{0})= - \bar{w}_{0}$
\item $\mathfrak{m}_3(\bar{w}_{2},w_{2},\bar{w}_{1})= - \bar{w}_{1}$
\item $\mathfrak{m}_3(\bar{w}_{2},w_{2},\bar{w}_{2})= - \bar{w}_{2}$
\item $\mathfrak{m}_3(\bar{w}_{2},w_{2},\bar{w}_{3})= - \bar{w}_{3}$
\item $\mathfrak{m}_3(\bar{w}_{2},w_{3},\bar{w}_{3})= - \bar{w}_{2}$
\item $\mathfrak{m}_3(\bar{w}_{2},\bar{w}_{2},w_{2})=\bar{w}_{2}$
\item $\mathfrak{m}_3(\bar{w}_{2},\bar{w}_{3},w_{3})=\bar{w}_{2}$
\item $\mathfrak{m}_3(\bar{w}_{3},w_{1},\bar{w}_{1})= - \bar{w}_{3}$
\item $\mathfrak{m}_3(\bar{w}_{3},w_{3},w_{2})=w_{2}$
\item $\mathfrak{m}_3(\bar{w}_{3},w_{3},\bar{w}_{0})= - \bar{w}_{0}$
\item $\mathfrak{m}_3(\bar{w}_{3},w_{3},\bar{w}_{1})= - \bar{w}_{1}$
\item $\mathfrak{m}_3(\bar{w}_{3},w_{3},\bar{w}_{3})= - \bar{w}_{3}$
\item $\mathfrak{m}_3(\bar{w}_{3},\bar{w}_{2},w_{2})=\bar{w}_{3}$
\end{itemize}
\end{multicols}
This gives $\mathfrak{m}_1^{\mathfrak{b}} = 0$ for $\mathfrak{b} = t_1 \bar{w}_1 + t_2 \bar{w}_2 + t_3 \bar{w}_3$. We let $(x_1,x_2,x_3) = (w_1,w_2,w_3)$. The deformed algebra over $k[t_1,t_2,t_3]$ is
\[\boxed{k[x_1,x_2,x_3] / (x_1^2 + t_1x_1, x_2^2 - t_2 x_2 , x_3^2 + t_3x_3, x_1 x_2 + t_1 x_2, x_1 x_3 + t_1x_3, x_2 x_3 - t_2 x_3)}\]
We give the isomorphism types of the algebras occurring in this family in the following table.  \[
\begin{array}{c|c}
  \hline
\text{All } t_i \neq 0 & k \times k \times k \times k \\ 
\text{Exactly one } t_i = 0 & k \times k \times k[x]/(x^2) \\ 
t_1 = t_3 = 0,\ t_2 \neq 0  & k[x]/(x^2) \times k[y]/(y^2) \\ 
t_1 = t_2 = 0 \text{ or } t_2 = t_3 = 0 & k \times k[x,y]/(x,y)^2 \\ 
\text{All } t_i = 0 & k[x,y,z]/(x,y,z)^2 \\ 
\hline
\end{array}
\]
\subsection*{n=3 : \texorpdfstring{\protect\stabsym}{} $1^+2^-3^-3^+2^+1^-$}

This curve is immersed in a five-punctured sphere. The boundary components are $m_1, m_2$ and $s_1, s_2, s_3$, where $m_1$ and $m_2$ are monogons, and $s_1$ is a bigon, $s_2$ is a triangle, and $s_3$ is a pentagon.  The gradings are determined by $|s_1|=|w_1|+|w_2|$, $|s_2|= 2 + |w_2| - 2|w_3|$, $|s_3| = 4 - 2|w_1| - 2|w_2| + |w_3|$, and $|m_1| = |w_1|$, $|m_2|=|w_3|$.  
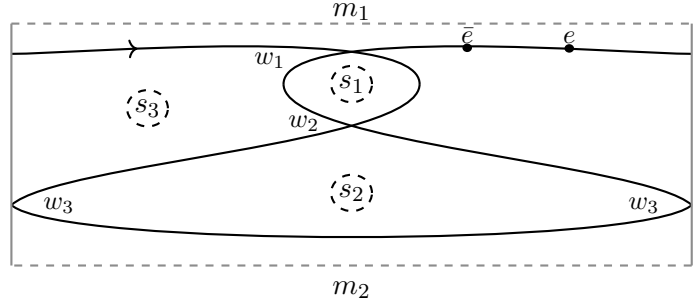
\begin{center}
  \begin{tikzpicture}[xscale=0.9, yscale=0.8]

    \tikzset{->-/.style={decoration={markings, mark=at position #1 with {\arrow{>}}},postaction={decorate}}}

    \node[] (t1) at (0,0) {};
    \node[] (t2) at (0,4) {};
    \node[] (t3) at (10,4) {};
    \node[] (t4) at (10,0) {};

    \node[] (L0left) at (0,1) {};
    \node[] (L1left) at (0,2) {};
    \node[] (L2left) at (0,3.5) {};

    \node[] (L0right) at (10,1) {};
    \node[] (L1right) at (10,2) {};
    \node[] (L2right) at (10,3.5) {};

   \draw[thick] (L0left.center)  .. controls +(45:1) and +(270:1) ..  (6,3) ;
   \draw[thick, ->-=.3] (L2left.center)  .. controls +(0:1) and +(90:1) .. (6,3);

   \draw[thick] (L0left.center)  .. controls +(315:1) and +(225:1) .. (L0right.center);
    
    \draw[thick] (L0right.center)  .. controls +(135:1) and +(270:1) .. (4,3);
    \draw[thick] (L2right.center)  .. controls +(180:1) and +(90:1) .. (4,3);

    \node[] (x1) at (3.8,3.4) {\small $w_1$}; 
    \node[] (z1) at (4.3,2.35) {\small $w_2$}; 
    \node[] (y1) at (0.7,1) {\small $w_3$}; 
    \node[] (w1) at (9.3,1) {\small $w_3$}; 

    \fill[black] (6.7,3.6) circle (2pt);    
    \fill[black] (8.2,3.59) circle (2pt);    

    \node[] (q) at (6.7,3.8) {\small $\bar{e}$};
    \node[] (e) at (8.2,3.78){\small $e$};
    
    \draw[gray!90, thick, dashed] (t2.center) -- (t3.center);
    \draw[gray!90, thick, dashed] (t4.center) -- (t1.center);
    \draw[gray!90, thick] (t1.center) -- (t2.center);
    \draw[gray!90, thick] (t3.center) -- (t4.center); 

    \draw[thick, dashed] (5,1.2) circle (0.3cm);
    \node (s2) at (5,1.2) {$s_2$};

    \node (s1) at (5,4.2) {$m_1$};
    \node (s3) at (5,-0.4) {$m_2$};

    \draw[thick, dashed] (5,3) circle (0.3cm);
    \node (s4) at (5,3) {$s_1$};

    \draw[thick, dashed] (2,2.6) circle (0.3cm);
    \node (s5) at (2,2.6) {$s_3$};    
  \end{tikzpicture}
\end{center}
\begin{multicols}{3}\small
\begin{itemize}[label={}]
\item $\mathfrak{m}_3(w_{1},w_{1},\bar{w}_{1})= - w_{1}$
\item $\mathfrak{m}_3(w_{1},w_{2},\bar{w}_{2})= - w_{1}$
\item $\mathfrak{m}_3(w_{1},w_{3},\bar{w}_{3})= - w_{1}$
\item $\mathfrak{m}_3(w_{1},\bar{w}_{1},w_{2})= - w_{2}$
\item $\mathfrak{m}_3(w_{1},\bar{w}_{1},w_{3})= - w_{3}$
\item $\mathfrak{m}_3(w_{1},\bar{w}_{1},\bar{w}_{0})=\bar{w}_{0}$
\item $\mathfrak{m}_3(w_{1},\bar{w}_{1},\bar{w}_{1})=\bar{w}_{1}$
\item $\mathfrak{m}_3(w_{1},\bar{w}_{1},\bar{w}_{2})=\bar{w}_{2}$
\item $\mathfrak{m}_3(w_{1},\bar{w}_{1},\bar{w}_{3})=\bar{w}_{3}$
\item $\mathfrak{m}_3(w_{1},\bar{w}_{2},w_{2})=w_{1}$
\item $\mathfrak{m}_3(w_{1},\bar{w}_{3},w_{3})=w_{1}$
\item $\mathfrak{m}_3(w_{2},w_{1},\bar{w}_{1})= - w_{2}$
\item $\mathfrak{m}_3(w_{2},\bar{w}_{2},\bar{w}_{0})=\bar{w}_{0}$
\item $\mathfrak{m}_3(w_{2},\bar{w}_{2},\bar{w}_{1})=\bar{w}_{1}$
\item $\mathfrak{m}_3(w_{3},w_{1},\bar{w}_{1})= - w_{3}$
\item $\mathfrak{m}_3(w_{3},\bar{w}_{2},w_{2})=w_{3}$
\item $\mathfrak{m}_3(w_{3},\bar{w}_{3},w_{2})= - w_{2}$
\item $\mathfrak{m}_3(w_{3},\bar{w}_{3},\bar{w}_{0})=\bar{w}_{0}$
\item $\mathfrak{m}_3(w_{3},\bar{w}_{3},\bar{w}_{1})=\bar{w}_{1}$
\item $\mathfrak{m}_3(\bar{w}_{1},w_{1},\bar{w}_{0})= - \bar{w}_{0}$
\item $\mathfrak{m}_3(\bar{w}_{1},w_{1},\bar{w}_{1})= - \bar{w}_{1}$
\item $\mathfrak{m}_3(\bar{w}_{2},w_{1},\bar{w}_{1})= - \bar{w}_{2}$
\item $\mathfrak{m}_3(\bar{w}_{2},w_{2},w_{2})=w_{2}$
\item $\mathfrak{m}_3(\bar{w}_{2},w_{2},w_{3})=w_{3}$
\item $\mathfrak{m}_3(\bar{w}_{2},w_{2},\bar{w}_{0})= - \bar{w}_{0}$
\item $\mathfrak{m}_3(\bar{w}_{2},w_{2},\bar{w}_{1})= - \bar{w}_{1}$
\item $\mathfrak{m}_3(\bar{w}_{2},w_{2},\bar{w}_{2})= - \bar{w}_{2}$
\item $\mathfrak{m}_3(\bar{w}_{2},w_{2},\bar{w}_{3})= - \bar{w}_{3}$
\item $\mathfrak{m}_3(\bar{w}_{2},w_{3},\bar{w}_{3})= - \bar{w}_{2}$
\item $\mathfrak{m}_3(\bar{w}_{2},\bar{w}_{2},w_{2})=\bar{w}_{2}$
\item $\mathfrak{m}_3(\bar{w}_{2},\bar{w}_{3},w_{3})=\bar{w}_{2}$
\item $\mathfrak{m}_3(\bar{w}_{3},w_{1},\bar{w}_{1})= - \bar{w}_{3}$
\item $\mathfrak{m}_3(\bar{w}_{3},w_{3},w_{2})=w_{2}$
\item $\mathfrak{m}_3(\bar{w}_{3},w_{3},w_{3})=w_{3}$
\item $\mathfrak{m}_3(\bar{w}_{3},w_{3},\bar{w}_{0})= - \bar{w}_{0}$
\item $\mathfrak{m}_3(\bar{w}_{3},w_{3},\bar{w}_{1})= - \bar{w}_{1}$
\item $\mathfrak{m}_3(\bar{w}_{3},w_{3},\bar{w}_{3})= - \bar{w}_{3}$
\item $\mathfrak{m}_3(\bar{w}_{3},\bar{w}_{2},w_{2})=\bar{w}_{3}$
\item $\mathfrak{m}_3(\bar{w}_{3},\bar{w}_{3},w_{3})=\bar{w}_{3}$
\end{itemize}

\end{multicols}
Turning on the $s_2$ deformation, we get
\[ \mathfrak{m}_2(w_3,w_3) = s_2w_2 , \mathfrak{m}_2(w_3,\bar{w}_2) = -s_2\bar{w}_3, \mathfrak{m}_2(\bar{w}_2, w_3) = s_2\bar{w}_3 \]
This gives $\mathfrak{m}_1^{\mathfrak{b}} = 0$ for $\mathfrak{b} = t_1 \bar{w}_1 + t_2 \bar{w}_2 + t_3 \bar{w}_3$. We let $(x_1,x_2,x_3) = (w_1,w_2,w_3)$. The deformed algebra over $k[t_1,t_2,t_3, s_2]$ is
\[\boxed{k[x_1,x_2,x_3] / (x_1^2 + t_1x_1, x_2^2 - t_2 x_2 , x_3^2 - t_3x_3 + s_2x_2, x_1 x_2 + t_1 x_2, x_1 x_3 + t_1x_3, x_2 x_3 - t_2 x_3)}\]
Let $\Delta = t_3^2 - 4s_2t_2$. The possible isomorphism types of algebras occurring in this family are given in the following table.
\[
\begin{array}{c|c}
  \hline
t_1 \neq 0, t_2 \neq 0, \Delta \neq 0 & k \times k \times k \times k \\ 
\text{Exactly one of } \{ t_1=0, t_2=0, \Delta=0\} & k \times k \times k[x]/(x^2) \\ 
t_1 = \Delta = 0,  t_2 \neq 0 & k[x]/(x^2) \times k[y]/(y^2) \\ 
t_2 = t_3 = 0, s_2 \neq 0, t_1 \neq 0 & k \times k[x]/(x^3) \\ 
t_1 = t_2 = 0, t_3 \neq 0 \text{ or } t_2 = t_3 = s_2 = 0, t_1 \neq 0  & k \times k[x,y]/(x,y)^2\\ 
t_1 = t_2 = t_3 = 0, s_2 \neq 0 & k[x,y]/(x^2, y^3, xy) \\ 
t_1 = t_2 = t_3 = s_2 = 0 & k[x,y,z]/(x,y,z)^2  \\
\hline
\end{array}
\]
\subsection*{n=3 : \texorpdfstring{\protect\doubleear}{} $1^+2^+2^-3^+3^-1^-$}
This curve is immersed in a five-punctured sphere. The boundary components are $m_1, m_2, m_3, s_1, s_2$, where $m_1, m_2$ and $m_3$ are monogons, and $s_1$ is a rectangle, and $s_2$ is a pentagon. The gradings are determined by $|s_1| =2 - 2|w_1| + |w_2| + |w_3|$, $|s_2|= 4 + |w_1| - |w_2|-|w_3|$ and $|m_1| = |w_1|, |m_2|= |w_2|, |m_3| = |w_3|$.  If we choose $|w_i|=0$, we have the following triple products.

\begin{center}
\begin{tikzpicture}[scale=1.5]
    \tikzset{->-/.style={decoration={markings, mark=at position #1 with {\arrow{>}}},postaction={decorate}}}

  \draw[thick, line cap=round, line join=round, use Hobby shortcut]
    ([closed=true]0, 1.4) .. 
    (0.9, 0.5)[->-=0.4] ..            
    (0.6, -0.3) ..          
    (0.3,  -0.1) ..           
    (0.2, 0.4) ..           
    (0.4, 0.4) ..           
    (0.3, -0.1) ..            
    (0.0, -0.2) ..           
    (-0.3, -0.1) ..           
    (-0.4, 0.4) ..          
    (-0.2, 0.4) ..          
    (-0.3, -0.1) ..          
    (-0.6, -0.3) ..         
    (-0.9, 0.5) ..           
    (0, 1.4) ..              
    (1, 1.1) ..            
    (1.2, 0.8) ..            
    (1, -1.1) ..           
    (0, -1.4) ..             
    (-1, -1.1) ..          
    (-1.2, 0.8) ..           
    (-1, 1.1);             

    \draw[thick, dashed] (0,0.7) circle (0.2cm);
    \node (s2) at (0,0.7) {\scriptsize $s_2$};

    \draw[thick, dashed] (0,-0.7) circle (0.2cm);
    \node (s1) at (0,-0.7) {\scriptsize $s_1$};

    \node[] (x1) at (0.35,1.34) {\scriptsize $w_1$}; 
    \node[] (z1) at (0.5,-0.1) {\scriptsize $w_2$}; 
    \node[] (y1) at (-0.5,-0.1) {\scriptsize $w_3$}; 
  
    \fill[black] (0.9,1.2) circle (1pt);    
    \fill[black] (1.15,0.9) circle (1pt);    

    \node[] (q) at (1.05,1.25) {\small $\bar{e}$};
    \node[] (e) at (1.25,1){\small $e$};

    \node[] (m1) at (2, 0.2) {\scriptsize $m_1$};
    \draw[thick, dashed] (2,0.2) circle (0.2cm);

    \node[] (m3) at (-0.3, 0.2) {\scriptsize $m_3$};
    \node[] (m2) at (0.3, 0.2) {\scriptsize $m_2$};
   
  \end{tikzpicture}
\end{center}

\begin{multicols}{3}\small
\begin{itemize}[label={}]
\item $\mathfrak{m}_3(w_{1},w_{1},\bar{w}_{1})= - w_{1}$
\item $\mathfrak{m}_3(w_{1},w_{2},\bar{w}_{2})= - w_{1}$
\item $\mathfrak{m}_3(w_{1},w_{3},\bar{w}_{3})= - w_{1}$
\item $\mathfrak{m}_3(w_{1},\bar{w}_{1},w_{2})= - w_{2}$
\item $\mathfrak{m}_3(w_{1},\bar{w}_{1},w_{3})= - w_{3}$
\item $\mathfrak{m}_3(w_{1},\bar{w}_{1},\bar{w}_{0})=\bar{w}_{0}$
\item $\mathfrak{m}_3(w_{1},\bar{w}_{1},\bar{w}_{1})=\bar{w}_{1}$
\item $\mathfrak{m}_3(w_{1},\bar{w}_{1},\bar{w}_{2})=\bar{w}_{2}$
\item $\mathfrak{m}_3(w_{1},\bar{w}_{1},\bar{w}_{3})=\bar{w}_{3}$
\item $\mathfrak{m}_3(w_{1},\bar{w}_{2},w_{2})=w_{1}$
\item $\mathfrak{m}_3(w_{1},\bar{w}_{3},w_{3})=w_{1}$
\item $\mathfrak{m}_3(w_{2},w_{1},\bar{w}_{1})= - w_{2}$
\item $\mathfrak{m}_3(w_{2},w_{2},\bar{w}_{2})= - w_{2}$
\item $\mathfrak{m}_3(w_{2},\bar{w}_{2},w_{3})= - w_{3}$
\item $\mathfrak{m}_3(w_{2},\bar{w}_{2},\bar{w}_{0})=\bar{w}_{0}$
\item $\mathfrak{m}_3(w_{2},\bar{w}_{2},\bar{w}_{1})=\bar{w}_{1}$
\item $\mathfrak{m}_3(w_{2},\bar{w}_{2},\bar{w}_{2})=\bar{w}_{2}$
\item $\mathfrak{m}_3(w_{2},\bar{w}_{2},\bar{w}_{3})=\bar{w}_{3}$
\item $\mathfrak{m}_3(w_{3},w_{1},\bar{w}_{1})= - w_{3}$
\item $\mathfrak{m}_3(w_{3},w_{2},\bar{w}_{2})= - w_{3}$
\item $\mathfrak{m}_3(w_{3},w_{3},\bar{w}_{3})= - w_{3}$
\item $\mathfrak{m}_3(w_{3},\bar{w}_{2},w_{2})=w_{3}$
\item $\mathfrak{m}_3(w_{3},\bar{w}_{3},\bar{w}_{0})=\bar{w}_{0}$
\item $\mathfrak{m}_3(w_{3},\bar{w}_{3},\bar{w}_{1})=\bar{w}_{1}$
\item $\mathfrak{m}_3(w_{3},\bar{w}_{3},\bar{w}_{3})=\bar{w}_{3}$
\item $\mathfrak{m}_3(\bar{w}_{1},w_{1},\bar{w}_{0})= - \bar{w}_{0}$
\item $\mathfrak{m}_3(\bar{w}_{1},w_{1},\bar{w}_{1})= - \bar{w}_{1}$
\item $\mathfrak{m}_3(\bar{w}_{2},w_{1},\bar{w}_{1})= - \bar{w}_{2}$
\item $\mathfrak{m}_3(\bar{w}_{2},w_{2},w_{3})=w_{3}$
\item $\mathfrak{m}_3(\bar{w}_{2},w_{2},\bar{w}_{0})= - \bar{w}_{0}$
\item $\mathfrak{m}_3(\bar{w}_{2},w_{2},\bar{w}_{1})= - \bar{w}_{1}$
\item $\mathfrak{m}_3(\bar{w}_{2},w_{2},\bar{w}_{2})= - \bar{w}_{2}$
\item $\mathfrak{m}_3(\bar{w}_{2},w_{2},\bar{w}_{3})= - \bar{w}_{3}$
\item $\mathfrak{m}_3(\bar{w}_{3},w_{1},\bar{w}_{1})= - \bar{w}_{3}$
\item $\mathfrak{m}_3(\bar{w}_{3},w_{2},\bar{w}_{2})= - \bar{w}_{3}$
\item $\mathfrak{m}_3(\bar{w}_{3},w_{3},\bar{w}_{0})= - \bar{w}_{0}$
\item $\mathfrak{m}_3(\bar{w}_{3},w_{3},\bar{w}_{1})= - \bar{w}_{1}$
\item $\mathfrak{m}_3(\bar{w}_{3},w_{3},\bar{w}_{3})= - \bar{w}_{3}$
\item $\mathfrak{m}_3(\bar{w}_{3},\bar{w}_{2},w_{2})=\bar{w}_{3}$
\end{itemize}
\end{multicols}
Turning on the $s_1$-deformation associated with the rectangle, we additionally obtain the following contributions:
\begin{multicols}{2}\small
\begin{itemize}[label={}]
\item $\mathfrak{m}_3(\bar{w}_3, \bar{w}_2, w_1) = s_1 \bar{w}_1$
\item $\mathfrak{m}_3(\bar{w}_2, w_1, w_1) = s_1 w_3$
\item $\mathfrak{m}_3(w_1, w_1, \bar{w}_3) = s_1 w_2$
\item $\mathfrak{m}_3(w_1, \bar{w}_3, \bar{w}_2) = - s_1 \bar{w}_1$  
\item $\mathfrak{m}_4(w_1, \bar{w}_3, \bar{w}_2, w_1) = s_1 w_0$
\item $\mathfrak{m}_4(\bar{w}_3, \bar{w}_2, w_1, \bar{w}_0) = -s_1 \bar{w}_1$
\item $\mathfrak{m}_4(\bar{w}_2, w_1, \bar{w}_0,w_1) = -s_1 w_3$
\item $\mathfrak{m}_4(w_1, \bar{w}_0, w_1, \bar{w}_3) = -s_1 w_2$
\item $\mathfrak{m}_4(\bar{w}_0, w_1, \bar{w}_3, \bar{w}_2) = s_1 \bar{w}_1$
\item $\mathfrak{m}_5(w_1, \bar{w}_3, \bar{w}_2,w_1,\bar{w}_0) = -s_1 w_0 $
\end{itemize}
\end{multicols}

Note that the winding number around $s_2$ is 4, so we have no more deformations. Letting $\mathfrak{b}= t_1 \bar{w}_1 + t_2 \bar{w}_2 + t_3 \bar{w}_3$, we get that $\mathfrak{m}_1^{\mathfrak{b}} =0$ and writing $(x_1,x_2,x_3) = (w_1,w_2,w_3)$, the deformed algebra over $k[t_1,t_2,t_3,s_1]$ is given by
\[\boxed{k[x_1,x_2,x_3] \Big/ \Big( \begin{aligned} &x_1^2 + t_1x_1 - s_1 t_3 x_2 - s_1 t_2 x_3 - s_1t_2t_3,\; x_2^2 + t_2 x_2,\; x_3^2 + t_3x_3,\\ &x_1 x_2 + t_1 x_2,\; x_1 x_3 + t_1x_3,\; x_2 x_3  \end{aligned} \Big)} \]
Let $\Delta = t_1^2 +  4s_1t_2t_3$. The possible isomorphism types of algebras occurring in this family are given in the following table.
\[
\begin{array}{c|c}
  \hline
t_2 \neq 0, t_3 \neq 0, \Delta \neq 0 & k \times k \times k \times k \\ 
\text{Exactly one of } \{ t_2=0, t_3=0, \Delta=0\} & k \times k \times k[x]/(x^2) \\ 
   t_2,s_1 \neq 0, t_1=t_3=0 \text{ or } t_3,s_1\neq 0, t_1=t_2=0  &  k \times k[x]/(x^3) \\ 
t_1 = t_2 = s_1=0, t_3 \neq 0 \text{ or } t_1 = t_3 = s_1 = 0, t_2 \neq 0  \text{ or } t_2=t_3=0, t_1 \neq 0& k \times k[x,y]/(x,y)^2\\ 
t_1 = t_2 = t_3 = 0 & k[x,y,z]/(x,y,z)^2  \\
\hline
\end{array}
\]
\subsection*{n=3 : \texorpdfstring{\protect\toric}{} $1^+2^+3^+3^-2^-1^-$}
This curve is immersed in a five-punctured sphere. The boundary components are $m_1, m_2$ and $s_1, s_2, s_3$, where $m_1$ and $m_2$ are monogons, and $s_1$ and $s_2$ are triangles, $s_3$ is a rectangle. The gradings are related by $|s_1| = 2 + |w_2| - 2|w_3|$, $|s_2| = 2 + |w_1| -2|w_2| + |w_3|$,  $|s_3|= 2 + |w_2| - 2|w_1|$, $|m_1|= |w_1|$ and $|m_2|= |w_3|$. 
\begin{center}
  \begin{tikzpicture}[xscale=0.9, yscale=0.8]
    \tikzset{->-/.style={decoration={markings, mark=at position #1 with {\arrow{>}}},postaction={decorate}}}

    \node[] (t1) at (0,0) {};
    \node[] (t2) at (0,4) {};
    \node[] (t3) at (10,4) {};
    \node[] (t4) at (10,0) {};

    \node[] (L0left) at (0,1.4) {};
    \node[] (L0right) at (10,0.4) {};

    \node[] (L1left) at (0,2.4) {};
    \node[] (L1right) at (10,1.4) {};

    \node[] (L2left) at (0,0.4) {};
    \node[] (L2right) at (10,3.4) {};

    \node[] (L3left) at (0,3.4) {};
    \node[] (L3right) at (10,2.4) {};

    \draw[gray!90, thick, dashed] (t2.center) -- (t3.center);
    \draw[gray!90, thick, dashed] (t4.center) -- (t1.center);
    \draw[gray!90, thick] (t1.center) -- (t2.center);
    \draw[gray!90, thick] (t3.center) -- (t4.center); 

    \draw[thick] (L0left.center) .. controls +(0:1) and +(180:1) .. (L0right.center);

    \draw[thick] (L2left.center) .. controls +(0:1) and +(180:1) .. (L2right.center);
    
    \draw[thick] (L3left.center) .. controls +(0:1) and +(180:1) .. (L3right.center);
  
    \draw[thick, ->-=.5] (L1left.center) .. controls +(0:1) and +(180:1) .. (L1right.center);

    \fill[black] (8.3,2.98) circle (2pt);    
    \fill[black] (9.2,3.24) circle (2pt);    

    \node[] (q) at (8.3,3.2) {\small $\bar{e}$};
    \node[] (e) at (9.2,3.45) {\small $e$};

   \node[] (z1) at (6.5,2.57) {\small $w_1$}; 
    \node[] (x1) at (1.75,1.05) {\small $w_3$}; 
    \node[] (y1) at (4.2,1.85) {\small $w_2$}; 

    \draw[thick, dashed] (7,1.2) circle (0.3cm);
    \node (s1) at (7,1.2) {$s_1$};
    \draw[thick, dashed] (2,1.8) circle (0.3cm);
    \node (s2) at (2,1.8) {$s_2$};
    \draw[thick, dashed] (4,2.5) circle (0.3cm);
    \node (s3) at (4,2.5) {$s_3$};

    \node (m1) at (5,4.2) {$m_1$};
    \node (m2) at (5,-0.4) {$m_2$};

  \end{tikzpicture}
\end{center}
\begin{multicols}{3}\small
\begin{itemize}[label={}]
\item $\mathfrak{m}_3(w_{1},w_{1},\bar{w}_{1})= - w_{1}$
\item $\mathfrak{m}_3(w_{1},w_{2},\bar{w}_{2})= - w_{1}$
\item $\mathfrak{m}_3(w_{1},w_{3},\bar{w}_{3})= - w_{1}$
\item $\mathfrak{m}_3(w_{1},\bar{w}_{1},w_{2})= - w_{2}$
\item $\mathfrak{m}_3(w_{1},\bar{w}_{1},w_{3})= - w_{3}$
\item $\mathfrak{m}_3(w_{1},\bar{w}_{1},\bar{w}_{0})=\bar{w}_{0}$
\item $\mathfrak{m}_3(w_{1},\bar{w}_{1},\bar{w}_{1})=\bar{w}_{1}$
\item $\mathfrak{m}_3(w_{1},\bar{w}_{1},\bar{w}_{2})=\bar{w}_{2}$
\item $\mathfrak{m}_3(w_{1},\bar{w}_{1},\bar{w}_{3})=\bar{w}_{3}$
\item $\mathfrak{m}_3(w_{1},\bar{w}_{2},w_{2})=w_{1}$
\item $\mathfrak{m}_3(w_{1},\bar{w}_{3},w_{3})=w_{1}$
\item $\mathfrak{m}_3(w_{2},w_{1},\bar{w}_{1})= - w_{2}$
\item $\mathfrak{m}_3(w_{2},w_{2},\bar{w}_{2})= - w_{2}$
\item $\mathfrak{m}_3(w_{2},w_{3},\bar{w}_{3})= - w_{2}$
\item $\mathfrak{m}_3(w_{2},\bar{w}_{2},w_{3})= - w_{3}$
\item $\mathfrak{m}_3(w_{2},\bar{w}_{2},\bar{w}_{0})=\bar{w}_{0}$
\item $\mathfrak{m}_3(w_{2},\bar{w}_{2},\bar{w}_{1})=\bar{w}_{1}$
\item $\mathfrak{m}_3(w_{2},\bar{w}_{2},\bar{w}_{2})=\bar{w}_{2}$
\item $\mathfrak{m}_3(w_{2},\bar{w}_{2},\bar{w}_{3})=\bar{w}_{3}$
\item $\mathfrak{m}_3(w_{2},\bar{w}_{3},w_{3})=w_{2}$
\item $\mathfrak{m}_3(w_{3},w_{1},\bar{w}_{1})= - w_{3}$
\item $\mathfrak{m}_3(w_{3},w_{2},\bar{w}_{2})= - w_{3}$
\item $\mathfrak{m}_3(w_{3},w_{3},\bar{w}_{3})= - w_{3}$
\item $\mathfrak{m}_3(w_{3},\bar{w}_{3},\bar{w}_{0})=\bar{w}_{0}$
\item $\mathfrak{m}_3(w_{3},\bar{w}_{3},\bar{w}_{1})=\bar{w}_{1}$
\item $\mathfrak{m}_3(w_{3},\bar{w}_{3},\bar{w}_{2})=\bar{w}_{2}$
\item $\mathfrak{m}_3(w_{3},\bar{w}_{3},\bar{w}_{3})=\bar{w}_{3}$
\item $\mathfrak{m}_3(\bar{w}_{1},w_{1},\bar{w}_{0})= - \bar{w}_{0}$
\item $\mathfrak{m}_3(\bar{w}_{1},w_{1},\bar{w}_{1})= - \bar{w}_{1}$
\item $\mathfrak{m}_3(\bar{w}_{2},w_{1},\bar{w}_{1})= - \bar{w}_{2}$
\item $\mathfrak{m}_3(\bar{w}_{2},w_{2},\bar{w}_{0})= - \bar{w}_{0}$
\item $\mathfrak{m}_3(\bar{w}_{2},w_{2},\bar{w}_{1})= - \bar{w}_{1}$
\item $\mathfrak{m}_3(\bar{w}_{2},w_{2},\bar{w}_{2})= - \bar{w}_{2}$
\item $\mathfrak{m}_3(\bar{w}_{3},w_{1},\bar{w}_{1})= - \bar{w}_{3}$
\item $\mathfrak{m}_3(\bar{w}_{3},w_{2},\bar{w}_{2})= - \bar{w}_{3}$
\item $\mathfrak{m}_3(\bar{w}_{3},w_{3},\bar{w}_{0})= - \bar{w}_{0}$
\item $\mathfrak{m}_3(\bar{w}_{3},w_{3},\bar{w}_{1})= - \bar{w}_{1}$
\item $\mathfrak{m}_3(\bar{w}_{3},w_{3},\bar{w}_{2})= - \bar{w}_{2}$
\item $\mathfrak{m}_3(\bar{w}_{3},w_{3},\bar{w}_{3})= - \bar{w}_{3}$
\end{itemize}
\end{multicols}  
Turning on the $s_1$-deformation, we get the following contributions
\begin{multicols}{3}\small
\begin{itemize}[label={}]
\item  $\mathfrak{m}_2 (w_3,w_3) = s_1w_2$
\item  $\mathfrak{m}_2 (w_3,\bar{w}_2) = -s_1\bar{w}_3$
\item  $\mathfrak{m}_2 (\bar{w}_2,w_3) = s_1\bar{w}_3$
\end{itemize}
\end{multicols}
In addition, turning on the $s_2$-deformation, we get the following contributions:
\begin{multicols}{3}\small
\begin{itemize}[label={}]
\item $\mathfrak{m}_2 (w_2,w_3) = s_1s_2w_1 $
\item $\mathfrak{m}_2 (w_3,\bar{w}_1) = -s_1s_2\bar{w}_2$
\item    $\mathfrak{m}_2 (\bar{w}_1,w_2) = s_1s_2\bar{w}_3$ 
\item
\item    $\mathfrak{m}_2 (w_3,w_2) = s_1s_2w_1$ 
\item    $\mathfrak{m}_2 (w_2,\bar{w}_1) = -s_1s_2\bar{w}_3$
\item    $\mathfrak{m}_2 (\bar{w}_1, w_3) = s_1s_2\bar{w}_2$
\item
\item    $\mathfrak{m}_3(w_2,\bar{w}_3, w_2) = s_2 w_1$
\item    $\mathfrak{m}_3(\bar{w}_3, w_2, \bar{w}_1) = -s_2 \bar{w}_2$
\item    $\mathfrak{m}_3(w_2,\bar{w}_1, w_2) = -s_2 w_3$
\item    $\mathfrak{m}_3(\bar{w}_1,w_2, \bar{w}_3) = s_2 \bar{w}_2$  
\end{itemize}
\end{multicols}
Finally, we also turn on the $s_3$-deformation. 
\begin{multicols}{2}\small
\begin{itemize}[label={}]
\item  $\mathfrak{m}_2 (w_1, w_1) = s_3 w_2$ 
\item  $\mathfrak{m}_2 (w_1, \bar{w}_2) = - s_3 \bar{w}_1$ 
\item  $\mathfrak{m}_2 (\bar{w}_2, w_1) = s_3 \bar{w}_1$ 
\item  $\mathfrak{m}_2 (w_1,w_2)  = s_2s_3 w_3$ 
\item  $\mathfrak{m}_2 (w_2,\bar{w}_3)  = -s_2s_3 \bar{w}_1$ 
\item  $\mathfrak{m}_2 (\bar{w}_3,w_1)  = s_2s_3 \bar{w}_2$ 
\item  $\mathfrak{m}_2 (w_2,w_1)  = s_2s_3 w_3$ 
\item  $\mathfrak{m}_2 (w_1,\bar{w}_3)  = -s_2s_3 \bar{w}_2$ 
\item  $\mathfrak{m}_2 (\bar{w}_3,w_2)  = s_2s_3 \bar{w}_1$
\item  $\mathfrak{m}_2 (w_1,w_3)  = s_1 s_2s_3 w_0$ 
\item  $\mathfrak{m}_2 (w_3,\bar{w}_0)  = -s_1 s_2s_3 \bar{w}_1$ 
\item  $\mathfrak{m}_2 (\bar{w}_0,w_1)  = s_1 s_2s_3 \bar{w}_3$ 
\item  $\mathfrak{m}_2 (w_3,w_1)  = s_1 s_2s_3 w_0$  
\item  $\mathfrak{m}_2 (w_1,\bar{w}_0)  = -s_1 s_2s_3 \bar{w}_3$ 
\item  $\mathfrak{m}_2 (\bar{w}_0,w_3)  = s_1s_2s_3 \bar{w}_1$ 
\item  $\mathfrak{m}_2 (w_2, w_2) = s_1s_2^2s_3 w_0$ 
\item  $\mathfrak{m}_2 (w_2, \bar{w}_0) = - s_1s_2^2s_3 \bar{w}_2$ 
\item  $\mathfrak{m}_2 (\bar{w}_0, w_2) = s_1s_2^2s_3 \bar{w}_2$
\item $\mathfrak{m}_3(w_2,\bar{w}_3,w_1) = s_2s_3 w_0 $
\item $\mathfrak{m}_3(\bar{w}_3,w_1, \bar{w}_0) = -s_2s_3 \bar{w}_2 $
\item $\mathfrak{m}_3(w_1,\bar{w}_0,w_2) = -s_2s_3 w_3 $
\item $\mathfrak{m}_3(\bar{w}_0,w_2, \bar{w}_3) = s_2s_3 \bar{w}_1 $
\item $\mathfrak{m}_3(w_1,\bar{w}_3,w_2) = s_2s_3 w_0 $
\item $\mathfrak{m}_3(\bar{w}_3,w_2, \bar{w}_0) = -s_2s_3 \bar{w}_1 $
\item $\mathfrak{m}_3(w_2,\bar{w}_0,w_1) = -s_2s_3 w_3 $
\item $\mathfrak{m}_3(\bar{w}_0,w_1, \bar{w}_3) = s_2s_3 \bar{w}_2 $
\item $\mathfrak{m}_3(w_1, \bar{w}_2,w_1) = s_3 w_0 $
\item $\mathfrak{m}_3(\bar{w}_2,w_1,\bar{w}_0) = -s_3 \bar{w}_1 $
\item $\mathfrak{m}_3(w_1, \bar{w}_0,w_1) = -s_3 w_2 $
\item $\mathfrak{m}_3(\bar{w}_0, w_1, \bar{w}_2) = s_3 \bar{w}_1 $
\item $\mathfrak{m}_3(w_3 ,w_1 ,\bar{w}_0) = -s_1s_2s_3 w_0 $
\item $\mathfrak{m}_3(w_1 ,w_3 ,\bar{w}_0) = -s_1s_2s_3 w_0 $
\item $\mathfrak{m}_3(w_2 ,w_2 ,\bar{w}_0) = -s_1s_2^2 s_3 w_0 $
\item $\mathfrak{m}_4(w_2,\bar{w}_3,w_1,\bar{w}_0) = -s_2s_3 w_0$
\item $\mathfrak{m}_4(w_1,\bar{w}_3,w_2,\bar{w}_0) = -s_2s_3 w_0$
\item $\mathfrak{m}_4(w_1,\bar{w}_2,w_1,\bar{w}_0) = -s_3 w_0$
\end{itemize}
\end{multicols}
 The deformed algebra over $k[t_1,t_2,t_3,s_1,s_2,s_3]$ is
\[ \boxed{k[x_1,x_2,x_3] \Big/ \Big( \begin{aligned} &x_1^2 + t_1x_1 - s_3x_2 - s_3t_2 ,\; x_2^2 + t_2 x_2 - s_2t_3 x_1 + s_2 t_1 x_3 - s_1s_2^2s_3,\; x_3^2 + t_3x_3 - s_1 x_2,\\
                                                     & x_1x_2 + t_1 x_2 - s_2s_3 x_3 - s_2 s_3 t_3,\; x_2 x_3 + t_2 x_3 - s_1s_2x_1,\; x_1x_3 + t_1 x_3 - s_1s_2s_3 \end{aligned} \Big)} \]
 Over an algebraically closed field $k$, every isomorphism type (of which there are 9) of commutative, associative, unital $k$-algebras of dimension $4$ appears as a fiber of this family. The following table gives realizations of all commutative algebras:
\[
\begin{array}{c|c}
  \hline
  s_i=t_i=0 & k[x,y,z]/(x,y,z)^2 \\ 
  s_1=1,\; s_2=s_3=t_i=0 & k[x,y] /(x^3,xy,y^2) \\
  s_1=s_3=1,\; s_2=t_i=0 & k[x_1,x_3]/ (x_1^2-x_3^2,x_1x_3) \simeq k[x,y]/(x^2,y^2) \\
  s_1=s_2=1,\; s_3=t_i=0 & k[x]/(x^4) \\
  s_2=t_3=1,\; s_1=s_3=t_1=t_2=0 & k \times k[x]/(x^3) \\
  t_1=1,\;  s_1=s_2=s_3=t_2=t_3=0 & k \times k[x,y]/(x,y)^2 \\
  t_2=1,\; s_1=s_2=s_3=t_1=t_3=0 & k[x]/(x^2) \times k[y]/(y^2) \\
  t_1=t_3=1,\; s_1=s_2=s_3=t_2=0  & k \times k \times k[x]/(x^2) \\
  t_1=t_2=t_3=1,\; s_1=s_2=s_3=0 & k \times k \times k \times k \\ \hline                      
\end{array}                                               
\]

\begin{rmk} For $f = s_1x^4 + t_3x^3y + t_2x^2y^2 + t_1xy^3 - s_3y^4$ we can consider the complex
  \[ \mathcal{O}_{\mathbb{P}^1} (-4) \xlongrightarrow{f} \mathcal{O}_{\mathbb{P}^1} \]
  As in the cubic case, we can consider the ring $R_f = H^0 R\pi_*(\mathcal{O}_{\mathbb{P}^1}(-4) \to \mathcal{O}_{\mathbb{P}^1})$. A concrete computation of this ring appears in \cite{wood} and its presentation is given by  
\[ \boxed{k[x_1,x_2,x_3] \Big/ \Big( \begin{aligned} &x_1^2 + t_1x_1 - s_3x_2 - s_3t_2 ,\; x_2^2 + t_2 x_2 - t_3 x_1 + t_1 x_3 - s_1s_3,\; x_3^2 + t_3x_3 - s_1 x_2,\\
                                                     & x_1x_2 + t_1 x_2 - s_3 x_3 - s_3 t_3,\; x_2 x_3 + t_2 x_3 - s_1x_1,\; x_1x_3 + t_1 x_3 - s_1s_3 \end{aligned} \Big)} \]
which is precisely our presentation after setting $s_2=1$. However, apart from $k[x,y,z]/(x,y,z)^2$ corresponding to the case $f=0$, these rings are necessarily monogenic because for $f \neq 0$ they are given by global functions on the scheme cut out by $f$, so they give only \[ k[x]/(x^4), k[x]/(x^2) \times k[y]/y^2, k \times k[x]/(x^3), k \times k \times k[x]/(x^2), k \times k \times k \times k \].                                             
\end{rmk}
\subsection*{n=3 : \texorpdfstring{\protect\trefoil}{} $1^+2^-3^+1^-2^+3^-$}
The ``trefoil'' curve is immersed in a five-punctured sphere. It was studied in \cite{Seidel} and also appears in \cite{EvansLekili} (with different grading choices). The boundary components are $s_1, s_2, s_3, s_4,s_5$, where $s_1,s_2$ are triangles, $s_3$, $s_4$ and $s_5$ are bigons. We will use the equivalent Gauss word $1^-2^+3^-1^+2^-3^+$ that is compatible with the below picture. The gradings are determined by
$|s_1| = |s_2| = |w_1|+|w_2|+|w_3|$, $|s_3| = 2- |w_1|-|w_2|$, $|s_4| = 2- |w_2|-|w_3|$, $|s_5| = 2- |w_1|-|w_3|$. We choose to use the gradings $|w_1|=|w_3|=1$, and $|w_2|=0$. In this way, $|s_1|=|s_2|=2$ hence can be compactified in the usual way, and $|s_3|=|s_4|=1$ hence can also be compactified with a $\mathbb{Z}/2\mathbb{Z}$-orbifold point.
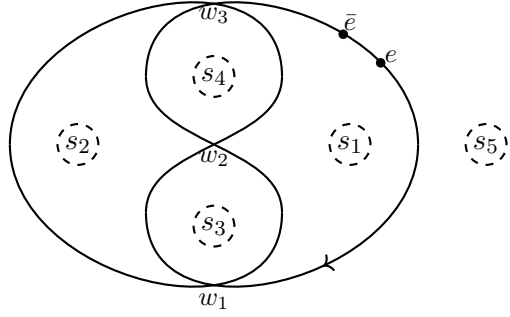
\begin{center}
  \begin{tikzpicture}[xscale=0.9, yscale=0.9]
    \tikzset{->-/.style={decoration={markings, mark=at position #1 with {\arrow{<}}},postaction={decorate}}}
    \node[] (L0left) at (2,2) {};
    \node[] (L00mid)  at (6,3) {};
    \node[] (L0mid)  at (6,1) {};
    \node[] (L0right) at (2,2) {};

    \node[] (L1left) at (8,2) {};
    \node[] (L11mid)  at (4,3) {};
    \node[] (L1mid)  at (4,1) {};
    \node[] (L1right) at (8,2) {};

    \node[] (x1) at (5,3.9) {\small $w_3$}; 

   \node[] (z1) at (5,1.8) {\small $w_2$}; 
  
    \node[] (y1) at (5,-0.3) {\small $w_1$}; 

    \fill[black] (6.9,3.62) circle (2pt);    
    \fill[black] (7.45,3.2) circle (2pt);    

    \node[] (q) at (7,3.8) {\small $\bar{e}$};
    \node[] (e) at (7.6,3.3){\small $e$};
    
    \draw[thick, dashed] (5,0.8) circle (0.3cm);
    \node (s3) at (5,0.8) {$s_3$};

    \draw[thick, dashed] (5,3) circle (0.3cm);
    \node (s4) at (5,3) {$s_4$};

    \draw[thick, dashed] (3,2) circle (0.3cm);
    \node (s2) at (3,2) {$s_2$};

    \draw[thick, dashed] (7,2) circle (0.3cm);
    \node (s1) at (7,2) {$s_1$};
    
    \draw[thick, dashed] (9,2) circle (0.3cm);
    \node (s5) at (9,2) {$s_5$};
  
    \draw[thick] (L0left.center)  .. controls +(270:2) and +(270:2) .. (L0mid.center);
    \draw[thick] (L0right.center)  .. controls +(90:2) and +(90:2) .. (L00mid.center);

    \draw[thick] (L1left.center)  .. controls +(90:2) and +(90:2) .. (L11mid.center);
    \draw[thick, ->-=0.6] (L1mid.center)  .. controls +(270:2) and +(270:2) .. (L1right.center);

    \draw[thick] (L0mid.center)  .. controls +(90:1) and +(270:1) .. (L11mid.center);
    \draw[thick] (L00mid.center)  .. controls +(270:1) and +(90:1) .. (L1mid.center);
  \end{tikzpicture}
\end{center}
 With these choices, the product and the triple products are determined as follows.
\begin{align*} \mathfrak{m}_2(w_i,w_0) &= w_i = (-1)^{|w_i|} \mathfrak{m}_2(w_0,w_i)  \text{ for }i=0,1,2,3\\
\mathfrak{m}_2(\bar{w}_i,w_0) &= \bar{w}_i = (-1)^{|\bar{w}_i|} \mathfrak{m}_2(w_0,\bar{w}_i) \text{ for } i=0,1,2,3\\
\mathfrak{m}_2(\bar{w}_1,w_1) &= \bar{w}_0 = \mathfrak{m}_2(w_1,\bar{w}_1)\\
\mathfrak{m}_2(\bar{w}_2,w_2) &= \bar{w}_0 = -\mathfrak{m}_2(w_2,\bar{w}_2)\\ 
\mathfrak{m}_2(\bar{w}_3,w_3) &= \bar{w}_0 = \mathfrak{m}_2(w_3,\bar{w}_3)
\end{align*}
\begin{multicols}{3}\small
\begin{itemize}[label={}]
\item $\mathfrak{m}_3(w_{1},\bar{w}_{1},w_{3})= - w_{3}$
\item $\mathfrak{m}_3(w_{1},\bar{w}_{1},\bar{w}_{0})= - \bar{w}_{0}$
\item $\mathfrak{m}_3(w_{1},\bar{w}_{1},\bar{w}_{2})= - \bar{w}_{2}$
\item $\mathfrak{m}_3(w_{2},w_{1},\bar{w}_{1})=w_{2}$
\item $\mathfrak{m}_3(w_{2},w_{2},\bar{w}_{2})= - w_{2}$
\item $\mathfrak{m}_3(w_{2},\bar{w}_{1},w_{1})=w_{2}$
\item $\mathfrak{m}_3(w_{2},\bar{w}_{2},w_{1})=w_{1}$
\item $\mathfrak{m}_3(w_{2},\bar{w}_{2},w_{3})=w_{3}$
\item $\mathfrak{m}_3(w_{2},\bar{w}_{2},\bar{w}_{0})=\bar{w}_{0}$
\item $\mathfrak{m}_3(w_{2},\bar{w}_{2},\bar{w}_{2})=\bar{w}_{2}$
\item $\mathfrak{m}_3(w_{2},\bar{w}_{2},\bar{w}_{3})= - \bar{w}_{3}$
\item $\mathfrak{m}_3(w_{2},\bar{w}_{3},w_{3})=w_{2}$
\item $\mathfrak{m}_3(w_{3},w_{2},\bar{w}_{2})= - w_{3}$
\item $\mathfrak{m}_3(w_{3},\bar{w}_{1},w_{1})=w_{3}$
\item $\mathfrak{m}_3(w_{3},\bar{w}_{3},\bar{w}_{0})= - \bar{w}_{0}$
\item $\mathfrak{m}_3(\bar{w}_{1},w_{1},w_{1})= - w_{1}$
\item $\mathfrak{m}_3(\bar{w}_{1},w_{1},w_{2})=w_{2}$
\item $\mathfrak{m}_3(\bar{w}_{1},w_{1},w_{3})= - w_{3}$
\item $\mathfrak{m}_3(\bar{w}_{1},w_{1},\bar{w}_{0})= - \bar{w}_{0}$
\item $\mathfrak{m}_3(\bar{w}_{1},w_{1},\bar{w}_{1})=\bar{w}_{1}$
\item $\mathfrak{m}_3(\bar{w}_{1},w_{1},\bar{w}_{2})= - \bar{w}_{2}$
\item $\mathfrak{m}_3(\bar{w}_{1},w_{1},\bar{w}_{3})=\bar{w}_{3}$
\item $\mathfrak{m}_3(\bar{w}_{1},w_{2},\bar{w}_{2})= - \bar{w}_{1}$
\item $\mathfrak{m}_3(\bar{w}_{1},\bar{w}_{1},w_{1})=\bar{w}_{1}$
\item $\mathfrak{m}_3(\bar{w}_{1},\bar{w}_{3},w_{3})=\bar{w}_{1}$
\item $\mathfrak{m}_3(\bar{w}_{2},w_{2},w_{3})= - w_{3}$
\item $\mathfrak{m}_3(\bar{w}_{2},w_{2},\bar{w}_{0})= - \bar{w}_{0}$
\item $\mathfrak{m}_3(\bar{w}_{2},w_{2},\bar{w}_{2})= - \bar{w}_{2}$
\item $\mathfrak{m}_3(\bar{w}_{2},\bar{w}_{1},w_{1})=\bar{w}_{2}$
\item $\mathfrak{m}_3(\bar{w}_{3},w_{1},\bar{w}_{1})=\bar{w}_{3}$
\item $\mathfrak{m}_3(\bar{w}_{3},w_{2},\bar{w}_{2})= - \bar{w}_{3}$
\item $\mathfrak{m}_3(\bar{w}_{3},w_{3},w_{1})= - w_{1}$
\item $\mathfrak{m}_3(\bar{w}_{3},w_{3},w_{3})= - w_{3}$
\item $\mathfrak{m}_3(\bar{w}_{3},w_{3},\bar{w}_{0})= - \bar{w}_{0}$
\item $\mathfrak{m}_3(\bar{w}_{3},w_{3},\bar{w}_{2})= - \bar{w}_{2}$
\item $\mathfrak{m}_3(\bar{w}_{3},w_{3},\bar{w}_{3})=\bar{w}_{3}$
\item $\mathfrak{m}_3(\bar{w}_{3},\bar{w}_{1},w_{1})=\bar{w}_{3}$
\item $\mathfrak{m}_3(\bar{w}_{3},\bar{w}_{2},w_{2})=\bar{w}_{3}$
\item $\mathfrak{m}_3(\bar{w}_{3},\bar{w}_{3},w_{3})=\bar{w}_{3}$
\end{itemize}
\end{multicols}
Let $\mathfrak{b} = t_1 w_1 + t_2 \bar{w}_2 + t_3 w_3 + \lambda \bar{w}_0$. Then, we have
$\mathfrak{m}_1^{\mathfrak{b}}$ is given as follows:
\begin{align*}
\mathfrak{m}_1^{\mathfrak{b}}(\bar{w}_{1}) &=  2t_1(1 -\lambda) \bar{w}_{0} - t^2_{1} w_{1} - t_{1} t_{2} \bar{w}_{2} - t_{1} t_{3} w_{3} \\
\mathfrak{m}_1^{\mathfrak{b}}(w_2) &= t_{1} t_{2} w_{1} -  t_{2} t_{3} w_{3} \\
\mathfrak{m}_1^{\mathfrak{b}}(\bar{w}_3) &= 2t_3(1-\lambda) \bar{w}_{0} - t_{1} t_{3} w_{1} - t_{2} t_{3} \bar{w}_{2} - t^2_{3} w_{3} 
\end{align*}
We will next turn on the deformations associated with $s_1,s_2,s_3,s_4$.
The only differentials come from the monogons double covering the union of $s_1$ or $s_2$, and $s_3$ or $s_4$. These give
\begin{multicols}{2}\small
\begin{itemize}[label={}]
\item $\mathfrak{m}_1(\bar{w}_1) = (s_2^2 -s_1^2) s_4 w_1$
\item $\mathfrak{m}_1(\bar{w}_3) = (s_2^2 -s_1^2) s_3 w_3$
\end{itemize}
\end{multicols}
The additional $\mathfrak{m}_2$ products come from (i) triangles at $s_1$, $s_2$, (ii) triangles at $s_1$ or $s_2$ union double cover of bigons at $s_3$ or $s_4$, (iii) monogons double covering $s_1$ and $s_3$ or $s_4$ and (iv) bigon with corners at $w_2$ which is supported on the union of $s_1,s_2,s_3,s_4$ and double covering $s_3,s_4$. These give the following contributions:
\begin{multicols}{3}\small
\begin{itemize}[label={}]
\item $\mathfrak{m}_2(\bar{w}_3, \bar{w}_1) = s_1 w_2$
\item $\mathfrak{m}_2(\bar{w}_1, \bar{w}_2) = s_1 w_3$
\item $\mathfrak{m}_2(\bar{w}_2, \bar{w}_3) = s_1 w_1$
\item $\mathfrak{m}_2(\bar{w}_1, \bar{w}_3) = s_2 w_2$
\item $\mathfrak{m}_2(\bar{w}_3, \bar{w}_2) = s_2 w_1$
\item $\mathfrak{m}_2(\bar{w}_2, \bar{w}_1) = s_2 w_3$
\item $\mathfrak{m}_2(\bar{w}_3, w_2) = s_1s_3 \bar{w}_1$
\item $\mathfrak{m}_2(w_2, w_1) = -s_1s_3 w_3$
\item $\mathfrak{m}_2(w_1, \bar{w}_3) = s_1s_3 \bar{w}_2$
\item $\mathfrak{m}_2(w_2, \bar{w}_1) = s_1s_4 \bar{w}_3$
\item $\mathfrak{m}_2(\bar{w}_1, w_3) = s_1s_4 \bar{w}_2$
\item $\mathfrak{m}_2(w_3, w_2) = s_1s_4 w_1$
\item $\mathfrak{m}_2(w_2, \bar{w}_3) = s_2s_3 \bar{w}_1$
\item $\mathfrak{m}_2(\bar{w}_3, w_1) = s_2s_3 \bar{w}_2$
\item $\mathfrak{m}_2(w_1, w_2) = s_2s_3 w_3$
\item $\mathfrak{m}_2(\bar{w}_1, w_2) = s_2s_4 \bar{w}_3$
\item $\mathfrak{m}_2(w_2, w_3) = -s_2s_4 w_1$
\item $\mathfrak{m}_2(w_3, \bar{w}_1) = s_2s_4 \bar{w}_2$
\item $\mathfrak{m}_2(\bar{w}_1,\bar{w}_1)= s_1^2s_4 w_0$
\item $\mathfrak{m}_2(\bar{w}_1,\bar{w}_0)= s_1^2s_4 w_1$
\item $\mathfrak{m}_2(\bar{w}_0,\bar{w}_1)= s_1^2 s_4 w_1$
\item $\mathfrak{m}_2(\bar{w}_3,\bar{w}_3)= s_1^2s_3 w_0$
\item $\mathfrak{m}_2(\bar{w}_3,\bar{w}_0)= s_1^2 s_3 w_3$
\item $\mathfrak{m}_2(\bar{w}_0,\bar{w}_3)= s_1^2 s_3 w_3$
\item $\mathfrak{m}_2(w_2,w_2) = s_1s_2s_3s_4 w_0$
\item $\mathfrak{m}_2(w_2,\bar{w}_0) = -s_1s_2s_3s_4 \bar{w}_2$
\item $\mathfrak{m}_2(\bar{w}_0,w_2) = s_1s_2s_3s_4 \bar{w}_2$
\end{itemize}
\end{multicols}
Next, the additional $\mathfrak{m}_3$ products come from (i) triangle at $s_1$, (ii) double covering of bigons over $s_3$ or $s_4$, (iii) triangle at $s_1$ union double cover of bigons at $s_3$ or $s_4$, (iv) monogons double covering $s_1$ and $s_3$ or $s_4$ and (v) bigon with corners at $w_2$ which is supported on the union of $s_1,s_2,s_3,s_4$ and double covering $s_3,s_4$. These give the following contributions:
\begin{multicols}{2}\small
\begin{itemize}[label={}]
\item $\mathfrak{m}_3(\bar{w}_1, \bar{w}_2, \bar{w}_3) = -s_1 w_0$
\item $\mathfrak{m}_3(\bar{w}_2, \bar{w}_3, \bar{w}_0) = -s_1 w_1$
\item $\mathfrak{m}_3(\bar{w}_3, \bar{w}_0, \bar{w}_1) = -s_1 w_2$
\item $\mathfrak{m}_3(\bar{w}_0, \bar{w}_1, \bar{w}_2) = -s_1 w_3$
\item $\mathfrak{m}_3(w_1,w_2,w_1) = s_3 \bar{w}_2$
\item $\mathfrak{m}_3(w_2,w_1,w_2) = s_3 \bar{w}_1$ 
\item $\mathfrak{m}_3(w_3,w_2,w_3) = -s_4 \bar{w}_2$
\item $\mathfrak{m}_3(w_2,w_3,w_2) = -s_4 \bar{w}_3$ 
\item $\mathfrak{m}_3(w_2,w_1,\bar{w}_3) = s_1s_3 w_0$
\item $\mathfrak{m}_3(w_1,\bar{w}_3, \bar{w}_0) = -s_1s_3 \bar{w}_2$
\item $\mathfrak{m}_3(\bar{w}_3,\bar{w}_0,w_2) = -s_1s_3 \bar{w}_1$
\item $\mathfrak{m}_3(\bar{w}_0,w_2,w_1) = s_1s_3 w_3$
\item $\mathfrak{m}_3(\bar{w}_1,w_3,w_2) = -s_1s_4 w_0$
\item $\mathfrak{m}_3(w_3,w_2,\bar{w}_0) = -s_1s_4 w_1$
\item $\mathfrak{m}_3(w_2,\bar{w}_0,\bar{w}_1) = -s_1s_4 \bar{w}_3$
\item $\mathfrak{m}_3(\bar{w}_0,\bar{w}_1,w_3) = -s_1s_4 \bar{w}_2$
\item $\mathfrak{m}_3 (\bar{w}_0, \bar{w}_3,\bar{w}_0) = - s_1^2 s_3 w_3$
\item $\mathfrak{m}_3 (\bar{w}_3, \bar{w}_3,\bar{w}_0) = -s_1^2 s_3 w_0$
\item $\mathfrak{m}_3 (\bar{w}_3, \bar{w}_0,\bar{w}_3) = -s_1^2 s_3 w_0$
\item $\mathfrak{m}_3 (\bar{w}_0, \bar{w}_1,\bar{w}_0) = - s_1^2 s_4 w_1$
\item $\mathfrak{m}_3 (\bar{w}_1, \bar{w}_1,\bar{w}_0) = -s_1^2 s_4 w_0$
\item $\mathfrak{m}_3 (\bar{w}_1, \bar{w}_0,\bar{w}_1) = -s_1^2 s_4 w_0$
\item $\mathfrak{m}_3 (w_2,w_2,\bar{w}_0) = - s_1s_2s_3s_4 w_0$  
\end{itemize}
\end{multicols}
Finally, there are $\mathfrak{m}_4$ products that come from (i) triangle at $s_1$, (ii) triangle at $s_1$ union double cover of bigons at $s_3$ or $s_4$, and (iii) monogons double covering $s_1$ and $s_3$ or $s_4$. These give the following contributions:
\begin{multicols}{2}\small
\begin{itemize}[label={}]
\item $\mathfrak{m}_4(\bar{w}_1, \bar{w}_2, \bar{w}_3, \bar{w}_0) = s_1 w_0$,
\item $\mathfrak{m}_4 (w_2,w_1, \bar{w}_3,\bar{w}_0) =  -s_1s_3 w_0$,
\item $\mathfrak{m}_4(\bar{w}_1,w_3,w_2,\bar{w}_0) = s_1s_4w_0$,  
\item $\mathfrak{m}_4 (\bar{w}_1, \bar{w}_0, \bar{w}_1, \bar{w}_0) = s_1^2 s_4w_0$,  
\item $\mathfrak{m}_4 (\bar{w}_3, \bar{w}_0, \bar{w}_3, \bar{w}_0) = s_1^2 s_3w_0$.
\item
\end{itemize}
\end{multicols}
That's it! Now, we let $\mathfrak{b}= t_1w_1+ t_2 \bar{w}_2 + t_3 w_3 + \lambda \bar{w}_0$ and we compute $\mathfrak{m}_1^{\mathfrak{b},s_1,s_2,s_3,s_4}$ as follows:
\begin{align*}
\mathfrak{m}_1^{\mathfrak{b}}(\bar{w}_{1}) &=  2t_1(1 -\lambda) \bar{w}_{0} + (s_2^2s_4 - t^2_{1}- s_1^2s_4 (1-\lambda)^2) w_{1} + (s_2s_4t_3 - t_{1} t_{2} + s_1s_4t_3(1-\lambda)) \bar{w}_{2} \\ &+ (s_2t_2 - t_{1} t_{3} + s_1t_2(1-\lambda)) w_{3} \\
\mathfrak{m}_1^{\mathfrak{b}}(w_2) &= (-s_2s_4t_3 + t_{1} t_{2} + s_1s_4t_3(1-\lambda)) w_{1} + (s_3t_1^2 - s_4 t_3^2) \bar{w}_2  + (s_2s_3t_1 -  t_{2} t_{3} - s_1s_3t_1(1-\lambda)) w_{3} \\
\mathfrak{m}_1^{\mathfrak{b}}(\bar{w}_3) &= 2t_3(1-\lambda) \bar{w}_{0} + (s_2t_2- t_{1} t_{3} + s_1t_2(1-\lambda)) w_{1}  + ( s_2s_3 t_1  - t_{2} t_{3} + s_1s_3t_1(1-\lambda)) \bar{w}_{2}\\ &+ (s_2^2s_3- t^2_{3} - s_1^2s_3(1-\lambda)^2) w_{3} 
\end{align*}
Thus, to get a flat deformation over $k[s_1,s_2,s_3,s_4,t_1,t_2,t_3,\lambda]$, we need
\[ 
  2t_1(1-\lambda)=0 = 2t_3(1-\lambda) \]
and the entries of the following (not quite symmetric) matrix should vanish.
\[\begin{pmatrix}
  s_2^2s_4 - t^2_1 - s_1^2s_4 (1-\lambda)^2  &  s_2s_4t_3 - t_{1} t_{2} + s_1s_4t_3(1-\lambda) & s_2t_2 - t_{1} t_{3} + s_1t_2(1-\lambda)  \\
  -s_2s_4t_3 + t_{1} t_{2} + s_1s_4t_3(1-\lambda) &  s_3t_1^2 - s_4 t_3^2  &    s_2s_3t_1 -  t_{2} t_{3} - s_1s_3t_1(1-\lambda)  \\
   s_2t_2 - t_{1} t_{3} + s_1t_2(1-\lambda) &    s_2s_3t_1 -  t_{2} t_{3} + s_1s_3t_1(1-\lambda) & s_2^2s_3- t^2_{3} - s_1^2s_3(1-\lambda)^2   
\end{pmatrix} = 0\]
A straightforward verification shows that these equations cut out 4 irreducible components. We now let $(x_1,x_2,x_3) = (\bar{w}_1, \bar{w}_3,w_2)$ and write out the presentation of the flat family over each of these components (by calculating $\mathfrak{m}_2^{\mathfrak{b},s_1,s_2,s_3,s_4}$ from the above computation and restricting to each component).

\emph{$t_1=t_3=0$, $s_2 = s_1(\lambda-1)$}

We get the flat family of algebras over $k[s_2,s_3,s_4,t_2,\lambda]$
\[ \boxed{k\{ x_1,x_2,x_3\} \Bigg/ \left( \begin{aligned} &x_1^2 - s_2^2s_4 ,\; x_2^2  - s_2^2s_3,\; x_3^2 + t_2 x_3 + s_2^2 s_3s_4,\\
  & x_1x_2 - s_2x_3 - s_2t_2, \; x_2x_1 + s_2 x_3, \\
  & x_1x_3 + t_2 x_1 - s_2s_4 x_2, \; x_3x_1 + s_2s_4 x_2, \\
  & x_2 x_3 + s_2 s_3 x_1,\; x_3 x_2 - s_2 s_3 x_1 + t_2 x_2 
   \end{aligned} \right)} \]

Note that when $t_2=0$ and $s_2 \neq 0$, letting $i = x_1/s_2$ and $j= x_2/s_2$, we get the quaternion algebra $(s_4,s_3)_k$ which is isomorphic to $M_2(k)$ when $k$ is algebraically closed. When $s_3=s_4=t_2=0$, we get the exterior algebra $\Lambda= k\{x_1,x_2 \} / (x_1^2,x_2^2, x_1x_2 + x_2x_1)$. When $t_2 \neq 0$ and $s_2=0$, we recover the quiver algebra {\scriptsize $\bullet \rightleftarrows \bullet$}. Thus, this family recovers all the algebras in the component of $\Alg_4$ that is the closure of $M_2(k)$.

\emph{$t_1=t_2=t_3=0$, $s_2 = s_1(1-\lambda)$}

We get the flat family of commutative algebras over $k[s_2,s_3,s_4,\lambda]$
\[ \boxed{k[x_1,x_2,x_3] / \left( \begin{aligned} &x_1^2 - s_2^2 s_4,\; x_2^2 - s_2^2s_3,\; x_3^2 - s_2^2s_3s_4, \\
  &x_1x_2 - s_2x_3,\; x_2x_3 - s_2s_3x_1,\; x_1x_3 - s_2s_4x_2  
\end{aligned} \right) }\]
We give the isomorphism types of the algebras occurring in this family (over an algebraically closed field $k$) in the following table.
\[
\begin{array}{c|c}
  \hline
  s_2 \neq 0,\; s_3 \neq0,\; s_4 \neq 0 &   k \times k \times k \times k \\
  s_2,s_4 \neq 0,\; s_3=0 \text{ or } s_2,s_3 \neq 0,\; s_4=0 &  k[x]/(x^2) \times k[y]/(y^2) \\
  s_2 \neq 0,\; s_3=s_4=0 & k[x,y]/ (x^2,y^2) \\
s_2 = 0 \text{ or } s_3=s_4=0&  k[x,y,z]/(x,y,z)^2 \\ \hline
\end{array}
\]
\emph{$t_1=t_2= t_3=0$, $s_3=s_4=0$}

We get the flat family of algebras over $k[s_1,s_2,\lambda]$ 
\[\boxed{k\{ x_1,x_2,x_3\} \Big/ \left( x_1^2,\; x_2^2,\; x_3^2,\; x_1x_2 - s_2x_3,\; x_2x_1 - s_1(1-\lambda)x_3, \; x_1x_3,\; x_3x_1,\; x_2x_3,\; x_3x_2 \right)}\]

This gives us all the quantum algebras $A_q= k\{x,y\}/(x^2,y^2,xy-qyx)$ with $s_1(1-\lambda)/s_2 = q$.

\emph{$\lambda=1$}

The equations describe the condition that the following matrix has rank $\le 1$:
\[ \begin{pmatrix} s_2 & t_1 & t_3 \\ t_1 & s_2s_4 & t_2 \\ t_3 & t_2 & s_2s_3 \end{pmatrix} \]
Over the dense open set where $s_2\neq0$, the variety can be rationally parameterized by setting $t_1=s_2u, t_3=s_2v, s_4=u^2, s_3=v^2$, and $t_2=s_2uv$, proving this locus is irreducible. Its closure defines our final $4$-dimensional component. We get the flat family of algebras over $k[s_1,s_2,u,v]$ given by
\[ \boxed{k\{ x_1,x_2,x_3\} \Bigg/ \left( \begin{aligned}
  & x_1^2 - 2s_2ux_1,\; x_2^2 - 2s_2vx_2,\; x_3^2 - s_2uv(vx_1-ux_2-x_3), \\
  &x_1x_2 - s_2(vx_1+ux_2+x_3),\; x_2x_1 - 2s_2ux_2 \\
  &x_1x_3 - s_2u(-vx_1+ux_2+x_3),\; x_3x_1 - 2s_2ux_3 \\
  &x_2x_3,\; x_3x_2 - s_2v(vx_1-ux_2+x_3)  
\end{aligned} \right) } \]
When $s_2,u,v \neq 0$, let $e_1 = \frac{x_1}{2s_2u}, e_2 = \frac{x_2}{2s_2v}$. Then, $e_1^2=e_1$, $e_2^2=e_2$ and $e_2e_1 =e_2$. Let $e= e_1-e_1e_2$. Then, it can be verified easily that $e$ is a central idempotent. Hence, it splits a field summand. The orthogonal idempotent is $1-e_1+e_1e_2$. It can be shown that $1-e_1, e_1e_2, e_2-e_1e_2$ form a basis of this summand, and give an isomorphism to the algebra ${\scriptsize \begin{pmatrix} k & k \\ 0 &k \end{pmatrix}}$. One can do a similar analysis to show that for $u=0$, $v,s_2 \neq 0$, we get an isomorphism to ${\scriptsize \begin{pmatrix} k[\epsilon]/\epsilon^2 & k \\ 0 &k \end{pmatrix}}$, and dually for $v=0, u,s_2 \neq 0$ to ${\scriptsize \begin{pmatrix} k & k \\ 0 &k[\epsilon]/\epsilon^2 \end{pmatrix}}$.
We give the isomorphism types of the algebras occurring in this family in the following table.
\[
\begin{array}{c|c}
  \hline
  s_2 \neq 0,\; u \neq0,\; v\neq 0 &   k \times {\scriptstyle(\bullet \to \bullet)}  \\
 s_2 \neq 0,\; u=0,\; v\neq 0 &  {\scriptsize \begin{tikzcd}
    \bullet \arrow[distance=1.5em, out=150,in=210,loop,swap] \arrow[r] & \bullet
    \end{tikzcd}}  \\
s_2 \neq 0,\; u \neq 0,\; v=0  & {\scriptsize\begin{tikzcd} \bullet \arrow[r] & \bullet \arrow[distance=1.5em,out=330,in=30,loop,swap]
    \end{tikzcd}}  \\
  s_2 \neq 0,\; u=v=0 & k\{x,y\}/(x^2,y^2,xy) \\
s_2 = 0 & k[x,y,z]/(x,y,z)^2 \\ \hline
\end{array}
\]
\subsection*{n=3 :  \texorpdfstring{\protect\eears}{}    $1^+1^-2^+2^-3^+3^-$}
This curve is immersed in a five-punctured sphere. The boundary components are $m_1, m_2, m_3, s_1, s_2$, where $m_1, m_2$ and $m_3$ are monogons, and $s_1$ is a triangle, and $s_2$ is a hexagon.  We do not give explicit computations in this case. Note that this case is related to the previous one by a Reidemeister III move (cf. \cite{PalmerWoodward}).

\subsection*{n=3 : $1^+2^+3^+1^-2^-3^-$}
This curve is immersed in a three-punctured torus. It was studied in \cite{LekiliTevelev}. The boundary components are $s_1, s_2, s_3$, where $s_1$ is a bigon, $s_2$ is a rectangle, and $s_3$ is a hexagon. The gradings are determined by $|s_1| = 2 - |w_1| - |w_3|$, $|s_2| = 2 - |w_1| + 2|w_2| - |w_3|$, and $|s_3| = 2+ 2|w_1| -2 |w_2| +  2|w_3|$.
\begin{center}
  \begin{tikzpicture}[xscale=0.9, yscale=0.9]

    \tikzset{->-/.style={decoration={markings, mark=at position #1 with {\arrow{>}}},postaction={decorate}}}

    \node[] (t1) at (0,0) {};
    \node[] (t2) at (0,4) {};
    \node[] (t3) at (10,4) {};
    \node[] (t4) at (10,0) {};

    \node[] (L0left) at (0,1.4) {};
    \node[] (L0up)  at (6,4) {};
    \node[] (L0down) at (6, 0) {};
    \node[] (L0right) at (10,0.4) {};

    \node[] (L1left) at (0,3.4) {};
    \node[] (L1right) at (10,1.4) {};

    \node[] (L2left) at (0,0.4) {};
    \node[] (L2right) at (10,3.4) {};

    \node[] (L3left) at (0,2.4) {};
    \node[] (L3right) at (10,2.4) {};

    \node[] (L1up)  at (3,4) {};
    \node[] (L1down) at (3, 0) {};
           
    \draw[gray!90, thick] (t1.center) -- (t2.center) -- (t3.center) -- (t4.center) -- (t1.center);

    \draw[thick, dashed] (0.7,0.8) circle (0.2cm);
    \node (s1) at (0.7,0.8) {$\scriptstyle s_1$};

    \draw[thick, dashed] (3.7,1.8) circle (0.2cm);
    \node (s2) at (3.7,1.8) {$\scriptstyle s_2$};

    \draw[thick, dashed] (5.7,2.5) circle (0.2cm);
    \node (s3) at (5.7,2.5) {$\scriptstyle s_3$};


    \draw[thick] (L2right.center) .. controls +(180:1) and +(330:1) .. (L0up.center);
    \draw[thick] (L3right.center) .. controls +(180:1) and +(330:1) .. (L1up.center);

    \draw[thick] (L1left.center) .. controls +(0:1) and +(180:1) .. (L0right.center);

    \draw[thick] (L0left.center) .. controls +(0:1) and +(150:1) .. (L1down.center) ;
    \draw[thick] (L3left.center) .. controls +(0:1) and +(150:1) .. (L0down.center) ;
   
    \draw[thick, ->-=.1] (L2left.center) .. controls +(0:1) and +(180:1) .. (L1right.center);
    
    \fill[black] (7.9,1.22) circle (2pt);    
    \fill[black] (8.7,1.31) circle (2pt);    

    \node[] (q) at (7.9,1.45) {\small $\bar{e}$};
    \node[] (e) at (8.7,1.52) {\small $e$};
    
    \node[] (x1) at (3.5,0.95) {\small $w_2$}; 
    \node[] (y1) at (1.5,0.7) {\small $w_1$}; 
    \node[] at (6.5,1.24) {\small $w_3$}; 

  \end{tikzpicture}
\end{center}

\begin{multicols}{3}\small
\begin{itemize}[label={}]
\item $\mathfrak{m}_3(w_{1},w_{1},\bar{w}_{1})= - w_{1}$
\item $\mathfrak{m}_3(w_{1},w_{2},\bar{w}_{2})= - w_{1}$
\item $\mathfrak{m}_3(w_{1},w_{3},\bar{w}_{3})= - w_{1}$
\item $\mathfrak{m}_3(w_{1},\bar{w}_{1},w_{2})= - w_{2}$
\item $\mathfrak{m}_3(w_{1},\bar{w}_{1},w_{3})= - w_{3}$
\item $\mathfrak{m}_3(w_{1},\bar{w}_{1},\bar{w}_{0})=\bar{w}_{0}$
\item $\mathfrak{m}_3(w_{1},\bar{w}_{1},\bar{w}_{1})=\bar{w}_{1}$
\item $\mathfrak{m}_3(w_{1},\bar{w}_{1},\bar{w}_{2})=\bar{w}_{2}$
\item $\mathfrak{m}_3(w_{1},\bar{w}_{1},\bar{w}_{3})=\bar{w}_{3}$
\item $\mathfrak{m}_3(w_{2},w_{1},\bar{w}_{1})= - w_{2}$
\item $\mathfrak{m}_3(w_{2},w_{2},\bar{w}_{2})= - w_{2}$
\item $\mathfrak{m}_3(w_{2},w_{3},\bar{w}_{3})= - w_{2}$
\item $\mathfrak{m}_3(w_{2},\bar{w}_{1},w_{1})=w_{2}$
\item $\mathfrak{m}_3(w_{2},\bar{w}_{2},w_{3})= - w_{3}$
\item $\mathfrak{m}_3(w_{2},\bar{w}_{2},\bar{w}_{0})=\bar{w}_{0}$
\item $\mathfrak{m}_3(w_{2},\bar{w}_{2},\bar{w}_{1})=\bar{w}_{1}$
\item $\mathfrak{m}_3(w_{2},\bar{w}_{2},\bar{w}_{2})=\bar{w}_{2}$
\item $\mathfrak{m}_3(w_{2},\bar{w}_{2},\bar{w}_{3})=\bar{w}_{3}$
\item $\mathfrak{m}_3(w_{3},w_{1},\bar{w}_{1})= - w_{3}$
\item $\mathfrak{m}_3(w_{3},w_{2},\bar{w}_{2})= - w_{3}$
\item $\mathfrak{m}_3(w_{3},w_{3},\bar{w}_{3})= - w_{3}$
\item $\mathfrak{m}_3(w_{3},\bar{w}_{1},w_{1})=w_{3}$
\item $\mathfrak{m}_3(w_{3},\bar{w}_{2},w_{2})=w_{3}$
\item $\mathfrak{m}_3(w_{3},\bar{w}_{3},\bar{w}_{0})=\bar{w}_{0}$
\item $\mathfrak{m}_3(w_{3},\bar{w}_{3},\bar{w}_{1})=\bar{w}_{1}$
\item $\mathfrak{m}_3(w_{3},\bar{w}_{3},\bar{w}_{2})=\bar{w}_{2}$
\item $\mathfrak{m}_3(w_{3},\bar{w}_{3},\bar{w}_{3})=\bar{w}_{3}$
\item $\mathfrak{m}_3(\bar{w}_{1},w_{1},\bar{w}_{0})= - \bar{w}_{0}$
\item $\mathfrak{m}_3(\bar{w}_{1},w_{1},\bar{w}_{1})= - \bar{w}_{1}$
\item $\mathfrak{m}_3(\bar{w}_{1},w_{1},\bar{w}_{2})= - \bar{w}_{2}$
\item $\mathfrak{m}_3(\bar{w}_{1},w_{1},\bar{w}_{3})= - \bar{w}_{3}$
\item $\mathfrak{m}_3(\bar{w}_{2},w_{1},\bar{w}_{1})= - \bar{w}_{2}$
\item $\mathfrak{m}_3(\bar{w}_{2},w_{2},\bar{w}_{0})= - \bar{w}_{0}$
\item $\mathfrak{m}_3(\bar{w}_{2},w_{2},\bar{w}_{2})= - \bar{w}_{2}$
\item $\mathfrak{m}_3(\bar{w}_{2},w_{2},\bar{w}_{3})= - \bar{w}_{3}$
\item $\mathfrak{m}_3(\bar{w}_{3},w_{1},\bar{w}_{1})= - \bar{w}_{3}$
\item $\mathfrak{m}_3(\bar{w}_{3},w_{2},\bar{w}_{2})= - \bar{w}_{3}$
\item $\mathfrak{m}_3(\bar{w}_{3},w_{3},\bar{w}_{0})= - \bar{w}_{0}$
\item $\mathfrak{m}_3(\bar{w}_{3},w_{3},\bar{w}_{3})= - \bar{w}_{3}$
\end{itemize}
\end{multicols}

We let $\mathfrak{b} = t_1 \bar{w}_1 + t_2 \bar{w}_2 + t_3 \bar{w}_3$. Then, $\mathfrak{m}_1^{\mathfrak{b}}$ is given as follows:
\begin{align*}
\mathfrak{m}_1^{\mathfrak{b}}(w_{1}) &=  - t_{1} t_{2} \bar{w}_{2} - t_{1} t_{3} \bar{w}_{3} \\
\mathfrak{m}_1^{\mathfrak{b}}(w_2) &= t_{1} t_{2} \bar{w}_{1} -  t_{2} t_{3} \bar{w}_{3} \\
\mathfrak{m}_1^{\mathfrak{b}}(w_3) &=  t_{1} t_{3} \bar{w}_{1} +  t_{2} t_{3} \bar{w}_{2}            
\end{align*}
Letting $(x_1,x_2,x_3)= (w_1,w_2,w_3)$. The deformed algebra over $k[t_1,t_2,t_3]/(t_1t_2,t_1t_3,t_2t_3)$ is
\[ 
\boxed{k\{x_1,x_2,x_3\}\Big/\left(
\begin{aligned}
& x_1^2+t_1x_1,\; x_2^2+ t_2x_2,\; x_3^2+t_3x_3,\\
& x_1x_2+t_2x_1+t_1x_2,\; x_2x_1,\; x_2x_3+t_2x_3 +t_3x_2,\\
& x_3x_2,\; x_1x_3+t_3x_1+t_1x_3,\; x_3x_1
\end{aligned}
\right)}
\]
We give the isomorphism types of the algebras occurring in this family in the following table.
\[
\begin{array}{c|c}
\hline 
 t_2=t_3 =0, t_1 \neq 0
  &\bullet \rightrightarrows \bullet \\[6pt]
  t_1= t_3=0 , t_2 \neq 0
  &\bullet \leftrightarrows \bullet \\[6pt]
  t_1= t_2=0 , t_3 \neq 0
  &\bullet \rightrightarrows \bullet  \\[6pt]
  (t_1,t_2,t_3) = (0,0,0)
  & k[x,y,z]/(x,y,z)^2 \\[6pt]  \hline
\end{array}
\]
Turning on the $s_1$ and $s_2$ deformations, we additionally get the following contributions:
\begin{align*}
  &\mathfrak{m}_1(w_1) = -s_1 \bar{w}_3,\quad \mathfrak{m}_1(w_3) = s_1 \bar{w}_1 \\
  &\mathfrak{m}_2(w_1,w_3) = s_1w_0, \quad \mathfrak{m}_2(w_3,\bar{w}_0) =- s_1\bar{w}_1, \quad \mathfrak{m}_2(\bar{w}_0,\bar{w}_1) =s_1\bar{w}_3 \\
  & \mathfrak{m}_3(w_1,w_3,\bar{w}_0) = - s_1 w_0
\end{align*}
\begin{align*}
  \mathfrak{m}_3(w_3, \bar{w}_2,w_1) &= s_2 w_2 \\
  \mathfrak{m}_3(\bar{w}_2,w_1,\bar{w}_2) &= -s_2 \bar{w}_3 \\
  \mathfrak{m}_3(w_1, \bar{w}_2,w_3) &= - s_2 w_2 \\
  \mathfrak{m}_3(\bar{w}_2,w_3,\bar{w}_2) &= s_2\bar{w}_1 
\end{align*}
We get an algebra over $k[t_1,t_2,t_3] / (t_1t_2,t_2t_3, t_1t_3 + s_2t_2^2 + s_1)$ given by 
\[ 
\boxed{k\{x_1,x_2,x_3\}\Big/\left(
\begin{aligned}
& x_1^2+t_1x_1,\; x_2^2+ t_2x_2,\; x_3^2+t_3x_3,\\
& x_1x_2+t_2x_1+t_1x_2,\; x_2x_1,\; x_2x_3+t_2x_3 +t_3x_2,\\
& x_3x_2,\; x_1x_3+t_3x_1+t_1x_3 - s_2 t_2x_2 +s_1 ,\; x_3x_1 - s_2 t_2 x_2 
\end{aligned}
\right)}
\]
The base has two irreducible components: Over the component $t_1=t_3=0$, $s_1 + s_2t_2^2 = 0$ we get a deformation to the matrix algebra $M_2(k)$, and over the component $t_2=0$, $s_1 +  t_1t_3= 0$, we get a deformation to $\bullet \rightrightarrows\bullet$. As both of these algebras are rigid, turning on $s_3$ will not lead to any new algebras. Moreover, there are infinitely many polygons to worry about if we also turn on $s_3$.

\begin{rmk} The algebra $\bullet \rightrightarrows\bullet$ arises as the endomorphism algebra of the exceptional collection $\langle \mathcal{O},\mathcal{O}(1)\rangle$ on $\mathbb{P}^1$. The mirror to this is well known and corresponds to thimbles of the Lefschetz fibration $\mathbb{C}^* \to \mathbb{C}$, $z \to z+\frac{1}{z}$ given by a double branched covering at $\{\pm 1\}$, hence this algebra can be realized as the endomorphism algebra of two arcs in a partially wrapped Fukaya category of the cylinder. One can stabilize this to the Lefschetz fibration $\mathbb{C}^* \times \mathbb{C} \to \mathbb{C}$ given by $(z,w) \to z + \frac{1}{z} + w^2$, and the restriction of the thimbles of this fibration to the fiber, which is a punctured torus, is closely related to our construction of the algebra $\bullet \rightrightarrows \bullet$ above.  This remark also gives a hint about a geometric meaning of what it means to retain only the degree 0 part, $\mathrm{End}(L)$, of an immersed curve $L$ on a surface $\Sigma$. Namely, one should imagine the algebra $\mathrm{End}(L)$ as the (full) endomorphism algebra of a (possibly immersed) Lagrangian on a 4-dimensional symplectic manifold equipped with a Lefschetz fibration whose fiber is $\Sigma$ such that $L$ is the restriction of that Lagrangian to a fiber of this Lefschetz fibration. This was the point of view exploited in \cite{LekiliTevelev} for a class of immersed curves on a punctured torus.
\end{rmk}
\subsection*{n=3 : $1^+2^+3^-3^+1^-2^-$}
This is a curve immersed in a three-punctured torus. The boundary components are $m_1$, $s_1$, $s_2$, where $m_1$ is a monogon, $s_1$ is a square and $s_2$ is a 7-gon. The grading is determined by $|s_1| = 4 - |w_1|-|w_2| -2|w_3|$, $|s_2| = 2 + |w_1|+|w_2| +|w_3|$, and $|m_1|= |w_3|$.
We choose to use the gradings such that $|w_1|=|w_2|=0$ and $|w_3|=1$. In this way, $s_1$ can be compactified in the usual way and $m_1$ can also be compactified with a $\mathbb{Z}/2\mathbb{Z}$-orbifold point.
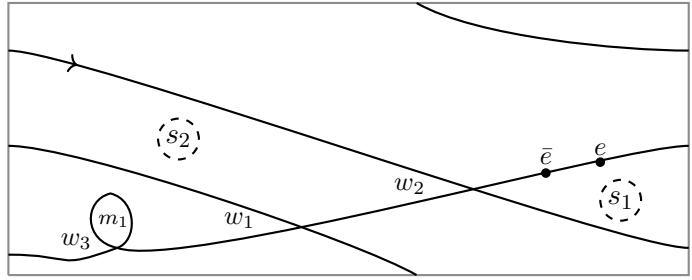
\begin{center}
\begin{tikzpicture}[xscale=0.9, yscale=0.9]
    \tikzset{->-/.style={decoration={markings, mark=at position #1 with {\arrow{>}}},postaction={decorate}}}

    \node[] (t1) at (0,0) {};
    \node[] (t2) at (0,4) {};
    \node[] (t3) at (10,4) {};
    \node[] (t4) at (10,0) {};

    \node[] (L0left) at (0,1.9) {};
    \node[] (L0up)  at (6,4) {};
    \node[] (L0down) at (6, 0) {};
    \node[] (L0right) at (10,0.4) {};

    \node[] (L1left) at (0,3.3) {};
    \node[] (L1right) at (10,1.9) {};

    \node[] (L2left) at (0,0.3) {};
    \node[] (L2right) at (10,3.3) {};

   \node[] (Lmid) at (1.6,0.4) {};

    \draw[gray!90, thick] (t1.center) -- (t2.center) -- (t3.center) -- (t4.center) -- (t1.center);

   \draw[thick, dashed] (2.5,2) circle (0.3cm);
    \node (s2) at (2.5,2) {$s_2$};
    \draw[thick, dashed] (9,1.1) circle (0.3cm);
    \node (s1) at (9,1.1) {$s_1$};
  
    \node (m1) at (1.55,0.8) {\scriptsize $m_1$};

    \draw[thick] (L2right.center) .. controls +(180:1) and +(330:1) .. (L0up.center);
    \draw[thick, ->-=.1] (L1left.center) .. controls +(0:1) and +(180:1) .. (L0right.center);

    \draw[thick] (L0left.center) .. controls +(0:1) and +(150:1) .. (L0down.center) ;
   
    \draw[thick] (L2left.center) .. controls +(0:1) and +(200:1) .. (Lmid.center);

    \draw[thick] (Lmid.center) .. controls (1.9,0.5) and (1.9,1.1) .. (1.5,1.2)
                               .. controls (1.1,1.1) and (1.1,0.5) .. (Lmid.center);

    \draw[thick] (Lmid.center) .. controls +(340:1) and +(180:1) .. (L1right.center);
    
    \fill[black] (7.9,1.5) circle (2pt);    
    \fill[black] (8.7,1.66) circle (2pt);    

    \node[] (q) at (7.9,1.75) {\small $\bar{e}$};
    \node[] (e) at (8.7,1.85) {\small $e$};
    
    \node[] (w1) at (1,0.5) {\small $w_3$}; 
    \node[] (w2) at (5.9,1.3) {\small $w_2$}; 
    \node[] (w3) at (3.4,0.8) {\small $w_1$}; 
  \end{tikzpicture}
\end{center}
With these choices, the product and the triple products are determined as follows.
\begin{align*} \mathfrak{m}_2(w_i,w_0) &= w_i = (-1)^{|w_i|} \mathfrak{m}_2(w_0,w_i)  \text{ for }i=0,1,2,3\\
\mathfrak{m}_2(\bar{w}_i,w_0) &= \bar{w}_i = (-1)^{|\bar{w}_i|} \mathfrak{m}_2(w_0,\bar{w}_i) \text{ for } i=0,1,2,3\\
\mathfrak{m}_2(\bar{w}_1,w_1) &= \bar{w}_0 = -\mathfrak{m}_2(w_1,\bar{w}_1)\\
\mathfrak{m}_2(\bar{w}_2,w_2) &= \bar{w}_0 = -\mathfrak{m}_2(w_2,\bar{w}_2)\\ 
\mathfrak{m}_2(\bar{w}_3,w_3) &= \bar{w}_0 = \mathfrak{m}_2(w_3,\bar{w}_3)
\end{align*}
\begin{multicols}{3}\small
  \begin{itemize}[label={}]
    \item $\mathfrak{m}_3(w_{1},w_{1},\bar{w}_{1})= - w_{1}$
\item $\mathfrak{m}_3(w_{1},w_{2},\bar{w}_{2})= - w_{1}$
\item $\mathfrak{m}_3(w_{1},w_{3},\bar{w}_{3})=w_{1}$
\item $\mathfrak{m}_3(w_{1},\bar{w}_{1},w_{2})= - w_{2}$
\item $\mathfrak{m}_3(w_{1},\bar{w}_{1},w_{3})=w_{3}$
\item $\mathfrak{m}_3(w_{1},\bar{w}_{1},\bar{w}_{0})=\bar{w}_{0}$
\item $\mathfrak{m}_3(w_{1},\bar{w}_{1},\bar{w}_{1})=\bar{w}_{1}$
\item $\mathfrak{m}_3(w_{1},\bar{w}_{1},\bar{w}_{2})=\bar{w}_{2}$
\item $\mathfrak{m}_3(w_{1},\bar{w}_{1},\bar{w}_{3})= - \bar{w}_{3}$
\item $\mathfrak{m}_3(w_{1},\bar{w}_{3},w_{3})=w_{1}$
\item $\mathfrak{m}_3(w_{2},w_{1},\bar{w}_{1})= - w_{2}$
\item $\mathfrak{m}_3(w_{2},w_{2},\bar{w}_{2})= - w_{2}$
\item $\mathfrak{m}_3(w_{2},w_{3},\bar{w}_{3})=w_{2}$
\item $\mathfrak{m}_3(w_{2},\bar{w}_{1},w_{1})=w_{2}$
\item $\mathfrak{m}_3(w_{2},\bar{w}_{2},w_{3})=w_{3}$
\item $\mathfrak{m}_3(w_{2},\bar{w}_{2},\bar{w}_{0})=\bar{w}_{0}$
\item $\mathfrak{m}_3(w_{2},\bar{w}_{2},\bar{w}_{1})=\bar{w}_{1}$
\item $\mathfrak{m}_3(w_{2},\bar{w}_{2},\bar{w}_{2})=\bar{w}_{2}$
\item $\mathfrak{m}_3(w_{2},\bar{w}_{2},\bar{w}_{3})= - \bar{w}_{3}$
\item $\mathfrak{m}_3(w_{2},\bar{w}_{3},w_{3})=w_{2}$
\item $\mathfrak{m}_3(w_{3},w_{1},\bar{w}_{1})= - w_{3}$
\item $\mathfrak{m}_3(w_{3},w_{2},\bar{w}_{2})= - w_{3}$
\item $\mathfrak{m}_3(w_{3},\bar{w}_{3},\bar{w}_{0})= - \bar{w}_{0}$
\item $\mathfrak{m}_3(w_{3},\bar{w}_{3},\bar{w}_{1})= - \bar{w}_{1}$
\item $\mathfrak{m}_3(w_{3},\bar{w}_{3},\bar{w}_{2})= - \bar{w}_{2}$
\item $\mathfrak{m}_3(\bar{w}_{1},w_{1},\bar{w}_{0})= - \bar{w}_{0}$
\item $\mathfrak{m}_3(\bar{w}_{1},w_{1},\bar{w}_{1})= - \bar{w}_{1}$
\item $\mathfrak{m}_3(\bar{w}_{1},w_{1},\bar{w}_{2})= - \bar{w}_{2}$
\item $\mathfrak{m}_3(\bar{w}_{2},w_{1},\bar{w}_{1})= - \bar{w}_{2}$
\item $\mathfrak{m}_3(\bar{w}_{2},w_{2},\bar{w}_{0})= - \bar{w}_{0}$
\item $\mathfrak{m}_3(\bar{w}_{2},w_{2},\bar{w}_{2})= - \bar{w}_{2}$
\item $\mathfrak{m}_3(\bar{w}_{3},w_{1},\bar{w}_{1})= - \bar{w}_{3}$
\item $\mathfrak{m}_3(\bar{w}_{3},w_{2},\bar{w}_{2})= - \bar{w}_{3}$
\item $\mathfrak{m}_3(\bar{w}_{3},w_{3},w_{3})= - w_{3}$
\item $\mathfrak{m}_3(\bar{w}_{3},w_{3},\bar{w}_{0})= - \bar{w}_{0}$
\item $\mathfrak{m}_3(\bar{w}_{3},w_{3},\bar{w}_{1})= - \bar{w}_{1}$
\item $\mathfrak{m}_3(\bar{w}_{3},w_{3},\bar{w}_{2})= - \bar{w}_{2}$
\item $\mathfrak{m}_3(\bar{w}_{3},w_{3},\bar{w}_{3})=\bar{w}_{3}$
\item $\mathfrak{m}_3(\bar{w}_{3},\bar{w}_{3},w_{3})=\bar{w}_{3}$
\end{itemize}
\end{multicols}
Turning on the $s_1$ parameter, we get the following contributions from the rectangle:
\begin{multicols}{2}\small
\begin{itemize}[label={}]
\item $\mathfrak{m}_3(w_3,w_3,w_2) = - s_1 \bar{w}_1$
\item $\mathfrak{m}_3(w_3,w_2,w_1) = - s_1 \bar{w}_3$
\item $\mathfrak{m}_3(w_{2},w_{1},w_{3}) = s_1 \bar{w}_{3}$
\item $\mathfrak{m}_3(w_{1},w_{3},w_{3}) = s_1 \bar{w}_{2}$
\item $\mathfrak{m}_4(w_{3},w_{3},w_{2},\bar{w}_{0}) = s_1 \bar{w}_{1}$
\item $\mathfrak{m}_4(w_{3},w_{2},\bar{w}_{0},w_{1}) = s_1 \bar{w}_{3}$
\item $\mathfrak{m}_4(w_{2},\bar{w}_{0},w_{1},w_{3}) = -s_1\bar{w}_{3}$
\item $\mathfrak{m}_4(\bar{w}_{0},w_{1},w_{3},w_{3}) = -s_1 \bar{w}_{2}$
\item $\mathfrak{m}_4(w_{1},w_{3},w_{3},w_{2}) = -s_1 w_{0}$
\item $\mathfrak{m}_5(w_{1},w_{3},w_{3},w_{2},\bar{w}_{0}) = s_1w_{0}$
\end{itemize}
\end{multicols}
Next, turning on $m_1$, we get the following contributions from (i) double covering of the $m_1$ monogon, (ii) the bigon that covers $s_1$ and double covers $m_1$.
\begin{multicols}{3}\small
  \begin{itemize}[label={}]
\item $\mathfrak{m}_3(\bar{w}_{3}) = -m_1 w_3$      
\item $\mathfrak{m}_1(w_{1}) = s_1m_1\bar{w}_{2}$
\item $\mathfrak{m}_1(w_{2}) = -s_1m_1 \bar{w}_{1}$
\item $\mathfrak{m}_2(w_{1}, w_{2}) = - s_1m_1w_{0}$
\item $\mathfrak{m}_2(w_{2}, \bar{w}_{0}) = s_1m_1 \bar{w}_{1}$
\item $\mathfrak{m}_2(\bar{w}_{0}, w_{1}) = - s_1m_1\bar{w}_{2}$
\item $\mathfrak{m}_3(w_{1}, w_{2}, \bar{w}_{0}) = s_1m_1 w_{0}$
\end{itemize}
\end{multicols}

  We now let $\mathfrak{b} = t_1 \bar{w}_1 + t_2 \bar{w}_2 + t_3 w_3 + \lambda \bar{w}_0$, and calculate:
\begin{align*}
\mathfrak{m}_1^{\mathfrak{b}}(w_{1}) &=  (- t_{1} t_{2}+ s_1(m_1+t_3^2)(1-\lambda)) \bar{w}_{2}  \\
\mathfrak{m}_1^{\mathfrak{b}}(w_2) &= (t_{1} t_{2} - s_1(m_1+t_3^2)(1-\lambda))\bar{w}_{1} \\
\mathfrak{m}_1^{\mathfrak{b}}(\bar{w}_3) &= 2t_3(1-\lambda)\bar{w}_0  -2 t_{1} t_{3} \bar{w}_{1} - 2 t_2t_3 \bar{w}_2 - (m_1+t_3^2) w_3            
\end{align*}
We let $(x_1,x_2,x_3) = (w_1,w_2,\bar{w}_3)$. Then, the deformed algebra over $k[t_1,t_2,t_3] / (t_1t_2, t_1t_3,t_2t_3)$ is
\[ 
\boxed{k\{x_1,x_2,x_3\}\Big/\left(
\begin{aligned}
& x_1^2+t_1x_1,\; x_2^2+ t_2x_2,\; x_3^2-2t_3x_3,\\
& x_1x_2+t_2x_1+t_1x_2+t_1t_2,\; x_2x_1,\; x_2x_3-2t_3x_2+t_2x_3,\\
& x_3x_2 + t_2 x_3,\; x_1x_3-2t_3x_1 + t_1x_3 ,\; x_3x_1 + t_1 x_3 
\end{aligned}
\right)}
\]
We give the isomorphism types of the algebras occurring in this family. 
\[
\begin{array}{c|c}
  \hline
 t_1=t_2=t_3= 0 &    k[x,y,z]/(x,y,z)^2  \\
 t_2=t_3=0, t_1 \neq 0 & {\scriptsize\begin{tikzcd} \bullet \arrow[r] & \bullet \arrow[distance=1.5em,out=330,in=30,loop,swap] \end{tikzcd}} \\[6pt]
 t_1=t_3=0, t_2\neq 0 & {\scriptsize \begin{tikzcd}
    \bullet \arrow[distance=1.5em, out=150,in=210,loop,swap] \arrow[r] & \bullet
    \end{tikzcd}}    \\[6pt]
  t_3\neq 0, t_1=t_2=0 &{\scriptsize\begin{tikzcd}[column sep=large] \bullet \rightrightarrows \bullet \end{tikzcd} }   \\[6pt]
\end{array}
\]
  \subsection*{n=3 : $1^+2^+1^-3^+3^-2^-$}
This curve is immersed in a three-punctured torus. The boundary components are a monogon $m_1$ and $s_1,s_2$, where $s_1$ is a bigon and $s_2$ is a $9$-gon. The gradings are determined by $|s_1| = 2- |w_1|-|w_2|, |s_2|=  4 + |w_1| + |w_2| - |w_3|$,  and $|m_1| = |w_3|$.
\begin{center}
  \begin{tikzpicture}[xscale=0.9, yscale=0.9]

    \tikzset{->-/.style={decoration={markings, mark=at position #1 with {\arrow{>}}},postaction={decorate}}}

    \node[] (t1) at (0,0) {};
    \node[] (t2) at (0,4) {};
    \node[] (t3) at (10,4) {};
    \node[] (t4) at (10,0) {};

    \node[] (L0left) at (0,1.9) {};
    \node[] (L0up)  at (6,4) {};
    \node[] (L0down) at (6, 0) {};
    \node[] (L0right) at (10,0.4) {};

    \node[] (L1left) at (0,3.3) {};
    \node[] (L1right) at (10,1.9) {};

    \node[] (L2left) at (0,0.3) {};
    \node[] (L2right) at (10,3.3) {};

   \node[] (Lmid) at (6,1.1) {};

    \draw[gray!90, thick] (t1.center) -- (t2.center) -- (t3.center) -- (t4.center) -- (t1.center);

   \draw[thick, dashed] (2.5,2) circle (0.3cm);
    \node (s2) at (2.5,2) {$s_2$};
    \draw[thick, dashed] (1,1.1) circle (0.3cm);
    \node (s1) at (1,1.1) {$s_1$};

    \node (m1) at (6.3,0.75) {\scriptsize $m_1$};
    \draw[thick] (L2right.center) .. controls +(180:1) and +(330:1) .. (L0up.center);
    \draw[thick, ->-=.1] (L1left.center) .. controls +(0:1) and +(180:1) .. (L0right.center);

    \draw[thick] (L0left.center) .. controls +(0:1) and +(150:1) .. (L0down.center) ;
   
    \draw[thick] (L2left.center) .. controls +(0:1) and +(160:1) .. (Lmid.center);

    \draw[thick] (Lmid.center) .. controls (7.3,0.8) and (7,0.5) .. (6.4,0.5)
                               .. controls (5.8,0.5) and (5.5,0.7) .. (Lmid.center);

    \draw[thick, ->-=.1] (Lmid.center) .. controls +(40:1) and +(180:1) .. (L1right.center);
    
    \fill[black] (7.9,1.73) circle (2pt);    
    \fill[black] (8.7,1.83) circle (2pt);    

    \node[] (q) at (7.9,1.95) {\small $\bar{e}$};
    \node[] (e) at (8.7,2.05) {\small $e$};
    
    \node[] (w1) at (5.5,1) {\small $w_3$}; 
    \node[] (w2) at (7.2,1.4) {\small $w_2$}; 
    \node[] (w3) at (3,1) {\small $w_1$}; 
  \end{tikzpicture}
\end{center}
\begin{multicols}{3}\small
\begin{itemize}[label={}]
\item $\mathfrak{m}_3(w_{1},w_{1},\bar{w}_{1})= - w_{1}$
\item $\mathfrak{m}_3(w_{1},w_{2},\bar{w}_{2})= - w_{1}$
\item $\mathfrak{m}_3(w_{1},\bar{w}_{1},w_{2})= - w_{2}$
\item $\mathfrak{m}_3(w_{1},\bar{w}_{1},w_{3})= - w_{3}$
\item $\mathfrak{m}_3(w_{1},\bar{w}_{1},\bar{w}_{0})=\bar{w}_{0}$
\item $\mathfrak{m}_3(w_{1},\bar{w}_{1},\bar{w}_{1})=\bar{w}_{1}$
\item $\mathfrak{m}_3(w_{1},\bar{w}_{1},\bar{w}_{2})=\bar{w}_{2}$
\item $\mathfrak{m}_3(w_{1},\bar{w}_{1},\bar{w}_{3})=\bar{w}_{3}$
\item $\mathfrak{m}_3(w_{2},w_{1},\bar{w}_{1})= - w_{2}$
\item $\mathfrak{m}_3(w_{2},w_{2},\bar{w}_{2})= - w_{2}$
\item $\mathfrak{m}_3(w_{2},w_{3},\bar{w}_{3})= - w_{2}$
\item $\mathfrak{m}_3(w_{2},\bar{w}_{1},w_{1})=w_{2}$
\item $\mathfrak{m}_3(w_{2},\bar{w}_{2},w_{3})= - w_{3}$
\item $\mathfrak{m}_3(w_{2},\bar{w}_{2},\bar{w}_{0})=\bar{w}_{0}$
\item $\mathfrak{m}_3(w_{2},\bar{w}_{2},\bar{w}_{1})=\bar{w}_{1}$
\item $\mathfrak{m}_3(w_{2},\bar{w}_{2},\bar{w}_{2})=\bar{w}_{2}$
\item $\mathfrak{m}_3(w_{2},\bar{w}_{2},\bar{w}_{3})=\bar{w}_{3}$
\item $\mathfrak{m}_3(w_{2},\bar{w}_{3},w_{3})=w_{2}$
\item $\mathfrak{m}_3(w_{3},w_{1},\bar{w}_{1})= - w_{3}$
\item $\mathfrak{m}_3(w_{3},w_{2},\bar{w}_{2})= - w_{3}$
\item $\mathfrak{m}_3(w_{3},w_{3},\bar{w}_{3})= - w_{3}$
\item $\mathfrak{m}_3(w_{3},\bar{w}_{1},w_{1})=w_{3}$
\item $\mathfrak{m}_3(w_{3},\bar{w}_{3},\bar{w}_{0})=\bar{w}_{0}$
\item $\mathfrak{m}_3(w_{3},\bar{w}_{3},\bar{w}_{2})=\bar{w}_{2}$
\item $\mathfrak{m}_3(w_{3},\bar{w}_{3},\bar{w}_{3})=\bar{w}_{3}$
\item $\mathfrak{m}_3(\bar{w}_{1},w_{1},w_{3})=w_{3}$
\item $\mathfrak{m}_3(\bar{w}_{1},w_{1},\bar{w}_{0})= - \bar{w}_{0}$
\item $\mathfrak{m}_3(\bar{w}_{1},w_{1},\bar{w}_{1})= - \bar{w}_{1}$
\item $\mathfrak{m}_3(\bar{w}_{1},w_{1},\bar{w}_{2})= - \bar{w}_{2}$
\item $\mathfrak{m}_3(\bar{w}_{1},w_{1},\bar{w}_{3})= - \bar{w}_{3}$
\item $\mathfrak{m}_3(\bar{w}_{2},w_{1},\bar{w}_{1})= - \bar{w}_{2}$
\item $\mathfrak{m}_3(\bar{w}_{2},w_{2},\bar{w}_{0})= - \bar{w}_{0}$
\item $\mathfrak{m}_3(\bar{w}_{2},w_{2},\bar{w}_{2})= - \bar{w}_{2}$
\item $\mathfrak{m}_3(\bar{w}_{3},w_{1},\bar{w}_{1})= - \bar{w}_{3}$
\item $\mathfrak{m}_3(\bar{w}_{3},w_{2},\bar{w}_{2})= - \bar{w}_{3}$
\item $\mathfrak{m}_3(\bar{w}_{3},w_{3},\bar{w}_{0})= - \bar{w}_{0}$
\item $\mathfrak{m}_3(\bar{w}_{3},w_{3},\bar{w}_{2})= - \bar{w}_{2}$
\item $\mathfrak{m}_3(\bar{w}_{3},w_{3},\bar{w}_{3})= - \bar{w}_{3}$
\item $\mathfrak{m}_3(\bar{w}_{3},\bar{w}_{1},w_{1})=\bar{w}_{3}$
\end{itemize}
\end{multicols}
Turning on the $s_1$ deformations, we additionally get the following contributions:
\begin{align*}
  &\mathfrak{m}_1(w_1) = -s_1 \bar{w}_2,\quad \mathfrak{m}_1(w_2) = s_1 \bar{w}_1 \\
  &\mathfrak{m}_2(w_1,w_2) = s_1w_0, \quad \mathfrak{m}_2(w_2,\bar{w}_0) =- s_1\bar{w}_1, \quad \mathfrak{m}_2(\bar{w}_0,\bar{w}_1) =s_1\bar{w}_3 \\
  & \mathfrak{m}_3(w_1,w_2,\bar{w}_0) = - s_1 w_0
\end{align*}
Let $\mathfrak{b} = t_1 \bar{w}_1 + t_2 \bar{w}_2 + t_3 \bar{w}_3$. Then, we have
\begin{align*}
\mathfrak{m}_1^{\mathfrak{b}}(w_{1}) &=  - (t_{1} t_{2} + s_1)\bar{w}_{2}  \\
\mathfrak{m}_1^{\mathfrak{b}}(w_2) &= (t_{1} t_{2} + s_1) \bar{w}_{1}  \\
\mathfrak{m}_1^{\mathfrak{b}}(w_3) &=  0.            
\end{align*}
Therefore, we get a flat deformation when $t_1t_2 + s_1= 0$. Writing $(x_1,x_2,x_3) = (w_1,w_2,w_3)$. The deformed algebra over $k[t_1,t_2,t_3]$ is given by :
\[ 
\boxed{k\{x_1,x_2,x_3\}\Big/\left(
\begin{aligned}
& x_1^2+t_1x_1,\; x_2^2+ t_2x_2,\; x_3^2+t_3x_3,\\
& x_1x_2+t_2x_1+t_1x_2 + t_1t_2,\; x_2x_1,\; x_2x_3+t_2x_3 ,\\
& x_3x_2 + t_2 x_3 ,\; x_1x_3 ,\; x_3x_1 
\end{aligned}
\right)}
\]
\[
\begin{array}{c|c}
  \hline
 t_1,t_2,t_3 \neq 0   &   k \times {\scriptstyle(\bullet \to \bullet)}  \\
 t_3 =0 , t_1,t_2\neq0   &   {\scriptsize\begin{tikzcd} \bullet \arrow[r] & \bullet \arrow[distance=1.5em,out=330,in=30,loop,swap] \end{tikzcd}} \\
 t_3\neq 0, t_1=t_2=0 & k \times k[x,y]/(x,y)^2 \\
t_1=t_2=t_3= 0 & k[x,y,z]/(x,y,z)^2 \\ \hline
\end{array}
\]
\subsection*{n=3 : $1^+2^+3^+3^-1^-2^-$}
This is a curve immersed in a three-punctured torus. The boundary components are $m_1$, $s_1$, $s_2$, where $m_1$ is a monogon, $s_1$ is a triangle and $s_2$ is an $8$-gon. The grading is determined by $|s_1|= 2-|w_1|-|w_2|+|w_3|$, $|s_2|=4+|w_1|+|w_2|-2|w_3|$ and $|m_1|= |w_3|$.  We do not give explicit computations in this case as this is related to the next case by a Reidemeister III move. 
\subsection*{n=3 : $1^+2^+3^-1^-3^+2^-$}
This is a curve immersed in a three-punctured torus. The boundary components are $s_1$, $s_2$, $s_3$, where $s_1$ is a bigon, $s_2$ is a triangle and $s_3$ is a 7-gon. The gradings are determined by $|s_1| = |w_2|+|w_3|$, $|s_2|= 2+|w_1|- |w_2|-|w_3|$, and $|s_3|= 4-|w_1|$.
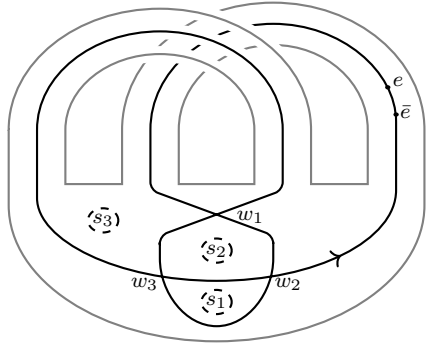
\begin{center}
  \begin{tikzpicture}[xscale=0.5, yscale=0.4]
        \tikzset{->-/.style={decoration={markings, mark=at position #1 with {\arrow{>}}},postaction={decorate}}}
\draw[gray, thick] 

   (-5.5,  0) 
   -- (-5.5, -2.5)
   arc (180:360:5.5 and 4.5)
   -- (5.5, -2.5)
   -- (5.5, 0)

   arc (0:180:4.0 and 4.2) 

   -- (-2.5, -1.8)
   -- (-4, -1.8)
   -- (-4, 0)

   arc (180:0:2.5 and 2.5) 

   -- (1, -1.8)
   -- (-1, -1.8)
   -- (-1, 0)

   arc (180:0:2.5 and 2.8) 

   -- (4, -1.8)
   -- (2.5, -1.8)
   -- (2.5, 0)

   arc (0:180:4.0 and 4.0) 
   -- cycle;

\draw[thick, rounded corners=3pt] 
  (1.75, 0)[->-=0.4]
  arc (0:180:3.25 and 3.25)
  (-4.75,0) -- (-4.75, -1) -- (-4.75, -2)
  arc (180:360:4.75 and 3)  (4.75, -2) -- (4.75, 0)
  arc (0:180: 3.25 and 3.5) (-1.75,0) -- (-1.75,-2) -- (1.5, -3.5) -- (1.5, -4) 
  arc(360:180: 1.5 and 2.5) -- (-1.5, -4) -- (-1.5, -3.5) -- (1.75, -2) -- (1.75,0);

\begin{scope}
  \clip (-7, 1.0) rectangle (7, 4.0);
  
  \draw[white, line width=6pt] 
    (2.5, 0) arc (0:180:4.0 and 4.0);
  \draw[gray, thick, rounded corners=0] 
    (2.5, 0) arc (0:180:4.0 and 4.0);
    
  \draw[white, line width=6pt] 
    (-4, 0) arc (180:0:2.5 and 2.5);
  \draw[gray, thick, rounded corners=0] 
    (-4, 0) arc (180:0:2.5 and 2.5);

   \draw[white, line width=6pt] 
     (1.75, 0) arc (0:180:3.25 and 3.25);
   \draw[thick, rounded corners=0] 
     (1.75, 0) arc (0:180:3.25 and 3.25);
\end{scope}

\draw[thick, dashed] (0,-5.7) circle (0.4cm);
    \node (s1) at (0,-5.7) {{\scriptsize $s_1$}};
    \draw[thick, dashed] (0,-4) circle (0.4cm);
    \node (s2) at (0,-4) {{\scriptsize $s_2$}};
  \draw[thick, dashed] (-3,-3) circle (0.4cm);
    \node (s3) at (-3,-3) {{\scriptsize $s_3$}};

    \node[] (w1) at (0.9,-2.9) {\scriptsize $w_1$}; 
    \node[] (w2) at (1.9,-5.1) {\scriptsize $w_2$}; 
    \node[] (w3) at (-1.9,-5.1) {\scriptsize $w_3$}; 

    \fill[black] (4.53,1.4) circle (2pt);    
    \fill[black] (4.75,0.5) circle (2pt);    

    \node[] (q) at (5,0.6) {\scriptsize $\bar{e}$};
    \node[] (e) at (4.8,1.6) {\scriptsize $e$};
\end{tikzpicture}
\end{center}

\begin{multicols}{3}\small
\begin{itemize}[label={}]
\item $\mathfrak{m}_3(w_{1},w_{1},\bar{w}_{1})= - w_{1}$
\item $\mathfrak{m}_3(w_{1},w_{2},\bar{w}_{2})= - w_{1}$
\item $\mathfrak{m}_3(w_{1},\bar{w}_{1},w_{2})= - w_{2}$
\item $\mathfrak{m}_3(w_{1},\bar{w}_{1},w_{3})= - w_{3}$
\item $\mathfrak{m}_3(w_{1},\bar{w}_{1},\bar{w}_{0})=\bar{w}_{0}$
\item $\mathfrak{m}_3(w_{1},\bar{w}_{1},\bar{w}_{1})=\bar{w}_{1}$
\item $\mathfrak{m}_3(w_{1},\bar{w}_{1},\bar{w}_{2})=\bar{w}_{2}$
\item $\mathfrak{m}_3(w_{1},\bar{w}_{1},\bar{w}_{3})=\bar{w}_{3}$
\item $\mathfrak{m}_3(w_{1},\bar{w}_{3},w_{3})=w_{1}$
\item $\mathfrak{m}_3(w_{2},w_{1},\bar{w}_{1})= - w_{2}$
\item $\mathfrak{m}_3(w_{2},w_{2},\bar{w}_{2})= - w_{2}$
\item $\mathfrak{m}_3(w_{2},w_{3},\bar{w}_{3})= - w_{2}$
\item $\mathfrak{m}_3(w_{2},\bar{w}_{1},w_{1})=w_{2}$
\item $\mathfrak{m}_3(w_{2},\bar{w}_{2},w_{3})= - w_{3}$
\item $\mathfrak{m}_3(w_{2},\bar{w}_{2},\bar{w}_{0})=\bar{w}_{0}$
\item $\mathfrak{m}_3(w_{2},\bar{w}_{2},\bar{w}_{1})=\bar{w}_{1}$
\item $\mathfrak{m}_3(w_{2},\bar{w}_{2},\bar{w}_{2})=\bar{w}_{2}$
\item $\mathfrak{m}_3(w_{2},\bar{w}_{2},\bar{w}_{3})=\bar{w}_{3}$
\item $\mathfrak{m}_3(w_{2},\bar{w}_{3},w_{3})=w_{2}$
\item $\mathfrak{m}_3(w_{3},w_{1},\bar{w}_{1})= - w_{3}$
\item $\mathfrak{m}_3(w_{3},w_{2},\bar{w}_{2})= - w_{3}$
\item $\mathfrak{m}_3(w_{3},\bar{w}_{3},\bar{w}_{0})=\bar{w}_{0}$
\item $\mathfrak{m}_3(w_{3},\bar{w}_{3},\bar{w}_{2})=\bar{w}_{2}$
\item $\mathfrak{m}_3(\bar{w}_{1},w_{1},w_{3})=w_{3}$
\item $\mathfrak{m}_3(\bar{w}_{1},w_{1},\bar{w}_{0})= - \bar{w}_{0}$
\item $\mathfrak{m}_3(\bar{w}_{1},w_{1},\bar{w}_{1})= - \bar{w}_{1}$
\item $\mathfrak{m}_3(\bar{w}_{1},w_{1},\bar{w}_{2})= - \bar{w}_{2}$
\item $\mathfrak{m}_3(\bar{w}_{2},w_{1},\bar{w}_{1})= - \bar{w}_{2}$
\item $\mathfrak{m}_3(\bar{w}_{2},w_{2},\bar{w}_{0})= - \bar{w}_{0}$
\item $\mathfrak{m}_3(\bar{w}_{2},w_{2},\bar{w}_{2})= - \bar{w}_{2}$
\item $\mathfrak{m}_3(\bar{w}_{3},w_{1},\bar{w}_{1})= - \bar{w}_{3}$
\item $\mathfrak{m}_3(\bar{w}_{3},w_{2},\bar{w}_{2})= - \bar{w}_{3}$
\item $\mathfrak{m}_3(\bar{w}_{3},w_{3},w_{3})=w_{3}$
\item $\mathfrak{m}_3(\bar{w}_{3},w_{3},\bar{w}_{0})= - \bar{w}_{0}$
\item $\mathfrak{m}_3(\bar{w}_{3},w_{3},\bar{w}_{1})= - \bar{w}_{1}$
\item $\mathfrak{m}_3(\bar{w}_{3},w_{3},\bar{w}_{2})= - \bar{w}_{2}$
\item $\mathfrak{m}_3(\bar{w}_{3},w_{3},\bar{w}_{3})= - \bar{w}_{3}$
\item $\mathfrak{m}_3(\bar{w}_{3},\bar{w}_{1},w_{1})=\bar{w}_{3}$
\item $\mathfrak{m}_3(\bar{w}_{3},\bar{w}_{3},w_{3})=\bar{w}_{3}$
\end{itemize}
\end{multicols}
We let $\mathfrak{b} = t_1 \bar{w}_1 + t_2 \bar{w}_2 + t_3 \bar{w}_3$.
Turning on the $s_2$-deformation, we additionally get
\[ \mathfrak{m}_2(w_2,w_3) = s_2w_1, \quad \mathfrak{m}_2(w_3,\bar{w}_1) =- s_2\bar{w}_2, \quad \mathfrak{m}_2(\bar{w}_1,w_2) =s_2\bar{w}_3 \]
 Then, $\mathfrak{m}_1^{\mathfrak{b},s_2}$ is given as follows:
\begin{align*}
\mathfrak{m}_1^{\mathfrak{b}}(w_{1}) &=  - t_{1} t_{2} \bar{w}_{2} + t_{1} t_{3} \bar{w}_{3} \\
\mathfrak{m}_1^{\mathfrak{b}}(w_2) &= t_{1} t_{2} \bar{w}_{1}  + s_2 t_1 \bar{w}_3 \\
\mathfrak{m}_1^{\mathfrak{b}}(w_3) &=  -t_{1} t_{3} \bar{w}_{1} - s_2 t_1 \bar{w}_2            
\end{align*}
The deformed algebra over $k[t_1,t_2,t_3,s_2] / (t_1t_2, t_1t_3,t_1s_2)$ is
\[\boxed{
k\{x_1,x_2,x_3\}\Big/\left(
\begin{aligned}
& x_1^2+t_1x_1,\; x_2^2+ t_2x_2,\; x_3^2-t_3x_3,\\
& x_1x_2+t_2x_1+t_1x_2,\; x_2x_1,\; x_2x_3+t_2x_3-s_2x_1,\\
& x_3x_2 + t_2 x_3,\; x_1x_3-t_3x_1 ,\; x_3x_1 + t_1 x_3 
\end{aligned}
\right)}
\]
We give the isomorphism types of the algebras occurring in this family. 
\[
\begin{array}{c|c}
  \hline
 t_1= 0, t_2, t_3 \neq 0 &    k \times {\scriptstyle(\bullet \to \bullet)}  \\
 t_1=t_2=0, t_3 \neq 0 & {\scriptsize\begin{tikzcd} \bullet \arrow[r] & \bullet \arrow[distance=1.5em,out=330,in=30,loop,swap] \end{tikzcd}} \\[6pt]
 t_1=t_3=0, t_2\neq 0 & {\scriptsize \begin{tikzcd}
    \bullet \arrow[distance=1.5em, out=150,in=210,loop,swap] \arrow[r] & \bullet
    \end{tikzcd}}    \\[6pt]
  t_1\neq 0, t_2=t_3=s_2=0 &{\scriptsize\begin{tikzcd}[column sep=large] \bullet \rightrightarrows \bullet \end{tikzcd} }   \\[6pt]
  t_1=t_2=t_3=0, s_2 \neq 0 & k\{x,y\} / (x^2,xy,y^2)\\[6pt]
 t_1=t_2=t_3=s_2=0 & k[x,y,z]/(x,y,z)^2 \\[6pt]  
  \hline
\end{array}
\]
\subsection*{n=3 : $1^+2^+3^+1^-3^-2^-$}
This curve is immersed in a genus 2 curve with one puncture. The boundary component $s$ is a 12-gon.

\begin{multicols}{3}\small
\begin{itemize}[label={}]
\item $\mathfrak{m}_3(w_{1},w_{1},\bar{w}_{1})= - w_{1}$
\item $\mathfrak{m}_3(w_{1},w_{2},\bar{w}_{2})= - w_{1}$
\item $\mathfrak{m}_3(w_{1},w_{3},\bar{w}_{3})= - w_{1}$
\item $\mathfrak{m}_3(w_{1},\bar{w}_{1},w_{2})= - w_{2}$
\item $\mathfrak{m}_3(w_{1},\bar{w}_{1},w_{3})= - w_{3}$
\item $\mathfrak{m}_3(w_{1},\bar{w}_{1},\bar{w}_{0})=\bar{w}_{0}$
\item $\mathfrak{m}_3(w_{1},\bar{w}_{1},\bar{w}_{1})=\bar{w}_{1}$
\item $\mathfrak{m}_3(w_{1},\bar{w}_{1},\bar{w}_{2})=\bar{w}_{2}$
\item $\mathfrak{m}_3(w_{1},\bar{w}_{1},\bar{w}_{3})=\bar{w}_{3}$
\item $\mathfrak{m}_3(w_{2},w_{1},\bar{w}_{1})= - w_{2}$
\item $\mathfrak{m}_3(w_{2},w_{2},\bar{w}_{2})= - w_{2}$
\item $\mathfrak{m}_3(w_{2},w_{3},\bar{w}_{3})= - w_{2}$
\item $\mathfrak{m}_3(w_{2},\bar{w}_{1},w_{1})=w_{2}$
\item $\mathfrak{m}_3(w_{2},\bar{w}_{2},w_{3})= - w_{3}$
\item $\mathfrak{m}_3(w_{2},\bar{w}_{2},\bar{w}_{0})=\bar{w}_{0}$
\item $\mathfrak{m}_3(w_{2},\bar{w}_{2},\bar{w}_{1})=\bar{w}_{1}$
\item $\mathfrak{m}_3(w_{2},\bar{w}_{2},\bar{w}_{2})=\bar{w}_{2}$
\item $\mathfrak{m}_3(w_{2},\bar{w}_{2},\bar{w}_{3})=\bar{w}_{3}$
\item $\mathfrak{m}_3(w_{2},\bar{w}_{3},w_{3})=w_{2}$
\item $\mathfrak{m}_3(w_{3},w_{1},\bar{w}_{1})= - w_{3}$
\item $\mathfrak{m}_3(w_{3},w_{2},\bar{w}_{2})= - w_{3}$
\item $\mathfrak{m}_3(w_{3},w_{3},\bar{w}_{3})= - w_{3}$
\item $\mathfrak{m}_3(w_{3},\bar{w}_{1},w_{1})=w_{3}$
\item $\mathfrak{m}_3(w_{3},\bar{w}_{3},\bar{w}_{0})=\bar{w}_{0}$
\item $\mathfrak{m}_3(w_{3},\bar{w}_{3},\bar{w}_{1})=\bar{w}_{1}$
\item $\mathfrak{m}_3(w_{3},\bar{w}_{3},\bar{w}_{2})=\bar{w}_{2}$
\item $\mathfrak{m}_3(w_{3},\bar{w}_{3},\bar{w}_{3})=\bar{w}_{3}$
\item $\mathfrak{m}_3(\bar{w}_{1},w_{1},\bar{w}_{0})= - \bar{w}_{0}$
\item $\mathfrak{m}_3(\bar{w}_{1},w_{1},\bar{w}_{1})= - \bar{w}_{1}$
\item $\mathfrak{m}_3(\bar{w}_{1},w_{1},\bar{w}_{2})= - \bar{w}_{2}$
\item $\mathfrak{m}_3(\bar{w}_{1},w_{1},\bar{w}_{3})= - \bar{w}_{3}$
\item $\mathfrak{m}_3(\bar{w}_{2},w_{1},\bar{w}_{1})= - \bar{w}_{2}$
\item $\mathfrak{m}_3(\bar{w}_{2},w_{2},\bar{w}_{0})= - \bar{w}_{0}$
\item $\mathfrak{m}_3(\bar{w}_{2},w_{2},\bar{w}_{2})= - \bar{w}_{2}$
\item $\mathfrak{m}_3(\bar{w}_{3},w_{1},\bar{w}_{1})= - \bar{w}_{3}$
\item $\mathfrak{m}_3(\bar{w}_{3},w_{2},\bar{w}_{2})= - \bar{w}_{3}$
\item $\mathfrak{m}_3(\bar{w}_{3},w_{3},\bar{w}_{0})= - \bar{w}_{0}$
\item $\mathfrak{m}_3(\bar{w}_{3},w_{3},\bar{w}_{2})= - \bar{w}_{2}$
\item $\mathfrak{m}_3(\bar{w}_{3},w_{3},\bar{w}_{3})= - \bar{w}_{3}$
\end{itemize}
\end{multicols}
We let $\mathfrak{b} = t_1 \bar{w}_1 + t_2 \bar{w}_2 + t_3 \bar{w}_3$. Then, $\mathfrak{m}_1^{\mathfrak{b}}$ is given as follows:
\begin{align*}
\mathfrak{m}_1^{\mathfrak{b}}(w_{1}) &=  - t_{1} t_{2} \bar{w}_{2} - t_{1} t_{3} \bar{w}_{3} \\
\mathfrak{m}_1^{\mathfrak{b}}(w_2) &= t_{1} t_{2} \bar{w}_{1}  \\
\mathfrak{m}_1^{\mathfrak{b}}(w_3) &=  t_{1} t_{3} \bar{w}_{1}             
\end{align*}
The deformed algebra over $k[t_1,t_2,t_3] / (t_1t_2, t_1t_3)$ is
\[ 
\boxed{k\{x_1,x_2,x_3\}\Big/\left(
\begin{aligned}
& x_1^2+t_1x_1,\; x_2^2+ t_2x_2,\; x_3^2+t_3x_3,\\
& x_1x_2+t_2x_1+t_1x_2,\; x_2x_1,\; x_2x_3+t_2x_3,\\
& x_3x_2 + t_2 x_3,\; x_1x_3+t_3x_1+t_1x_3 ,\; x_3x_1 
\end{aligned}
\right)}
\]
We give the isomorphism types of the algebras occurring in this family. 
\[
\begin{array}{c|c}
  \hline
  t_1 = 0, t_2,t_3 \neq 0  &    k \times {\scriptstyle(\bullet \to \bullet)}  \\
  t_3 \neq 0, t_1=t_2 = 0 
  & {\scriptsize\begin{tikzcd} \bullet \arrow[r] & \bullet \arrow[distance=1.5em,out=330,in=30,loop,swap] \end{tikzcd}} \\[6pt]
  t_2\neq 0, t_1=t_3=0
  & {\scriptsize \begin{tikzcd}
    \bullet \arrow[distance=1.5em, out=150,in=210,loop,swap] \arrow[r] & \bullet
    \end{tikzcd}}    \\[6pt]
  t_1 \neq 0, t_2=t_3 = 0
  &{\scriptsize\begin{tikzcd}[column sep=large] \bullet \rightrightarrows \bullet \end{tikzcd} }   \\[6pt]
  (t_1,t_2,t_3) = (0,0,0) 
  &k[x,y,z]/(x,y,z)^2 \\[6pt]  
  \hline
\end{array}
\]

\section{Some variations and further comments}

One can easily extend the combinatorial description given here to immersions whose domains are disjoint unions of circles and intervals. In the latter case, one works with the (partially) wrapped Fukaya category.

In order to realize a rank $r$ algebra, one can try to take an immersion with $r'$ self-intersection points with $r < r'$, but choose the grading so that exactly $r-1$ of the self-intersections give generators in degrees $\{0,1\}$.

It is also possible to work with $\mathbb{Z}/2\mathbb{Z}$-graded Fukaya categories. Then, many more compactifications (hence, deformations) become available.   

We used a perturbation scheme in order to determine finite descriptions of the $A_\infty$ algebras; however, our perturbation breaks the cyclic symmetry. For various reasons, it is desirable to find cyclic $A_\infty$ models (cf. \cite{aybe}). 

In this paper, we worked mainly over an algebraically closed field $k$ of characteristic zero. In fact, all the endomorphism algebras that we computed are actually defined over $\mathbb{Z}$. In particular, one can ask which low rank algebras are realized over non-closed fields. For instance, we have seen that all the quaternion algebras over $\mathbb{Q}$ are realized.

We remark that the irreducible components of the base of the families of algebras that we constructed are all affine spaces (in particular, they are all smooth). 

An obvious next goal is to study immersions with 4 and 5 nodes. The corresponding classification of isomorphism types of algebras forming the points of $\Alg_5$ and $\Alg_6$ is known. It has taken us substantial effort to work through all the immersions with 3 or fewer self-intersections, but part of the difficulty was to figure out a combinatorial model for the $A_\infty$ structure, and now that's settled. Thus, glossing over the problem of enumerating all the possible (finitely many!) polygons contributing to the higher products in a given immersion, we have a finite and combinatorial problem at hand.

\vspace{\fill} 
\noindent
\begin{tabular}{@{}l}
  \textit{Department of Mathematics,} \\
  \textit{Imperial College,} \\
  \textit{London, UK}
\end{tabular}

\end{document}